\documentclass{amsart}

\usepackage{amsmath,amssymb, amscd, latexsym, graphicx, psfrag}

\newcommand{\ep}{\epsilon}
\newcommand{\si}{\sigma}
\newcommand{\Si}{\Sigma}
\newcommand{\Ga}{\Gamma}

\newcommand{\bC}{\mathbb{C}}

\newcommand{\bP}{\mathbb{P}}
\newcommand{\bQ}{\mathbb{Q}}
\newcommand{\bR}{\mathbb{R}}
\newcommand{\bT}{\mathbb{T}}
\newcommand{\bZ}{\mathbb{Z}}

\newcommand{\cB}{\mathcal{B}}
\newcommand{\cC}{\mathcal{C}}
\newcommand{\cD}{\mathcal{D}}
\newcommand{\cE}{\mathcal{E}}
\newcommand{\cH}{\mathcal{H}}
\newcommand{\cI}{\mathcal{I}}
\newcommand{\cK}{\mathcal{K}}
\newcommand{\cL}{\mathcal{L}}
\newcommand{\cM}{\mathcal{M}}
\newcommand{\cO}{\mathcal{O}}
\newcommand{\cP}{\mathcal{P}}
\newcommand{\cR}{\mathcal{R}}
\newcommand{\cQ}{\mathcal{Q}}
\newcommand{\cT}{\mathcal{T}}
\newcommand{\cU}{\mathcal{U}}

\newcommand{\cX}{\mathcal{X}}
\newcommand{\cY}{\mathcal{Y}}

\newcommand{\fp}{\mathfrak{p}}
\newcommand{\fl}{\mathfrak{l}}
\newcommand{\fn}{\mathfrak{n}}
\newcommand{\fg}{\mathfrak{g}}

\newcommand{\Hom}{\mathrm{Hom}}
\newcommand{\age}{\mathrm{age}}

\newcommand{{\inv} }{\mathrm{inv}}
\newcommand{\ev}{\mathrm{ev}}
\newcommand{\Aut}{\mathrm{Aut}}
\newcommand{\Res}{\mathrm{Res}}
\newcommand{\rank}{\mathrm{rank}}
\newcommand{\val}{ {\mathrm{val}} }
\newcommand{\vir}{{\mathrm{vir}}}
\newcommand{\CR}{  {\mathrm{CR}}  }

\newcommand{\one}{\mathbf{1}}

\newcommand{\be}{\mathbf{e}}
\newcommand{\bu}{\mathbf{u}}

\newcommand{\bGa}{\mathbf{\Ga}}
\newcommand{\btau}{\boldsymbol\tau }
\newcommand{\zero}{\mathbf{0}}
\newcommand{\bmu}{\boldsymbol\mu}
\newcommand{\boSi}{\boldsymbol\Si}

\newcommand{\su}{\mathsf{u}}
\newcommand{\sv}{\mathsf{v}}
\newcommand{\sw}{\mathsf{w}}

\newcommand{\tit}{ {\widetilde{t}} }
\newcommand{\tA}{ {\widetilde{A}} }
\newcommand{\tW}{{\widetilde{W}}}

\newcommand{\tp}{ {\widetilde{p}} }
\newcommand{\tI}{{\widetilde{I}}}
\newcommand{\tT}{\widetilde{T}}

\newcommand{\talpha}{ {\widetilde{\alpha}} }
\newcommand{\tGamma}{{\widetilde{\Gamma}}}
\newcommand{\txi}{ {\widetilde{\xi}} }
\newcommand{\tSi}{ {\widetilde{\Si}} }

\newcommand{\hF}{\hat{F}}
\newcommand{\hu}{\hat{u}}
\newcommand{\hbu}{\hat{\bu}}
\newcommand{\hxi}{\hat{\xi}}

\newcommand{\vh}{\vec{h}}
\newcommand{\vGa}{\vec{\Ga}}
\newcommand{\vmu}{\vec{\mu}}

\newcommand{\tGammar}{{\tGamma^{\mathrm{red}}_{\alpha,q}}}
\newcommand{\Bsi}{\mathrm{Box}(\si)}
\newcommand{\Mbar}{\overline{\cM}}

\newcommand{\BG}{\mathcal{B} G}
\newcommand{\IBG}{\mathcal{IB} G}
\newcommand{\IX}{\mathcal{IX}}

\newtheorem{dummy}{dummy}[section]
\newtheorem{lemma}[dummy]{Lemma}
\newtheorem{theorem}[dummy]{Theorem}

\newtheorem{corollary}[dummy]{Corollary}
\newtheorem{proposition}[dummy]{Proposition}

\newtheorem{definition}[dummy]{Definition}
\newtheorem{remark}[dummy]{Remark}
\newtheorem{example}[dummy]{Example}

\allowdisplaybreaks[1]

\begin{document}

\title[\tiny All Genus Open-Closed Mirror Symmetry for Affine Toric CY 3-Orbifolds]{All Genus Open-Closed Mirror Symmetry for Affine Toric Calabi-Yau 3-Orbifolds}

\author{Bohan Fang}
\email{bohanfang@gmail.com}
\address{Beijing International Center for Mathematical Research, Peking University, 5 Yiheyuan Road, Beijing 100871, China}

\author{Chiu-Chu Melissa Liu}
\email{ccliu@math.columbia.edu}
\address{Department of Mathematics, Columbia University, 2990 Broadway, New York, NY 10027}

\author{Zhengyu Zong}
\email{zyzong@math.tsinghua.edu.cn}
\address{Yau Mathematical Sciences Center, Tsinghua University, Jin Chun Yuan West Building,
Tsinghua University, Haidian District, Beijing 100084, China}

\begin{abstract}
The Remodeling Conjecture proposed by Bouchard-Klemm-Mari\~{n}o-Pasquetti \cite{BKMP09, BKMP10} relates all genus open and closed Gromov-Witten invariants of a semi-projective toric
Calabi-Yau 3-manifolds/3-orbifolds $\cX$ to the Eynard-Orantin invariants
of the mirror curve of $\cX$. In this paper, we present a proof of the Remodeling Conjecture
for open-closed orbifold Gromov-Witten invariants of an arbitrary affine toric Calabi-Yau 3-orbifold relative to a framed Aganagic-Vafa Lagrangian brane.
This can be viewed as an all genus open-closed mirror symmetry for affine toric Calabi-Yau 3-orbifolds.
\end{abstract}

\keywords{Open Gromov-Witten invariants, mirror symmetry, topological recursion}

\subjclass[2010]{Primary 14N35, Secondary 14J33}
\maketitle


\section{Introduction}

\subsection{Background and motivation}
Mirror symmetry relates the A-model topological string theory
on a Calabi-Yau 3-fold $\cX$ to the B-model topological string theory
on another Calabi-Yau 3-fold $\check{\cX}$, the mirror of $\cX$.
The genus $g$ free energy of the topological A-model on $\cX$ is
mathematically defined as a generating function $F_g^{\cX}$ of genus $g$ Gromov-Witten invariants of $\cX$,
which is a function on a (formal) neighborhood of the large radius limit in the (complexified) K\"{a}hler moduli of $\cX$.
The genus $g$ free energy of the topological B-model on $\check{\cX}$
is a section of $\mathbb{L}^{2-2g}$, where $\mathbb{L}$ is a line bundle over the complex moduli of $\check{\cX}$,
so locally it is a function $\check{F}_g^{\check{X}}$ on the complex moduli of $\check{\cX}$.
A mathematical consequence of mirror symmetry is
$\check{F}_g^{\check{\cX}}= F_g^{\cX} + \delta_{g,0}a_0 + \delta_{g,1} a_1$ under the mirror map,
where $a_0$  (resp. $a_1$) is a cubic  (resp. linear) function in K\"{a}hler parameters.
The mirror map and $\check{F}_0^{\check{\cX}}$
are determined by period integrals of a holomorphic 3-form on $\check{\cX}$.
Period integrals of Calabi-Yau complete intersections in toric manifolds
can be expressed in terms of explicit hypergeometric functions.

Let $Q$ be the quintic 3-fold, which is a Calabi-Yau hypersurface in $\bP^4$, and let $\check{Q}$ be the mirror of $Q$.
Candelas, de la Ossa, Green, and Parkes computed $\check{F}_0^{\check{Q}}$
and the mirror map explicitly, and obtained a conjectural formula of the number of rational curves of arbitrary degree in $Q$;
this formula was proved independently by Givental \cite{Gi96} and Lian-Liu-Yau \cite{LLY1},
who later extended their results to Calabi-Yau complete intersections in projective toric manifolds
\cite{Gi98, LLY2, LLY3}. Their proofs rely on a good understanding of genus-zero Gromov-Witten theory of $Q$.
The mirror formula for $F_1^Q$ was conjectured by Bershadsky-Cecotti-Ooguri-Vafa \cite{BCOV} and first proved by Zinger \cite{Zi09}.
Combining the techniques of BCOV, results of Yamaguchi-Yau \cite{YY04}, and boundary conditions,
Huang-Klemm-Quacken \cite{HKQ} proposed a mirror conjecture on $F_g^Q$ up to $g=51$. 
Maulik-Pandharipande provided a mathematical determination of Gromov-Witten invariants of $Q$ in all genera and
degrees \cite{MaPa}.

In contrast, higher genus Gromov-Witten invariants of {\em toric} Calabi-Yau 3-folds (which must be non-compact)
are much better understood. In general, Gromov-Witten invariants are defined for projective manifolds (or  more generally
compact almost K\"{a}hler manifolds), but in the toric case one may use localization to define Gromov-Witten invariants of
certain non-compact toric manifolds. By virtual localization \cite{GrPa}, all genus Gromov-Witten invariants of
toric manifolds can be reduced to Hodge integrals, which can be evaluated by effective algorithms.
When the toric manifold $\cX$ is a Calabi-Yau 3-fold, the Topological Vertex \cite{AKMV, LLLZ,MOOP} provides a much more efficient
algorithm of computing Gromov-Witten invariants of $\cX$, as well as open Gromov-Witten invariants of $\cX$
relative to an Aganagic-Vafa Lagrangian brane $\cL$ (defined in \cite{KL, DF, Liu02, LLLZ} in several ways), in all genera and degrees. The topological B-model on the mirror
$\check{\cX}$ of a smooth semi-projective toric Calabi-Yau 3-fold $\cX$ can be reduced to a theory on the mirror curve of $\cX$.
Under mirror symmetry, $F_0^{\cX}$ corresponds to integrals of 1-forms on the mirror curve along loops,
whereas the generating function $F_{0,1}^{\cX,\cL}$ of genus-zero open Gromov-Witten invariants
(which count holomorphic disks in $\cX$ bounded by $\cL$) corresponds to integrals of 1-forms on the mirror curve along paths \cite{AV, AKV}.
Based on the work of Eynard-Orantin \cite{EO07} and Mari\~{n}o \cite{Ma},
Bouchard-Klemm-Mari\~{n}o-Pasquetti \cite{BKMP09} proposed a new formalism of the topological B-model on $\cX$ in terms
of the Eynard-Orantin invariants $\omega_{g,n}$ of the mirror curve, and conjectured a precise correspondence, known as the Remodeling Conjecture,
between $\omega_{g,n}$ (where $n>0$) and the generating function $F_{g,n}^{\cX,\cL}$ of
open Gromov-Witten invariants counting holomorphic maps from bordered Riemann surfaces with
$g$ handles and $n$ holes to $\cX$ with boundaries in $\cL$. This can be viewed as a version
of all genus open mirror symmetry; the closed sector of the Remodeling Conjecture relates $\omega_{g,0}$ to $F_g^{\cX}$.
The open string part of the Remodeling Conjecture for $\bC^3$ was proved independently
by  L. Chen \cite{Ch09} and J. Zhou \cite{Zh09}. The free energy part of the Remodeling
Conjecture for $\bC^3$ \cite{BS12} was proved independently by Bouchard-Catuneanu-Marchal-Su{\l}kowski
\cite{BCMS} and S. Zhu \cite{Zhu}. B. Eynard and N. Orantin provided a proof of the
Remodeling Conjecture for general smooth semi-projective toric Calabi-Yau 3-folds in \cite{EO15}.

Bouchard-Klemm-Mari\~{n}o-Pasquetti have extended the Remodeling Conjecture to toric Calabi-Yau 3-orbifolds \cite{BKMP10}.
Topological string theory on orbifolds was constructed decades ago by physicists \cite{DHVW1, DHVW2}, and many works followed in both mathematics and physics (e.g. \cite{HV, DFMS, BR, Kn, Roan, CV}).  Zaslow discussed orbifold quantum cohomology \cite{Zas93} along
with many examples, in both abelian and non-abelian quotients. Later, the mathematical definition of orbifold Gromov-Witten theory and quantum
cohomology was laid by Chen-Ruan \cite{CR02} in the symplectic setting and Abramovich-Graber-Vistoli \cite{AGV02, AGV08} in the algebraic setting.
Orbifold Gromov-Witten invariants of a toric Calabi-Yau 3-orbifold $\cX$, and
open orbifold Gromov-Witten invariants of $\cX$ relative to an Aganagic-Vafa
Lagrangian brane in $\cX$, can be expressed in terms of the orbifold Gromov-Witten
vertex, which is a generating function of abelian Hurwitz Hodge integrals \cite{Ro14}.
The algorithm of the Topological Vertex is equivalent to the Gromov-Witten(GW)/Donaldson-Thomas(DT) correspondence for smooth toric Calabi-Yau 3-folds \cite{MNOP1}. The GW/DT correspondence
has been conjectured for Calabi-Yau 3-orbifolds satisfying the Hard Lefshetz condition \cite{BCY}
and proved for toric Calabi-Yau 3-orbifolds with transverse $A_n$-singularities
\cite{Zo15, RZ13, RZ14, R14}. This provides an efficient algorithm of computing closed and open Gromov-Witten
invariants of toric Calabi-Yau 3-orbifolds with transverse $A_n$-singularities, in all genera and degrees.
The precise statement of GW/DT correspondence  for toric Calabi-Yau 3-orbifolds which
do not satisfy the Hard Lefshetz condition (e.g. $[\bC^3/\bmu_3]$, where $\bmu_3$ acts diagonally)
is not known even conjecturally. The Remodeling Conjecture provides
a recursive algorithm to compute closed and open Gromov-Witten invariants of all semi-projective toric Calabi-Yau 3-orbifolds in all genera. Explicit predictions of open orbifold Gromov-Witten
invariants of $[\bC^3/\bmu_3]$ are given in \cite{BKMP10}.

\subsection{Statement of the main result and outline of the proof}
In this paper, we study the Remodeling Conjecture
for all affine toric Calabi-Yau 3-orbifolds $[\bC^3/G]$, where
$G$ can be any finite subgroup of the maximal torus of $SL(3,\bC)$.
We consider a general framed Aganagic-Vafa brane $(\cL,f)$,
where $\cL \cong [(S^1\times \bC)/G]$ and $f\in \bZ$.
We define open-closed orbifold Gromov-Witten invariants of
$\cX=[\bC^3/G]$ relative to $(\cL,f)$ as certain
equivariant relative orbifold invariants of a toric Calabi-Yau
3-orbifold $\cY_f$ relative to a divisor $\cD_f\cong [\bC^2/\bmu_m]$,
where $\bmu_m\cong \bZ_m$ is the stabilizer of the $G$-action on
$S^1\times \{0\}\subset S^1\times \bC$, and the coarse moduli space
$\bC^2/\bmu_m$ of $\cD_f$ is the $A_{m-1}$ surface singularity.
We define generating functions of open-closed Gromov-Witten invariants:
$$
F_{g,n}^{\cX,(\cL,f)}(\btau;X_1,\ldots, X_n)
$$
which are $H^*_{\CR}(\cB\bmu_m;\bC)^{\otimes n}$-valued formal
power series in A-model closed string coordinates $\btau=(\tau_1,\ldots, \tau_p)$
and A-model open string coordinates $X_1,\ldots, X_n$;
here $H^*_{\CR}(\cB\bmu_m;\bC)\cong \bC^m$ is the Chen-Ruan orbifold cohomology
of the classifying space $\cB\bmu_m$ of $\bmu_m$.
We use the Eynard-Orantin invariants $\omega_{g,n}$ of the framed mirror curve
to define B-model potentials
$$
\check{F}_{g,n}(\btau;X_1,\ldots,X_n)
$$
which are $H^*_{\CR}(\cB\bmu_m;\bC)^{\otimes n}$-valued
functions in B-model closed string flat coordinates
$\btau=(\tau_1,\ldots, \tau_p)$ and B-model open string
coordinates $X_1,\ldots,X_n$, analytic in an open
neighborhood of the origin in $\bC^p\times\bC^n$. The mirror map
relates B-model flat coordinates $(\tau_1,\ldots,\tau_p)$
to the complex parameters $(q_1,\ldots, q_p)$ of the framed mirror curve.
Our main result is:

\medskip

\paragraph{\bf Theorem \ref{thm:main}}
{\it For any $g\in \bZ_{\geq 0}$ and $n\in \bZ_{>0}$,}
\begin{equation}\label{eqn:main-eqn}
\check{F}_{g,n}(\btau;X_1,\ldots, X_n) =(-1)^{g-1+n}F_{g,n}^{\cX, (\cL,f)}(\btau;X_1,\ldots, X_n).
\end{equation}
This is indeed more general than the original conjecture in
\cite{BKMP10}, which covers the $m=1$ case, i.e., when $\cL$ is
on an effective leg.

We now give an outline of our proof of the above theorem. For simplicity,
we consider the stable case $2g-2+n>0$ in this outline. (The unstable cases
$(g,n)=(0,1)$, $(0,2)$ will be treated separately.)  The proof consists of three steps:
\begin{enumerate}
\item[1.]({\em A-model graph sum}) In \cite{FLZ1}, we used Tseng's orbifold quantum Riemann-Roch theorem \cite{Ts10} to write
down a graph sum formula for the total descendant equivariant Gromov-Witten potential of $\cX$. We use this formula
and localization to derive a graph sum formula for the A-model potential $F^{\cX,(\cL,f)}_{g,n}$:
$$
F_{g,n}^{\cX,(\cL,f)}  =\sum_{\vGa\in \bGa_{g,n}(\cX)}\frac{w_A(\vGa)}{|\Aut(\vGa)|}
$$
where $\bGa_{g,n}(\cX)$ is certain set of decorated graphs, $\Aut(\vGa)$ is the automorphism group of the decorated graph $\vGa$,
and $w_A(\vGa)$ is the A-model weight of the decorated graph $\vGa$ defined by \eqref{eqn:wA}.

\item[2.]({\em B-model graph sum})
The Eynard-Orantin invariants $\omega_{g,n}$ can be expressed as a sum over labeled graphs \cite{KO, E11, E14, DOSS}.  We use special geometry to obtain the Taylor series expansion of
the graph sum formula in \cite[Theorem 3.7]{DOSS} in B-model closed string flat coordinates $\btau=(\tau_1,\ldots,\tau_p)$
at $\btau=0$, and derive a graph sum formula for the B-model potential $\check{F}_{g,n}$:
$$
\check{F}_{g,n} = \sum_{\vGa\in \bGa_{g,n}(\cX)}\frac{w_B(\vGa)}{|\Aut(\vGa)|}
$$
where $w_B(\vGa)$ is the B-model weight of the decorated graph $\vGa$ defined by \eqref{eqn:wB}.

\item[3.]({\em Comparison of weights}) For each decorated graph $\vGa \in \bGa_{g,n}(\cX)$, we prove the following identity
relating A-model and B-model weights:
$$
w_B(\vGa) = (-1)^{g-1+n} w_A(\vGa).
$$
\end{enumerate}

\subsection{Remarks on the BKMP Remodeling Conjecture in the general case}
In \cite{FLZ3}, the authors provide a proof of the BKMP Remodeling
Conjecture for all semi-projective toric Calabi-Yau 3-orbifolds.
The proof in \cite{FLZ3} relies on
(i) the quantization formula for the total descendant potential of equivariant GW theory of GKM orbifolds \cite{Zo},
(ii) the B-model graph sum formula in \cite[Theorem 3.7]{DOSS},
(iii) the genus-zero mirror theorem for toric DM stacks \cite{CCIT, CCK}, and
(iv) the genus-zero open mirror theorem for semi-projective toric Calabi-Yau 3-orbifolds \cite{FLT};
(i) is used to derive the A-model graph sum formula, whereas (iii) and (iv)
are used to match the A-model and B-model graph sums.  In the affine case,
the proof in this paper is more direct than the specialization of the proof in \cite{FLZ3} to the affine case:
the proof in this paper relies on Tseng's orbifold quantum Riemann-Roch theorem \cite{Ts10} and (ii), but not on (i), (iii), (iv).
The proof in \cite{FLZ3} also relies on the computation of oscillating integrals on a mirror curve in this
paper: we obtain the desired result by integrating in the Landau-Ginzburg model and dimensional reduction (Theorem \ref{thm:f-zero}).


\subsection{Overview of the paper}
In Section \ref{sec:ABgeometry}, we describe affine toric Calabi-Yau 3-orbifolds
and their mirror curves.  In Section \ref{sec:vertex},  for each affine toric Calabi-Yau 3-orbifold $\cX$ and a
framed Aganagic-Vafa A-brane $(\cL,f)$, we construct a relative Calabi-Yau 3-orbifold $(\cY_f, \cD_f)$, where $\cX=\cY_f\setminus\cD_f$.
In Section \ref{sec:Chen-Ruan}, we describe the Chen-Ruan orbifold cohomology
of the classifying space $\BG$ of the finite abelian group $G$ and the $\tT$-equivariant  Chen-Ruan orbifold cohomology of $\cX$ and $\cY_f$.
In Section \ref{sec:Amodel}, we give the precise definition of the A-model partition functions $F_{g,n}^{\cX,(\cL,f)}$ as a generating function
of equivariant relative Gromov-Witten invariants of $(\cY_f,\cD_f)$, and derive the A-model graph sum formula (Theorem \ref{thm:Asum}).
In Section \ref{sec:mirror-curve}, we study the geometry and topology of the mirror curve.
In particular, we clarify the choice of A-cycles and B-cycles, and the definition of the B-model
flat coordinates. In Section \ref{sec:Bmodel}, we give the precise definition of B-model partition functions $\check{F}_{g,n}$,
derive the B-model graph sum formula (Theorem \ref{thm:Bsum}), and complete the proof of the main result (Theorem \ref{thm:main}).

\subsection*{Acknowledgments} We thank Vincent Bouchard,
Igor Dolgachev, Bertrand Eynard, Robert Friedman, Nicolas Orantin, Dustin Ross and Jie Zhou for helpful comments
and conversations.
The research of the first author is partially supported by NSF DMS-1206667.
The research of the second and third authors is partially supported by NSF DMS-1159416.

\section{Affine toric Calabi-Yau 3-orbifolds and their Mirrors}\label{sec:ABgeometry}

\subsection{The A-model geometry}\label{sec:A-geometry}
Let $\bT=(\bC^*)^3$, $N=\Hom(\bC^*,\bT)\cong \bZ^3$, and let
$M=\Hom(\bT,\bC^*)=\Hom(N,\bZ)$.
Let $\sigma\subset N_\bR := N\otimes_\bZ \bR\cong \bR^3$
be a simplicial cone spanned by $b_1, b_2, b_3\in N$,
such that the simplicial affine toric variety
$X_\si:=\mathrm{Spec}\, \bC[\sigma^\vee\cap M]$ has trivial
canonical divisor. Then there exists $u\in M$ such that
$\langle u,b_1\rangle = \langle u,b_2\rangle = \langle u,b_3\rangle =1$.
We may choose a $\bZ$-basis $\su_1,\su_2,\su_3$ of $M$ such that
$\su_3=u$. Let $e_1,e_2,e_3$ be the dual $\bZ$-basis of $N$. Then
$$
b_1 = re_1-se_2+e_3,\quad
b_2 = me_2+e_3,\quad
b_3 = e_3,
$$
where $r$ and $m$ are positive integers and
$s\in \{0,1,\ldots, r-1\}$.

We have a short exact sequence of abelian groups
$$
1\to G\to \tT=(\bC^*)^3\stackrel{\phi}{\to} \bT=(\bC^*)^3\to 1
$$
where
$$
\phi(\tit_1, \tit_2, \tit_3)=(\tit_1^r, \tit^{-s}_1 \tit_2^m,
\tit_1\tit_2\tit_3).
$$

For $i=1,2,3$, let $\chi_i: G\to \bC^*$ be the projection to the $i$-th factor.
Then $\chi_1,\chi_2,\chi_3\in G^*=\Hom(G,\bC^*)$, and $\chi_1\chi_2\chi_3=1$.
Given a positive integer $r$, let $\bmu_r:=\{ z\in \bC^*: z^r=1\}\cong \bZ_r$ be the group of $r$-th roots of unity.
The image of $\chi_1$ is $\bmu_r$ and the kernel of $\chi_1$ is
isomorphic to $\bmu_m$. We have a short exact sequence of
finite abelian groups
$$
1\to \bmu_m \to G  \stackrel{\chi_1}{\longrightarrow}\bmu_r \to 1.
$$
Then $X_\si= \bC^3/G$ which is the coarse moduli space of $\cX:=[\bC^3/G]$.
Let $N_\si= \bZ b_1 \oplus \bZ b_2 \oplus \bZ b_3$.
Define
$$
\Bsi=\{ v\in N: v=c_1 b_1 + c_2 b_2 + c_3 b_3, 0\leq c_i<1\}.
$$
There is a bijection $\Bsi \longrightarrow N/N_\si$ given by
$v\mapsto v+ N_\si$; there is a bijection $\Bsi \longrightarrow G$ given by
\begin{equation}\label{eqn:Box-to-G}
c_1 b_1 + c_2 b_2 + c_3 b_3\mapsto (e^{2\pi\sqrt{-1}c_1}, e^{2\pi\sqrt{-1}c_2}, e^{2\pi\sqrt{-1}c_3}).
\end{equation}
For example,  for $j\in \{1,\ldots,m-1\}$ we have
$$
(0,j,1) =\frac{j}{m}b_2 + (1-\frac{j}{m})b_3 \in \Bsi,\quad
(1, e^{2\pi\sqrt{-1}\frac{j}{m}}, e^{-2\pi\sqrt{-1}\frac{j}{m}}) \in G.
$$
Define $\age: \Bsi \to \{0,1,2\}$ by
$$
c_1 b_1 + c_2 b_2 + c_3 b_3 \mapsto c_1+ c_2 + c_3.
$$
Then we have a disjoint union
$$
\Bsi =\{0\}\cup \{v\in \Bsi: \age(v)=1\} \cup \{ v\in \Bsi:\age(v)=2\}.
$$

Given $h\in G$ and $i\in \{1,2,3\}$, define $c_i(h)\in [0,1)$ and $\age(h)\in \{0,1,2\}$ by
\begin{equation}\label{eqn:ci-h}
e^{2\pi\sqrt{-1}c_i(h)} =\chi_i(h),
\end{equation}
\begin{equation}\label{eqn:age-h}
\age(h)= c_1(h)+c_2(h)+c_3(h).
\end{equation}
The inverse map of \eqref{eqn:Box-to-G} is  given by $G\longrightarrow \Bsi$,
$$
h\mapsto c_1(h)b_1+c_2(h)b_2 + c_3(h) b_3.
$$

\subsection{The B-model geometry} \label{sec:B-geometry}
There exist $(m_1,n_1),\ldots, (m_p,n_p)\in \bZ^2$ such that
$$
\{v\in \Bsi:\age (v)=1\} = \{ (m_a,n_a,1): a=1,\ldots,p\}.
$$
We define $b_{3+a}=(m_a,n_a,1)$ for $a=1,\ldots, p$, and define
$$
H_f(X,Y,q)=X^r Y^{-s-rf} + Y^m + 1 +\sum_{a=1}^p q_a X^{m_a} Y^{n_a-m_a f}.
$$

The mirror of $\cX$ is a hypersurface in $\bC^2\times (\bC^*)^2$:
$$
\check{\cX}_q =\{ (u,v,X,Y)\in \bC^2\times (\bC^*)^2: H_f(X,Y,q)-uv=0 \}.
$$
This is a family of non-compact Calabi-Yau 3-folds parametrized
by $q=(q_1,\ldots, q_p)$. Different framing $f\in \bZ$ gives
isomorphic $\check{\cX}_q$ but different superpotential $X:\check{\cX}_q\to \bC^*$.
The framed mirror curve is
$$
\Si_q=\{(X,Y)\in (\bC^*)^2: H_f(X,Y,q)=0\}.
$$
Then
\begin{enumerate}
\item[(i)] $\check{\cX}_q$ is smooth iff $\Si_q$ is smooth.
\item[(ii)] All the branch points of $X:\Si_q\to \bC^*$ are simple iff
the critical points of $X: \check{\cX}_q \to \bC^*$ are isolated and non-degenerate.
\end{enumerate}
Note that $\Si_0 =\{ (X,Y)\in (\bC^*)^2: X^r Y^{-s-rf} + Y^m +1=0\}$ is smooth and all the branch points
of $X:\Si_0\to \bC^*$ are simple.  So (i) and (ii) hold
for small enough $q$.

\begin{remark}
When $m=1$, $G=\bmu_r$, $q_a=0$, the framed mirror curve is given by
$$
X^r Y^{-s-rf} + Y + 1 =0,
$$
or equivalently,
\begin{equation} \label{eqn:effective}
Y^{s+rf}(1+Y) + X^r =0.
\end{equation}
Up to sign, which is a matter of convention,  \eqref{eqn:effective}
agrees with the following framed mirror curve given
by Equation (3.28) of \cite{BSLM13}:
$$
Y^{s+rf}(1-Y)-X^r =0.
$$
In this paper, we assume the framing $f$ is a positive integer.
\end{remark}

We define \begin{equation}\label{eqn:weights}
w_1 =\frac{1}{r}, \quad
w_2 =\frac{s+rf}{rm}, \quad
w_3 = -w_1 -w_2 = \frac{-s-rf-m}{rm}.
\end{equation}

For later convenience, we fix two explicit bijections (which
are not necessarily group homomorphisms)
$\iota:\bZ_r\times \bZ_m$ and $\iota^*: G^*\to \bZ_r\times \bZ_m$.
\begin{itemize}
\item $\iota: \bZ_r\times \bZ_m \to G,\quad  (j,\ell)\mapsto \eta_1^j \eta_2^\ell$, where
\begin{equation}\label{eqn:eta}
\eta_1 = (e^{2\pi\sqrt{-1}w_1}, e^{2\pi\sqrt{-1}w_2}, e^{2\pi\sqrt{-1}w_3}),\quad
\eta_2 = (1, e^{2\pi\sqrt{-1}/m}, e^{-2\pi\sqrt{-1}/m}).
\end{equation}
\item $\iota^*: \bZ_r\times \bZ_m \to G^*, \quad  (j,\ell)\mapsto \chi_1^j \chi_2^\ell$, where
$\chi_1,\chi_2\in G^*$ are defined as in Section \ref{sec:A-geometry}.
\end{itemize}

\section{The 1-leg framed orbifold topological vertex} \label{sec:vertex}

Let $L=\{(z_1,z_2, z_3)\in \bC^3: |z_1|^2-|z_2|^2 = c,\ |z_1|^2-|z_3|^2=c,\
\Im(z_1z_2z_3)=0\}$, where $c>0$ and $\Im(z)$ means the imaginary part of $z$. Then $L$ is a special Lagrangian
of $\bC^3$ (Harvey-Lawson \cite{HL82}), and  $\cL:= [L/G]$
is a special Lagrangian sub-orbifold of the affine toric Calabi-Yau 3-orbifold $\cX=[\bC^3/G]$.
A framed Aganagic-Vafa A-brane of $\cX$ is a pair $(\cL,f)$, where $\cL\subset \cX$
is as above and $f\in \bZ$.
The 1-leg framed orbifold topological vertex is a relative
toric Calabi-Yau 3-orbifold $(\cY_f, \cD_f)$ associated
to $(\cX,\cL,f)$.

\subsection{The toric graph}
$N_\si =\bZ b_1\oplus \bZ b_2 \oplus \bZ b_3$ is a sublattice
of $N=\bZ e_1 \oplus \bZ e_2 \oplus \bZ e_3$ of index $rm$.
$\{ b_1,b_2,b_3 \}$ is a $\bQ$-basis of $N_\bQ:=N\otimes_{\bZ}\bQ$.
Let $\{\sw_1, \sw_2, \sw_3\}$ be the dual $\bQ$-basis of
$M_\bQ:= M\otimes_{\bZ}\bQ$. Then
$$
\sw_1 = \frac{1}{r}\su_1,\quad
\sw_2 = \frac{s}{rm}\su_1 +\frac{1}{m}\su_2,\quad
\sw_3 = -\frac{s+m}{rm}\su_1 -\frac{1}{m}\su_2 + \su_3.
$$
Note that
$$
\sw_1+\sw_2+\sw_3= \su_3
$$
is the weight of the $\bT$-action on $T_{\fp}\cX$, where $\fp = \BG$ is
the unique $\bT$-fixed point in the toric 3-orbifold $\cX$.
Let $\bT'$ be the kernel of $\su_3 \in M=\Hom(\bT,\bC^*)$.

The relative formal toric Calabi-Yau (FTCY) graph of the
1-leg framed orbifold topological vertex is shown in Figure 1.
This generalizes the 1-leg framed topological vertex \cite{LLLZ}.
\begin{figure}[h]
\begin{center}
\psfrag{w1}{\footnotesize $\sw_1=\frac{\su_1}{r}$}
\psfrag{w2}{\footnotesize $\sw_2= \frac{s\su_1}{rm}+\frac{\su_2}{m}$}
\psfrag{w3}{\footnotesize $\sw_3= -\frac{m+s}{mr}\su_1 -\frac{\su_2}{m}$}
\psfrag{-u1}{\footnotesize $-\su_1$}
\psfrag{u2-fu1}{\small $\frac{\su_2-f\su_1}{m}$}
\psfrag{-u2+fu1}{\small $\frac{-\su_2+f\su_1}{m}$}
\includegraphics[scale=0.8]{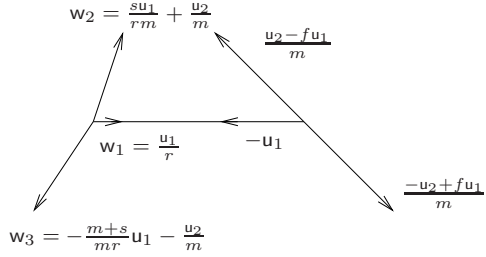}
\end{center}
\caption{The toric graph of the 1-leg framed orbifold topological vertex}
\end{figure}

\subsection{The fan} \label{sec:fan}
The toric graph in Figure 1 defines a relative toric Calabi-Yau 3-orbifold
$\cY_f$ which is a partial compactification of $\cX$. We now
describe the fan $\Si_f$ defining the coarse moduli space
$Y_f$ of $\cY_f$. Let
$$
v_1= b_1 = re_1-se_2 +e_3,\quad
v_2= b_2 = me_2+e_3,\quad
v_3= b_3 = e_3,\quad
v_4 = -e_1 -fe_2.
$$
Let $\si$ be the 3-cone spanned by
$v_1, v_2, v_3$, as before; let $\si'$ be
the 3-cone spanned by $v_2, v_3, v_4$.
Let $\tau$ be the 2-cone $\si\cap \si'$.
Let $\rho_i$ be the 1-cone spanned by $v_i$. Then
$$
\Si_f(1) = \{\rho_1,\rho_2,\rho_3,\rho_4\},\quad
\Si_f(3) = \{\si,\si'\}.
$$
The coarse moduli space $D_f$ is the $\bT$-invariant
divisor in $Y_f$ associated to $\rho_4$.
The pair $(\cY_f,\cD_f)$ is a relative toric Calabi-Yau
3-orbifold, where the relative Calabi-Yau condition is
$K_{\cY_f} + \cD_f =0$.
Let $\cD_i$ be the $\bT$-invariant divisor associated to $\rho_i$. Then
$$
K_{\cY_f} = -\cD_1-\cD_2-\cD_3-\cD_4,\quad
K_{\cY_f}+\cD_f = -\cD_1 -\cD_2 -\cD_3.
$$
Note that $K_{\cX_f}+\cD_f$ is the principal $\bT$-divisor associated
to $\chi^{\su_3} \in \bC[M]$.

\subsection{The root construction} \label{sec:root}
In this subsection, we give another description of the relative toric 3-orbifold
$(\cY_f,\cD_f)$. Let $\fp_0=[0,1]$ and $\fp_\infty=[1,0]$ be the
two torus fixed point in $\bP(1,r)$, where $\bP(1,r)$ is the weighted projective line with an orbifold point of order $r$.  Then $\fp_0$ is
the unique stacky point in $\bP(1,r)$, and
any $\bC^*$-equivariant line bundle on $\bP(1,r)$
is of the form
$$
\cO_{\bP(1,r)}(a \fp_0 + b\fp_\infty),
$$
where $a,b\in \bZ$. We have
$$
\deg \cO_{\bP(1,r)}(a \fp_0 + b\fp_\infty) = \frac{a}{r}+b.
$$
Let $p:\fl \to \bP(1,r)$ be the $\bZ_m$-gerbe obtained by
applying the $m$-th root construction to the equivariant line bundle
$$
\cO_{\bP(1,r)}(s\fp_0 + f\fp_\infty).
$$
Let $\cL_2$ be the tautological line bundle over $\fl$, so that
$$
\cL_2^{\otimes m}= p^* \cO_{\bP(1,r)}(s\fp_0+f\fp_\infty).
$$
Then
$$
\deg \cL_2 = \frac{s}{rm} +\frac{f}{m} = w_2.
$$
Let
$$
\cL_3:= \cL_2^{-1}\otimes p^*\cO_{\bP(1,r)}(-\fp_0).
$$
Then
$$
\deg \cL_3= -\frac{s}{rm}-\frac{f}{m} -\frac{1}{r} = w_3.
$$
The toric 3-orbifold $\cY_f$ is total space of the rank 2 vector bundle
$$
V_f:= \cL_2 \oplus \cL_3 \to \fl.
$$
The divisor $\cD_f$ is the fiber of $V_f$ over the stacky point
$p^{-1}(\fp_\infty)$ in $\fl$.

When $G$ is trivial, we have $r=m=1$, $s=0$, $\fl=\bP(1,r)=\bP^1$, and
$V_f= \cO_{\bP^1}(f)\oplus \cO_{\bP^1}(-f-1) \to \bP^1$.

\section{Chen-Ruan orbifold cohomology}\label{sec:Chen-Ruan}

In this section, we describe the Chen-Ruan orbifold cohomology \cite{CR04}
of the classifying space
$\BG$ of the finite abelian group $G$ and the $\tT$-equivariant  Chen-Ruan orbifold
cohomology of $\cX=[\bC^3/G]$ and $\cY_f$.
The Chen-Ruan orbifold cohomology of the classifying space
of any finite group is described in \cite{JK}.

\subsection{Chen-Ruan orbifold cohomology of $\BG$}
The inertia stack of $\BG$ is
$$
\IBG = \bigcup_{h\in G}(\BG)_h,
$$
where
$$
(\BG)_h =[\{h\}/G]\cong \BG.
$$
As a graded vector space over $\bC$,
$$
H^*_\CR(\BG;\bC) = H^*(\IBG;\bC) = \bigoplus_{h\in G} H^0((\BG)_h;\bC),
$$
where $H^0((\BG)_h;\bC)=\bC\one_h$.
The orbifold Poincar\'{e} pairing of $H^*_\CR(\BG;\bC)$ is given by
$$
\langle\one_h, \one_{h'}\rangle = \frac{\delta_{h^{-1},h'}}{|G|}.
$$
The orbifold cup product of $H^*_\CR(\BG;\bC)$ is given by
$$
\one_h \star \one_{h'} = \one_{h h'}.
$$

We now define a canonical basis for the semi-simple algebra
$H^*_\CR(\BG;\bC)$.
Given a character $\gamma\in G^*=\Hom(G,\bC^*)$, define
$$
\phi_\gamma :=\frac{1}{|G|}\sum_{h\in G}\chi_\gamma(h^{-1})\one_h.
$$
Then
$$
H^*_\CR(\BG;\bC) =\bigoplus_{h\in G}\bC \one_h  =\bigoplus_{\gamma \in G^*}\bC \phi_\gamma.
$$
Recall that we have the orthogonality of characters:
\begin{enumerate}
\item For any $\gamma,\gamma'\in G^*$,
$\displaystyle{
\frac{1}{|G|} \sum_{h\in G} \chi_{\gamma}(h^{-1})\chi_{\gamma'}(h)=\delta_{\gamma,\gamma'}
},$
\item For any $h, h'\in G$,
$\displaystyle{
\frac{1}{|G|}\sum_{\gamma\in G^*}\chi_\gamma(h^{-1})\chi_\gamma(h') =\delta_{h,h'}.
}$
\end{enumerate}
Therefore,
$$
\langle \phi_{\gamma}, \phi_{\gamma'}\rangle = \frac{\delta_{\gamma,\gamma'}}{|G|^2},
$$
and
$$
\phi_{\gamma}\star\phi_{\gamma'}= \delta_{\gamma, \gamma'}\phi_{\gamma}.
$$
Then $\{ \phi_\gamma: \gamma\in G^*\}$ is a canonical basis
of $H^*_\CR(\BG;\bC)$.

\subsection{Equivariant Chen-Ruan orbifold cohomology of $\cX=[\bC^3/G]$}
Given any $h\in G$, define $c_i(h)\in [0,1)\cap \bQ$ and $\age(h)\in \{0,1,2\}$
by Equation \eqref{eqn:ci-h} and \eqref{eqn:age-h} in

Section \ref{sec:B-geometry}, respectively.
Let $(\bC^3)^h$ denote the $h$-invariant subspace of $\bC^3$. Then
$$
\dim_\bC (\bC^3)^h  = \sum_{i=1}^3 \delta_{c_i(h),0}.
$$
The inertial stack of $\cX$ is
$$
\IX  = \bigcup_{h\in G}\cX_h,
\quad\textup{where }
\cX_h =[(\bC^3)^h/G].
$$
In particular,
$$
\cX_1 =[\bC^3/G]=\cX.
$$
As a graded vector space over $\bC$,
$$
H^*_\CR(\cX;\bC) = \bigoplus_{h\in G} H^*(\cX_h; \bC)[2\age(h)] =\bigoplus_{h\in G}\bC \one_h,
$$
where $\deg(\one_h)=2\age(h) \in \{0,2, 4\}$.
The orbifold Poincar\'{e} pairing of the (non-equivariant) Chen-Ruan
orbifold cohomology $H_\CR^*(\cX;\bC)$ is given by
$$
\langle \one_{h}, \one_{h'}\rangle_{\cX} =\frac{1}{|G|}\delta_{h^{-1},h'}
\cdot \delta_{0,\dim_\bC \cX_h}.
$$

Let $\cR=H^*(\cB\tT;\bC)=\bC[\sw_1,\sw_2,\sw_3]$, where
$\sw_1,\sw_2,\sw_3$ are the first Chern classes of the universal line bundles over $\cB\tT$.
The $\tT$-equivariant Chen-Ruan orbifold cohomology $H^*_{\CR,\tT}([\bC^3/G];\bC)$
is an $\cR$-module.  Given $h\in G$, define
$$
\be_h := \prod_{i=1}^3 \sw_i^{\delta_{c_i(h),0}} \in \cR.
$$
In particular,
$$
\be_1 = \sw_1\sw_2\sw_3.
$$
Then the $\tT$-equivariant Euler class of $\zero_h:=[0/G]$ in $\cX_h = [(\bC^3)^h/G]$ is
$$
e_{\tT}(T_{\zero_h}\cX_h) = \be_h \one_h \in H^*_{\tT}(\cX_h;\bC) = \cR \one_h.
$$

Let $\chi_1,\chi_2,\chi_3\in G^*$ be defined as in Section \ref{sec:A-geometry}.
The image of $\chi_i:G\to \bC^*$ is $\bmu_{r_i}$ for some positive integer $r_i$.
In particular $r_1=r$.
Define
$$
\cR' =\bC[\sw_1^{1/r_1},\sw_2^{1/r_2}, \sw_3^{1/r_3}]
$$
which is a finite extension of $\cR$. Let
$$
\cQ=\bC(\sw_1,\sw_2,\sw_3),\quad
\cQ'=\bC(\sw_1^{1/r_1}, \sw_2^{1/r_2}, \sw_3^{1/r_3})
$$
be the fractional fields of $\cR$, $\cR'$, respectively.
The $\tT$-equivariant Poincar\'{e} pairing of
$H^*_{\tT,\CR}(\cX;\bC)\otimes_{\cR}\cQ$ (which is isomorphic
to $H^*_{\CR}(\cX;\cQ)$ as a vector space over $\cQ$) is given by
$$
\langle \one_{h}, \one_{h'}\rangle_{\cX} =\frac{1}{|G|}\cdot\frac{\delta_{h^{-1},h'} }{\be_h}\in \cQ.
$$
The $\tT$-equivariant orbifold cup product  of $H^*_{\tT,\CR}(\cX;\bC)\otimes_{\cR}\cQ$ is given by
$$
\one_h \star_{\cX} \one_{h'} = \Bigl(\prod_{i=1}^3 \sw_i^{c_i(h)+c_i(h')-c_i(hh')}\Bigr) \one_{hh'}.
$$

Define
\begin{equation}\label{eqn:bar-one}
\bar{\one}_h:= \frac{\one_h}{\prod_{i=1}^3 \sw_i^{c_i(h)}} \in H^*_{\tT,\CR}(\cX;\bC)\otimes_{\cR}\cQ'.
\end{equation}
Then
$$
\langle \bar{\one}_h, \bar{\one}_{h'}\rangle_{\cX} =
\frac{\delta_{h^{-1},h'}}{|G|\sw_1\sw_2\sw_3},\quad
\bar{\one}_h \star_{\cX}\bar{\one}_{h'} =\bar{\one}_{hh'}.
$$

We now define a canonical basis for the semisimple algebra
$H^*_{\tT,\CR}(\cX;\bC)\otimes_{\cR}\cQ'$.
Given $\gamma\in G^*$, define
\begin{equation}\label{eqn:bar-phi}
\bar{\phi}_\gamma := \frac{1}{|G|}\sum_{h\in G} \chi_\gamma(h^{-1}) \bar{\one}_h.
\end{equation}
Then
$$
\langle \bar{\phi}_{\gamma},\bar{\phi}_{\gamma'}\rangle_{\cX} = \frac{\delta_{\gamma\gamma'}}{|G|^2\sw_1\sw_2\sw_3 },
\quad
\bar{\phi}_{\gamma}\star_{\cX} \bar{\phi}_{\gamma'} =  \delta_{\gamma \gamma'} \bar{\phi}_{\gamma}.
$$
So $\{\bar{\phi}_\gamma: \gamma\in G^*\}$ is a canonical basis of $H^*_{\tT,\CR}(\cX;\bC)\otimes_{\cR}\cQ'$.

\subsection{Equivariant Chen-Ruan orbifold cohomology of $\cY_f$}
We use the notation in Section \ref{sec:fan} and Section \ref{sec:root}: the coarse moduli
space $Y_f$ of the toric 3-orbifold $\cY_f$ is defined by
a simplicial fan $\Si_f$ with two 3-dimensional cones $\si$, $\si'$.
Let $\cX=[\bC^3/G]$ and $\cX'=[\bC^3/\bmu_m]$ be the affine toric
sub-orbifolds of $\cY_f$ associated to $\si$ and $\si'$, respectively.
As $\cQ$-vector spaces,
$$
H^*_{\tT,\CR}(\cY_f;\bC)\otimes_{\cR}\cQ= H^*_\CR(\cY_f;\cQ) = H^*_\CR(\cX;\cQ)\oplus H^*_\CR(\cX';\cQ),
$$
where
$$
H^*_\CR(\cX;\cQ)=\bigoplus_{h\in G}\cQ\one_h,\quad
H^*_\CR(\cX';\cQ) =\cQ \one' \oplus \bigoplus_{k=1}^{m-1}\cQ \one'_\frac{k}{m}.
$$
We have the following isomorphisms of $\cQ$-vector spaces (which do
not preserve the grading and the orbifold cup product):
\begin{eqnarray*}
&& H^*_\CR(\cY_f;\cQ) \cong H^*_\CR(\fl;\cQ),\quad
H^*_\CR(\cX;\cQ) \cong H^*_\CR(\BG;\cQ),\\
&& H^*_\CR(\cX';\cQ) \cong H^*_\CR(\cD_f;\cQ)\cong H^*_\CR(\fp_\infty;\cQ)
=H^*_\CR(\cB\bmu_m;\cQ).
\end{eqnarray*}

\section{A-model Topological String} \label{sec:Amodel}

\subsection{Relative orbifold Gromov-Witten invariants of $(\cY_f,\cD_f)$}
In this section, we define open-closed orbifold Gromov-Witten invariants
of $\cX$ as certain equivariant relative orbifold Gromov-Witten
invariants of $(\cY_f,\cD_f)$.
Recall from Section \ref{sec:root} that the orbifold $\cY_f$ can be viewed
as the total space of the rank 2 vector bundle $V = \cL_2 \oplus \cL_3 \to \fl$, where $\fl$ is a $\bmu_m$-gerbe over $\bP(1,r)$, so equivariant
relative orbifold Gromov-Witten invariants of $(\cY_f,\cD_f)$
are twisted equivariant relative orbifold Gromov-Witten invariants
of $(\fl,\fp_\infty)$.

Relative orbifold Gromov-Witten theory has been
constructed in algebraic geometry by Abramovich-Fantechi \cite{AF}
and in symplectic geometry by Chen-Li-Sun-Zhao \cite{CLSZ}.
We refer to \cite{AF, CLSZ} for the precise definitions of moduli spaces
of relative stable maps to smooth DM stacks.

We first introduce some notation.
We fix non-negative integers $g,\ell$ and $\vmu=((\mu_1,k_1),\ldots, (\mu_n,k_n))$,
where $\mu_j\in \bZ_{>0}$ and $k_j\in\{0,1\ldots, m-1\}$.
\begin{enumerate}
\item Let $\Mbar_{g,\ell}(\cY_f/\cD_f, \vmu)$
(resp. $\Mbar_{g,\ell}(\fl/\fp_\infty,\vmu)$)
be the moduli space of genus $g$ relative stable maps to $(\cY_f,\cD_f)$ (resp.  $(\fl,\fp_\infty)$) with $\ell$ marked points and $n$ relative points, where the relative points are ordered, and the ramification index (resp. monodromy) of the $j$-th relative point is $\mu_j$ (resp. $k_j$)

\item Let $\ev_i: \Mbar_{g,\ell}(\cY_f/\cD_f,\vmu)\to \cI\cY_f$
and $\ev_i:\Mbar_{g,\ell}(\fl/\fp_\infty,\vmu)\to \cI\fl$
be the evaluation map at the $i$-th marked point, where $i=1,\ldots, \ell$.

\item Let $\pi:\cU\to \Mbar_{g,\ell}(\fl/\fp_\infty,\vmu)$ be the  universal
curve, $\cT\to \Mbar_{g,\ell}(\fl/\fp_\infty,\vmu)$ be the universal
target, $F:\cU\to \cT$ be the universal map.

\item There is a contraction map $\cT\to \fl$.  Let $\hF:\cU\to \fl$
be the composition of this contraction map and the universal map
$F:\cU\to \cT$.

\item Let $D=\displaystyle{\bigcup_{1\leq j\leq n} D_j \subset \cU}$ be
the universal relative points corresponding to $\vmu$, and let $D_0=\displaystyle{\bigcup_{\substack{1\leq j\leq n\\ k_j=0}} D_j}
\subset D$.

\item  Let $\bT_f=\{(t, t^f, 1):t\in \bC^*\} \subset \bT\cong (\bC^*)^3$.
This induces
$$
\phi_f:H^*(B\bT;\bQ)=\bQ[\su_1,\su_2,\su_3]=\bQ[\sw_1,\sw_2,\sw_3]
\to H^*(B\bT_f;\bQ)=\bQ[\sv],
$$
$$
\su_1\mapsto \sv,\quad \su_2\mapsto f\sv,\quad \su_3\mapsto 0
$$
or equivalently,
$$
\sw_1\mapsto\frac{1}{r}\sv= w_1\sv,\quad
\sw_2\mapsto \frac{s+rf}{rm}\sv= w_2\sv,\quad
\sw_3\mapsto \frac{-s-rf-m}{rm}\sv= w_3\sv.
$$
We extend $\phi_f$ to
$\phi_f: \cQ=\bC(\su_1,\su_2,\su_3)=\bC(\sw_1,\sw_2,\sw_3)
\to \bC(\sv)$.
\end{enumerate}

For any
$$
\gamma_1,\ldots, \gamma_\ell \in H^*_\CR(\cX;\cQ)
=\bigoplus_{h\in G}\cQ\one_h \subset H^*_\CR(\cY_f;\cQ) \cong H^*_\CR(\fl;\cQ),
$$
we define the $\tT$-equivariant relative orbifold Gromov-Witten invariant
\begin{eqnarray*}
 \langle \gamma_1,\ldots, \gamma_\ell \rangle_{g,\vmu}^{\cX,(\cL,f)}
&:=& \frac{1}{|\Aut(\vmu)|}\cdot \phi_f \Big( \int_{[\Mbar_{g, \ell}(\cX_f/\cD_f,\vmu)^{\tT}]^{\textrm{vir}} }
\frac{(\su_2-f\su_1)^{\sum_{j=1}^n \delta_{k_j,0}} }{e_{\tT}(N^\vir)} \prod_{i=1}^\ell\ev_i^*\gamma_i \Big) \\
&=& \frac{1}{|\Aut(\vmu)|}\cdot \phi_f \Big(\int_{[\Mbar_{g,\ell}(\fl/\fp_\infty,\vmu)^{\tT}]^{\textrm{vir}} }
\frac{e_{\tT}\big(-R\pi_*(\hF^*\cL_2(-D_0)\oplus \hF^*\cL_3)\big)}{e_{\tT}(N^\vir)}
\prod_{i=1}^\ell \ev_i^* \gamma_i \Big)\\
&\in & \bQ(\sv).
\end{eqnarray*}
Note that $\phi_f(\su_2-f\su_1)=0$. The factor
$(\su_2-f\su_1)^{\sum_{j=1}^n \delta_{k_j,0}}$ is included to
cancel such factor in $e_{\tT}(N^\vir)$.

If $\gamma_1,\ldots, \gamma_\ell$ are homogeneous, then
\begin{equation}\label{eqn:v}
\langle \gamma_1,\ldots,\gamma_\ell\rangle_{g,\vmu}^{\cX,(\cL,f)}
= \sv^{\sum_{i=1}^\ell (\deg\gamma_i/2-1)}
\cdot\langle \gamma_1,\ldots,\gamma_\ell\rangle_{g,\vmu}^{\cX,(\cL,f)}\Big|_{\sv=1}.
\end{equation}

\subsection{Descendant orbifold Gromov-Witten invariants of $\cX$} \label{sec:open-closed}
In this (sub)section, we use localization to compute
$\langle\gamma_1,\ldots, \gamma_\ell\rangle_{g,\vmu}^{\cX,(\cL,f)}$ and
obtain an expression in terms of cubic abelian Hurwitz-Hodge
integrals.

We first give some definitions.
For $i=1,2,3$, let $\chi_i\in G^*=\Hom(G,\bC^*)$ be characters of $G$ defined
in Section \ref{sec:A-geometry}, and let $L_{\chi_i}$ be the line bundle
over $\BG$ associated to the irreducible character $\chi_i\in G^*$.
Then $\cX=[\bC^3/G]$ is the total space of
$$
\bigoplus_{i=1}^3 L_{\chi_i} \to \BG.
$$
Let $\Mbar_{g,n}(\BG)$ denote the moduli space of genus $g$, $n$-pointed
twisted stable maps to $\BG$. Let $\pi: \cC_{g,n} \to \Mbar_{g,n}(\BG)$
be the universal curve, and let $F:\cC_{g,n}\to \BG$ be the universal map.
Define
\begin{equation}\label{eqn:egn}
\be_{g,n} := \phi_f \Bigl( e_{\tT}
\Bigl(R\pi_* F^*(\bigoplus_{i=1}^3 L_{\chi_i})\Bigr)\Bigr)
=\prod_{i=1}^3 (w_i\sv)^{(\rank R\pi_* F^* L_{\chi_i}) }
c_{\frac{1}{w_i\sv}}(R\pi_* F^*L_{\chi_i}),
\end{equation}
where $c_t(E)= 1+ tc_1(E)+ t^2 c_2(E) + \cdots $ denotes the Chern polynomial
of a complex vector bundle $E$ (or more generally, an element in the
$K$-theory). The rank of $R\pi_*F^*L_{\chi_i}$ is
constant on a connected component of $\Mbar_{g,n}(\BG)$, but
might be different on different connected components of $\Mbar_{g,n}(\BG)$.

A moduli point in $\Mbar_{g,\ell}(\fl/\fp_\infty,\vmu)$
is represented by a morphism
$$
\rho:(\cC, x_1,\ldots, x_\ell, y_1,\ldots, y_n)\to \fl[k]=\fl\cup \Delta_1\cup \cdots\cup \Delta_k,
$$
where each $\Delta_i$ is a trivial $\bmu_m$-gerbe over $\bP^1$. Let
$$
\Mbar_{g,\ell}^0(\fl/\fp_\infty,\vmu)\subset \Mbar_{g,\ell}(\fl/\fp_\infty,\vmu)
$$
be the substack where the target is $\fl[0]=\fl$. Computations in \cite{RZ13} show that
the contribution from a $\tT$ fixed locus outside $\Mbar_{g,\ell}^0(\fl/\fp_\infty,\vmu)$
vanishes. Therefore,
$$
\langle \gamma_1,\ldots, \gamma_\ell \rangle_{g,\vmu}^{\cX,(\cL,f)}
= \frac{1}{|\Aut(\vmu)|}\cdot
\phi_f \Big(\int_{[\Mbar_{g,\ell}^0(\fl/\fp_\infty,\vmu)^{\tT}]^{\textrm{vir}} }
\frac{e_{\tT}\big(-R\pi_*(\hF^*\cL_2(-D_0)\oplus \hF^*\cL_3)\big)}{e_{\tT}(N^\vir)}
\prod_{i=1}^\ell \ev_i^* \gamma_i \Big).
$$

Suppose that $\rho:(\cC,x_1,\ldots, x_\ell,y_1,\ldots,y_n)\to \fl$ represents
a point in $\Mbar_{g,n+\ell}^0(\fl/\fp_\infty,\vmu)^{\tT}$. Then
$$
\cC=\cC_0\cup\bigcup_{j=1}^n \cC_j,
$$
where $\cC_0$ is an orbicurve of genus $g$, $\cC_1,\ldots,\cC_n$ are
1-dimensional toric orbifolds, $x_1,\ldots, x_\ell\in \cC_0$,
$y_j\in \cC_j$, and $\cC_0$ and $\cC_j$ intersect at a node $z_j$ for
$j=1,\ldots, n$. We view $z_j$ as a point on $\cC_j$ in order to determine its monodromy. Let $\rho_j:=\rho|_{\cC_j}:\cC_j\to \fl$. Then
$\rho_0$ is a constant map to $\fp_0=[0,1]$, and we have
the following diagram for $1\leq j\leq n$:
\begin{equation}\label{eqn:leg}
\begin{CD}
\cC_j@>{\rho_j}>> \fl\\
@VVV  @VVV\\
\bP^1@>{\bar{\rho}_j}>> \bP^1
\end{CD}
\end{equation}
where $\bar{\rho}_j([x,y])=[x^{\mu_j},y^{\mu_j}]$ is the map between coarse moduli spaces,
the monodromy around $y_j$ is $e^{2\pi\sqrt{-1}k_j/m}\in \bmu_m$, and
the monodromy around $z_j$ is
$$
(e^{2\pi\sqrt{-1}\mu_j w_1}, e^{2\pi\sqrt{-1}(\mu_j w_2 -\frac{k_j}{m})} , e^{2\pi\sqrt{-1}(\mu_j w_3 +\frac{k_j}{m})})=\eta_1^{\mu_j} \eta_2^{-k_j} \in G.
$$
For $(d_0,k)\in \bZ\times\bZ_m$, define $h(d_0,k):= \eta_1^{d_0}\eta_2^{-k}\in G$, and define
\begin{equation}\label{eqn:disk-function}
D'(d_0,k) =-\frac{r}{\sv}(-1)^{\lfloor d_0 w_3+\frac{k}{m}\rfloor}
\big(\frac{\sv}{d_0}\big)^{\age(h(d_0,k))-1}
\cdot \frac{\Gamma(d_0 (w_1+w_2)+c_3(h(d_0,k))) }{
\Gamma(d_0 w_1-c_1(h(d_0,k))+1)\Gamma(d_0 w_2 -c_2(h(d_0,k))+1)}.
\end{equation}
Then $D'(d_0,k)\in \bC \sv^{\age(h(d_0,k))-2}$.
Up to sign, $D'(d_0, k)$ is essentially $|G|$ times the disk function in \cite[Section 3.3]{Ro14}.
Note that
$$
(-1)^{\lfloor d_0 w_3+\frac{k}{m}\rfloor}
= (-1)^{d_0 w_3 - c_3(h(d_0,k))+\frac{k}{m}}.
$$

By virtual localization, we have:
\begin{proposition}
\begin{eqnarray*}
\langle \ \rangle^{\cX,(\cL,f)}_{0,(d_0,k)} &=&  \frac{\delta_{1,h(d_0,k)} }{|G|}   D'(d_0,k)\cdot \big(\frac{\sv}{d_0}\big)^2\\
\langle \gamma \rangle^{\cX,(\cL,f)}_{0,(d_0,k)} &=& \langle \gamma, \be_{h(d_0,k)}\one_{h(d_0,k)^{-1}}\rangle_{\cX}
\cdot D'(d_0,k)\cdot \frac{\sv}{d_0}\\
\langle \  \rangle^{\cX,(\cL,f)}_{0, (\mu_1,k_1), (\mu_2,k_2)} &=&
\langle \be_{h(\mu_1,k_1)}\one_{h(\mu_1,k_1)^{-1}},\be_{h(\mu_2,k_2)} \one_{h(\mu_2,k_2)^{-1}}\rangle_{\cX}\\
&&\cdot \frac{D'(\mu_1,k_1) D'(\mu_2,k_2)}{|\Aut((\mu_1,k_1),(\mu_2,k_2))|}\cdot \frac{\sv}{\mu_1+\mu_2}.
\end{eqnarray*}
If $\vmu=((\mu_1,k_1),\ldots, (\mu_n,k_n))$ and $2g-2+\ell+n>0$, then
$$
\langle \gamma_1,\ldots, \gamma_\ell \rangle_{g,\vmu}^{\cX,(\cL,f)}
= \frac{1}{|\Aut(\vmu)|}\prod_{j=1}^{n} D'(\mu_j,k_j)  \int_{\Mbar_{g,\ell+n}(BG)}
\frac{\prod_{i=1}^\ell \ev_i^*\gamma_i}{\be_{g,\ell+n}}
\prod_{j=1}^n  \frac{\ev_{\ell+j}^*(\be_{h(\mu_j,k_j)}\one_{h(\mu_j,k_j)^{-1}})}{
1-\frac{\mu_j}{\sv}\bar{\psi}_{\ell+j}}.
$$
\end{proposition}

Introduce variables $X_1,\ldots, X_n$ and
let
$$
\btau =\sum_{a=1}^p \tau_a \one_{h_a},
$$
where $h_a\in G$ corresponds to $(m_a,n_a,1)\in \{ v\in \Bsi: \age(v)=1\}$. For $k\in \{0,1, \ldots, m-1\}$, let
$\one'_{\frac{-k}{m}} :=  \one'_{1-\delta_{0,k}-\frac{k}{m}}\in H^*_{\CR}(\cB\bmu_m;\bC)=\bigoplus_{k=0}^{m-1} \bC\one'_\frac{k}{m}$. Define
\begin{eqnarray*}
F_{g,n}^{\cX,(\cL,f)}(\btau; X_1,\ldots, X_n)
&=&\sum_{\mu_1,\ldots,\mu_n>0}\sum_{k_1,\ldots,k_n=0}^{m-1}\sum_{\ell\geq 0}
\frac{\langle \btau^\ell \rangle^{\cX,(\cL,f)}_{g,(\mu_1,k_1),\ldots, (\mu_n,k_n)}}{\ell !}\\
&&\cdot
\prod_{j=1}^n (X_j)^{\mu_j}\cdot
(-1)^\frac{-k_1}{m}\one'_\frac{-k_1}{m}\otimes \cdots \otimes (-1)^\frac{-k_n}{m}\one'_\frac{-k_n}{m},
\end{eqnarray*}
which is an $H^*_{\CR}(\cB\bmu_m;\bC)^{\otimes n}$-valued function, where
$$
H^*_{\CR}(\cB\bmu_m;\bC)=\bigoplus_{k=0}^{m-1} \bC\one'_\frac{k}{m}.
$$

\begin{remark}
For each $k\in\{0,1,\ldots, m-1\}$, $\{ \mu_j: k_j=k\}$ determines a partition
$$
\mu^k = (\mu^k_1\ge \cdots \ge \mu^k_{\ell(\mu^k)} >0).
$$
We have
\begin{eqnarray*}
&& \Aut((\mu_1,k_1),\ldots, (\mu_n,k_n)) = \prod_{k=0}^{m-1}\Aut(\mu^k),\\
&& \ell(\mu^0)+\cdots + \ell(\mu^{m-1})=n.
\end{eqnarray*}
The correlator $\langle \btau^\ell\rangle^{\cX,(\cL,f)}_{g,(\mu_1,k_1),\ldots, (\mu_n,k_n)}$
is invariant under permutation of $(\mu_1,k_1),\ldots (\mu_n,k_n)$, so it only depends
on the $m$-tuple of partitions $\mu^0,\ldots,\mu^{m-1}$. So
$F_{g,n}^{\cX,(\cL,f)}(\btau; X_1,\ldots, X_n)$ is a symmetric function in
$X_1, \ldots, X_n$, and can be rewritten as
\begin{eqnarray*}
F_{g,n}^{\cX,(\cL,f)}(\btau; X_1,\ldots, X_n)
&=&\sum_{\substack{\mu^0,\ldots,\mu^{m-1}\in \cP\\ \ell(\mu^0)+\cdots + \ell(\mu^{m-1})=n } }
\sum_{\ell\geq 0}
\frac{\langle \btau^\ell \rangle^{\cX,(\cL,f)}_{g,\mu^0,\ldots,\mu^{m-1}}}{\ell !}
X_{\mu^0,\ldots,\mu^{m-1}},
\end{eqnarray*}
where $\cP$ is the set of partitions, and
$$
X_{\mu^0,\ldots,\mu^{m-1}} =\sum_{\si\in S_n} \prod_{k=0}^{m-1}
\prod_{j=\ell(\mu^0)+\cdots+\ell(\mu^{k-1})+1}^{\ell(\mu^0)+\cdots+ \ell(\mu^k)}
X^{\mu_j^k}_{\si(j)}((-1)^\frac{-k}{m}\one'_{\frac{-k}{m}} )^{\otimes \ell(\mu^k)}.
$$
\end{remark}

We introduce some notation.
\begin{enumerate}
\item Given $h\in G$, define
\begin{eqnarray*}
\Phi^h_0(X)&:=&\frac{1}{|G|}
\sum_{\substack{(d_0,k)\in \bZ_{\geq 0}\times \bZ_m \\ h(d_0,k)=h} } D'(d_0,k) X^{d_0} (-1)^{-k/m}\one'_\frac{-k}{m} \\
&=& -\sum_{\substack{(d_0,k)\in \bZ_{\geq 0}\times \bZ_m \\ h(d_0,k)=h} }
\frac{1}{m\sv}e^{\sqrt{-1}\pi(d_0 w_3-c_3(h))}
\big(\frac{\sv}{d_0}\big)^{\age(h)-1} \\
&& \cdot \frac{\Gamma(d_0 (w_1+w_2)+c_3(h)) }{
\Gamma(d_0 w_1-c_1(h)+1)\Gamma(d_0 w_2 -c_2(h)+1)}X^{d_0} \one'_\frac{-k}{m}.
\end{eqnarray*}
Then $\Phi^h_0(X)$ takes values in $\bigoplus_{k=0}^{m-1}\bC \sv^{\age(v)-2} \one'_\frac{k}{m}$.

For $a\in \bZ$ and $h\in G$,  we define
$$
\Phi^h_a(X) := \frac{1}{|G|}\sum_{\substack{d_0> 0 \\ h(d_0,k)=h} } D'(d_0,k)(\frac{d_0}{\sv})^a X^{d_0}(-1)^{-k/m}
\one'_\frac{-k}{m}.
$$
Then $\Phi^h_a(X)$ takes values in $\bigoplus_{k=0}^{m-1}\bC \sv^{\age(v)-2-a} \one'_\frac{k}{m}$, and
$$
\Phi^h_{a+1}(X)=(\frac{1}{\sv}X\frac{d}{dX})\Phi^h_a(X).
$$

\item For $a\in \bZ$ and $\gamma\in G^*$, we define
$$
\txi^\gamma_a(X):= |G|\sum_{h\in G}\chi_\gamma(h^{-1})
\bigl(\prod_{i=1}^3 (w_i\sv)^{1-c_i(h)}\bigr) \Phi^h_a(X).
$$
Then $\txi^\gamma_a(X)$ takes values in $\bigoplus_{k=0}^{m-1}\bC \sv^{1-a} \one'_\frac{k}{m}$.
\item Given $\gamma_1,\gamma_2\in H^*_{\CR}(\cX;\bC)\otimes_{\bC}\bC(\sv)$
and $a_1,a_2\in \bZ_{\geq 0}$, define
\begin{eqnarray*}
\langle \tau_{a_1}(\gamma_1)\rangle^{\cX, \bT_f}_{0,1} &=& \delta_{a_1,0}\langle 1, \gamma_1\rangle_{\cX}\cdot \sv^2 \\
\langle\tau_{a_1}(\gamma_1)\tau_{a_2}(\gamma_2)\rangle^{\cX, \bT_f}_{0,2} &=&  \delta_{a_1,0}\delta_{a_2,0}\cdot \sv
\langle \gamma_1,\gamma_2\rangle_{\cX}.
\end{eqnarray*}

\item Given $\gamma_1,\ldots, \gamma_n\in H^*_{\CR}(\cX;\bC)\otimes\bC(\sv)$, and $a_1,\ldots, a_n\in \bZ_{\geq 0}$,
define
$$
\langle \tau_{a_1}(\gamma_1)\cdots \tau_{a_n}(\gamma_n)\rangle^{\cX,\bT_f}_{g,n}
=\int_{\Mbar_{g,n}(\BG)} \frac{1}{\be_{g,n}} \prod_{i=1}^n (\ev_i^*(\gamma_i)\bar{\psi}_i^{a_i}).
$$

\item Given $h\in G$, define $\one_h^*=|G|\phi_f(\be_h\one_{h^{-1}})$.
\item In the rest of this paper, we consider $\bT_f$-equivariant cohomology,
and write $\phi_\gamma$ and $\bar{\phi}_\gamma$ instead of $\phi_f(\phi_\gamma)$ and $\phi_f(\bar{\phi}_\gamma)$,
respectively. (Recall that $\phi_f(\sw_i)=w_i\sv$.)
\end{enumerate}

With the above notation, we have:
\begin{proposition}\label{thm:open-descendant}
\begin{enumerate}
\item {\em (disk invariants)}
\begin{eqnarray*}
&& F_{0,1}^{\cX,(\cL,f)}(\btau;X)\\
&=&\Phi^1_{-2}(X) +\sum_{a=1}^p \tau_a \Phi^{h_a}_{-1}(X)+ \sum_{a\in \bZ_{\geq 0}}
\sum_{h\in G}
\sum_{\ell\geq 2}\frac{\langle \btau^\ell, \tau_a(\one_h^*)\rangle^{\cX,\bT_f}_{0,\ell+1} }{\ell!}
\Phi^h_a(X) \\
&=& \frac{1}{|G|^2 w_1 w_2 w_3}\Big( \sum_{\gamma\in G^*}\txi^{\gamma}_{-2}(X)
+ \sum_{a=1}^p \tau_a \prod_{i=1}^3 w_i^{c_i(h_a)}\sum_{\gamma\in G^*}
\chi_\gamma(h_a)\txi^\gamma_{-1}(X)\Big)\Big|_{\sv=1}\\
&&  +\sum_{a\in \bZ_{\geq 0}}\sum_{\gamma\in G^*}
\sum_{\ell\geq 2}\frac{\langle \btau^\ell, \tau_a(\bar{\phi}_\gamma)\rangle^{\cX,\bT_f}_{0,\ell+1} }{\ell!}\txi^\gamma_a(X).
\end{eqnarray*}
\item{\em (annulus invariants)}
\begin{eqnarray*}
&& F_{0,2}^{\cX,(\cL,f)}(\btau; X_1,X_2)-F^{\cX,(\cL,f)}_{0,2}(0;X_1,X_2)\\
&=& \sum_{a_1,a_2\in \bZ_{\geq 0}}
\sum_{h_1,h_2\in G}
\sum_{\ell\geq 1}\frac{\langle \btau^\ell, \tau_{a_1}(\one_{h_1}^*)\tau_{a_2}(\one_{h_2}^*)\rangle^{\cX,\bT_f}_{0,\ell+2} }{\ell!}
\Phi^{h_1}_{a_1}(X_1)\Phi^{h_2}_{a_2}(X_2) \\
&=& \sum_{a_1, a_2\in \bZ_{\geq 0}}
\sum_{\gamma_1, \gamma_2\in G^*}
\sum_{\ell\geq 1}\frac{\langle \btau^\ell, \tau_{a_1}(\bar{\phi}_{\gamma_1})
\tau_{a_2}(\bar{\phi}_{\gamma_2})\rangle^{\cX,\bT_f}_{0,\ell+2} }{\ell!}
\txi^{\gamma_1}_{a_1}(X_1)\txi^{\gamma_2}_{a_2}(X_2),
\end{eqnarray*}
where
\begin{equation}\label{eqn:annulus-zero}
\begin{aligned}
& (X_1\frac{\partial}{\partial X_1}+ X_2\frac{\partial}{\partial X_2}) F_{0,2}^{\cX,(\cL,f)}(0;X_1,X_2)\\
=&|G| \Big(\sum_{h\in G} \be_h \Phi^h_0(X_1) \Phi^{h^{-1}}_0(X_2)\Big)\Big|_{\sv=1}
= \frac{1}{|G|^2 w_1w_2w_3} \Big(\sum_{\gamma\in G^*}\big(\txi^\gamma_0(X_1)\txi^\gamma_0(X_2)\Big)\Big|_{\sv=1}.
\end{aligned}
\end{equation}

\item For $2g-2+n>0$,
\begin{eqnarray*}
&& F_{g,n}^{\cX,(\cL,f)}(\btau; X_1,\ldots, X_n)\\
&=&\sum_{a_1,\ldots, a_n\in \bZ_{\geq 0}}
\sum_{h_1,\ldots, h_n\in G}
\sum_{\ell\geq 0}\frac{\langle \btau^\ell, \tau_{a_1}(\one_{h_1}^*)\ldots\tau_{a_n}(\one_{h_n}^*)\rangle^{\cX,\bT_f}_{g,\ell+n} }{\ell!}
\prod_{j=1}^n \Phi^{h_j}_{a_j}(X_j) \\
&=&\sum_{a_1,\ldots, a_n\in \bZ_{\geq 0}}
\sum_{\gamma_1,\ldots, \gamma_n\in G^*}
\sum_{\ell\geq 0}\frac{\langle \btau^\ell, \tau_{a_1}(\bar{\phi}_{\gamma_1})\ldots
\tau_{a_n}(\bar{\phi}_{\gamma_n})\rangle^{\cX,\bT_f}_{g,\ell+n} }{\ell!}
\prod_{j=1}^n \txi^{\gamma_j}_{a_j}(X_j).
\end{eqnarray*}
\end{enumerate}
\end{proposition}
\begin{remark}
$F_{0,2}^{\cX,(\cL,f)}(0;X_1,X_2)$ is an $H^*(\cB\bmu_m;\bC)^{\otimes 2}$-valued
power series in $X_1,X_2$ which vanishes at $(X_1,X_2)=(0,0)$,
so it is determined by \eqref{eqn:annulus-zero}.
\end{remark}

\subsection{A-model graph sum}\label{sec:Agraph}
Introduce formal variables
\begin{eqnarray*}
\hbu &=& \sum_{a\geq 0} \hu_a z^a =\sum_{a\geq 0}\sum_{\beta\in G^*}  u_a^\beta \bar{\phi}_\beta z^a
=\sum_{\beta\in G^*} \bu^\beta(z) \bar{\phi}_\beta,
\end{eqnarray*}
where
$$
u_a^\beta\in \cQ,\quad \hu_a =\sum_{\beta\in G^*}u_a^\beta \bar{\phi}_\beta \in H^*_\CR(\cX;\cQ),\quad
\bu^\beta(z)=\sum_{a\geq 0}u^\beta_a z^a.
$$
For $j=1,\dots,n$, define $\hbu_j$ and $\bu_j^\beta(z)$ similarly. Define
$$
\langle \hbu^\ell, \hbu_1,\ldots,\hbu_n\rangle^{\cX, \bT_f}_{g,\ell+n}
= \sum_{a_1,\ldots,a_\ell,b_1,\ldots,b_n\in \bZ_{\geq 0}}
\langle \tau_{a_1}(\hu_{a_1})\cdots \tau_{a_\ell}(\hu_{a_\ell}) \tau_{b_1}((\hu_1)_{b_1})\cdots
\tau_{b_n}((\hu_n)_{b_n})\rangle_{g,\ell+n}^{\cX,\bT_f}.
$$

By \cite[Theorem 4.2]{FLZ1}, $\displaystyle{ \frac{1}{\ell!}\langle \hbu^\ell,\hbu_1,\ldots,\hbu_n\rangle^{\cX,\bT_f}_{g,\ell+n} }$
can be written as a graph sum.  By Proposition \ref{thm:open-descendant},
\begin{equation}\label{eqn:Fgn}
F^{\cX,(\cL,f)}_{g,n}(\btau;X_1,\ldots,X_n)
=\sum_{\ell\geq 0}\frac{1}{\ell!}\langle \btau^\ell,\txi(X_1),\ldots,\txi(X_n)\rangle^{\cX,\bT_f}_{g,\ell+n},
\end{equation}
where
$$
\txi(X)=\sum_{a\in\bZ_{\geq 0}}\sum_{\beta\in G^*} \txi^\beta_a(X) \bar{\phi}_\beta z^a.
$$
Therefore, a graph sum of $F^{\cX,(\cL,f)}_{g,n}(\btau;X_1,\ldots, X_n)$ can be obtained
by the following specialization of \cite[Theorem 2]{FLZ1}:
$$
\hbu=\btau =\sum_{a=1}^p \tau_a \one_{h_a},\quad \hbu_j = \txi(X_j),\quad \sw_i = w_i \sv.
$$
By \eqref{eqn:v}, we may let $\sv=1$, i.e., $\sw_i = w_i$.
In the rest of this subsection, we give the precise statement of
this graph sum.

Given a connected graph $\Ga$, we introduce the following notation.
\begin{enumerate}
\item $V(\Ga)$ is the set of vertices in $\Ga$.
\item $E(\Ga)$ is the set of edges in $\Ga$.
\item $H(\Ga)$ is the set of half edges in $\Gamma$.
\item $L^O(\Ga)$ is the set of open leaves in $\Ga$.
The open leaves are ordered: $L^O(\Ga)=\{ l_1,\ldots, l_n\}$ where
$n$ is the number of open leaves. (Open leaves correspond
to {\em ordered} ordinary leaves in \cite{FLZ1}.)
\item $L^o(\Ga)$ is the set of primary leaves in $\Ga$.
The primary leaves are unordered. (Primary leaves correspond
to {\em unordered} ordinary leaves in \cite{FLZ1}.)
\item $L^1(\Ga)$ is the set of dilaton leaves in $\Ga$. The dilaton leaves
are not ordered.
\end{enumerate}

With the above notation, we introduce the following labels:
\begin{enumerate}
\item (genus) $g: V(\Ga)\to \bZ_{\geq 0}$.
\item (marking) $\alpha: V(\Ga) \to G^*$. This induces
$$
\alpha:L(\Ga)=L^O(\Ga)\cup L^o(\Ga)\cup L^1(\Ga)\to G^*,
$$
as follows: if $l\in L(\Ga)$ is a leaf attached to a vertex $v\in V(\Ga)$,
define $\alpha(l)=\alpha(v)$.
\item (height) $k: H(\Ga)\to \bZ_{\geq 0}$.
\end{enumerate}

Given a vertex $v\in V(\Ga)$, let $H(v)$ denote the set of half edges
emanating from $v$. The the valency $\val(v)$ of the vertex $v$
is equal to the size $|H(v)|$ of the set $H(v)$.
A labeled graph $\vGa=(\Ga,g,\alpha,k)$ is {\em stable} if
$$
2g(v)-2 + \val(v) >0
$$
for all $v\in V(\Ga)$.

Let $\bGa(\BG)$ denote the set of all stable labeled graphs
$\vGa=(\Gamma,g,\alpha,k)$. The genus of a stable labeled graph
$\vGa$ is defined to be
$$
g(\vGa):= \sum_{v\in V(\Ga)}g(v)  + |E(\Ga)|- |V(\Ga)|  +1
=\sum_{v\in V(\Ga)} (g(v)-1) + (\sum_{e\in E(\Gamma)} 1) +1.
$$
Define
$$
\bGa_{g,\ell,n}(\BG)=\{ \vGa=(\Gamma,g,\alpha,k)\in \bGa(\BG): g(\vGa)=g, |L^o(\Ga)|=\ell, |L^O(\Ga)|=n\}.
$$

Let
$$
R(z)=  \prod_{i=1}^3 \exp\Bigl(\sum_{m\geq 1}\frac{(-1)^{m}}{m(m+1)}  A^i_{m+1} (\frac{z}{w_i})^m \Bigr),
$$
where $A^i_{m}$ is the operator defined by
$$
A^i_m: H^*(\IBG;\bC)\to H^*(\IBG;\bC),\quad
\one_h \mapsto  B_m(c_i(h)) \one_h,
$$
and $B_m(x)$ is the Bernoulli polynomial. Then, relative to the basis $\{\phi_\gamma:\gamma\in G^*\}$,
\begin{equation}\label{eqn:Rcan}
R(z)^\alpha_\beta =\frac{1}{|G|}\sum_{h\in G} \chi_\alpha(h)
\chi_\beta(h^{-1})\prod_{i=1}^3 \exp\Bigl(\sum_{m=1}^\infty \frac{(-1)^m}{m(m+1)} B_{m+1}(c_i(h))(\frac{z}{w_i})^m \Bigr).
\end{equation}
Given $\beta\in G^*$, define
$$
\txi^\beta(z,X) = \sum_{a\geq 0} z^a \txi^\beta_a(X)\big|_{\sv=1}.
$$

We assign weights to leaves, edges, and vertices of a labeled graph $\vGa\in \bGa(\cY)$ as follows.
\begin{enumerate}
\item {\em Open leaves.} To each open leaf $l_j \in L^O(\Ga)$ with  $\alpha(l_j)= \alpha\in G^*$
and  $k(l_j)= k\in \bZ_{\geq 0}$, we assign:
$$
(\cL^{\txi})^\alpha_k(l_j) = [z^k] (\frac{1}{|G|\sqrt{w_1w_2w_3}}
\sum_{\beta\in G^*}R(-z)^\beta_\alpha \txi^\beta(z,X_j)).
$$
\item {\em Primary leaves.} To each primary leaf $l\in L^o(\Ga)$ with
$\alpha(l)= \alpha\in G^*$ and  $k(l)= k\in \bZ_{\geq 0}$, we assign:
\begin{equation}\label{eqn:primaryA}
(\cL^{\btau})^\alpha_k(l) =[z^k] (\frac{1}{|G|\sqrt{w_1w_2w_3}}\sum_{\beta\in G^*}\sum_{a=1}^p
\prod_{i=1}^{3}w_i^{c_i(h_a)}R(-z)^\beta_\alpha\chi_\beta(h_a)\tau_a ).
\end{equation}
By the orthogonality of the characters, we have
\begin{eqnarray*}
&& \sum_{\beta\in G^*} \sum_{a=1}^p\prod_{i=1}^{3}w_i^{c_i(h_a)} R(-z)^\beta_\alpha\chi_\beta(h_a)\tau_a \\
&=&   \sum_{a=1}^p\prod_{i=1}^{3}w_i^{c_i(h_a)} \chi_\alpha(h_a)\tau_a
\exp \Big( \sum_{m=1}^\infty\frac{(-1)^{m+1}}{m(m+1)} B_{m+1}  \left( c_i(h_a) \right) (\frac{z}{w_i})^m\Bigr).
\end{eqnarray*}
Here we used the fact that $B_m(1-x)=(-1)^mB_m(x)$, $B_m=0$ when $m$ is odd and $c_i(h)+c_i(h^{-1})=1-\delta_{c_i(h),0}$.

\item {\em Dilaton leaves.} To each dilaton leaf $l \in L^1(\Ga)$ with $\alpha(l)=\alpha \in G^*$
and $2\leq k(l)=k \in \bZ_{\geq 0}$, we assign
$$
(\cL^1)^\alpha_k(l) = [z^{k-1}](-\frac{1}{|G|\sqrt{w_1w_2w_3}} \sum_{\beta\in G^*} R(-z)^\beta_\alpha).
$$
\item {\em Edges.} To an edge connecting a vertex marked by $\alpha\in G^*$ and
a vertex marked by $\beta\in G^*$,
and with heights $k$ and $l$ at the corresponding half-edges, we assign
$$
\cE^{\alpha,\beta}_{k,l}(e) = [z^k w^l]
\Bigl(\frac{1}{z+w} (\delta_{\alpha,\beta}-\sum_{\gamma\in G^*} R(-z)^\gamma_\alpha R(-w)^\gamma_\beta)\Bigr).
$$
\item {\em Vertices.} To a vertex $v$ with genus $g(v)=g\in \bZ_{\geq 0}$ and with
marking $\alpha(v)=\gamma\in G^*$, with $n$ primary or open
leaves and half-edges attached to it with heights $k_1, ..., k_n \in \bZ_{\geq 0}$ and $m$
dilaton leaves with heights $k_{n+1}, \ldots, k_{n+m}\in \bZ_{\geq 0}$, we assign
$$
(|G|\sqrt{w_1w_2w_3})^{2g-2+\val(v)} \int_{\Mbar_{g,n+m}}\psi_1^{k_1} \cdots \psi_{n+m}^{k_{n+m}}.
$$
\end{enumerate}

Given a labeled graph $\vGa\in \bGa_{g,\ell,n}(\BG)$ with
$L^O(\Ga)=\{l_1,\ldots,l_n\}$, we define its A-model weight to be
\begin{equation}\label{eqn:wA}
\begin{aligned}
w_A(\vGa) =& \prod_{v\in V(\Ga)} (|G|\sqrt{w_1w_2w_3})^{2g-2+\val(v)} \langle \prod_{h\in H(v)} \tau_{k(h)}\rangle_{g(v)}
 \prod_{e\in E(\Ga)} \cE^{\alpha(v_1(e)),\alpha(v_2(e))}_{k(h_1(e)),k(h_2(e))}(e)\\
& \cdot \prod_{l\in L^o(\Ga)}(\cL^{\btau})^{\alpha(l)}_{k(l)}(l)
\prod_{j=1}^n(\cL^{\txi})^{\alpha(l_j)}_{k(l_j)}(l_j)
\prod_{l\in L^1(\Ga)}(\cL^1)^{\alpha(l)}_{k(l)}(l).
\end{aligned}
\end{equation}
Setting  $\hbu=\btau$, $\hbu_j =\txi(X_j)$, and $\sw_i=w_i$ in
\cite[Theorem 4.2]{FLZ1}, we obtain:
\begin{equation}\label{eqn:gln}
\frac{1}{\ell!} \langle \btau^\ell, \txi(X_1),\ldots , \txi(X_n)\rangle^{\cX}_{g,\ell+n}
=\sum_{\vGa\in \bGa_{g,\ell,n}(\BG)} \frac{w_A(\vGa)}{|\Aut(\vGa)|}.
\end{equation}
Let
$$
\bGa_{g,n}(\cX):= \bigcup_{\ell\geq 0} \bGa_{g,\ell,n}(\BG).
$$
Equations \eqref{eqn:Fgn}, \eqref{eqn:gln} and Proposition \ref{thm:open-descendant}
imply the following A-model graph sum formulae.
\begin{theorem}[A-model graph sum] \label{thm:Asum}
\begin{enumerate}
\item {\em (disk invariants)}
$$
F_{0,1}^{\cX,(\cL,f)}(\btau;X) = \Phi^1_{-2}(X) + \sum_{a=1}^p \tau_a \Phi^{h_a}_{-1}(X)
 +\sum_{\vGa\in \bGa_{0,1}(\cX)}\frac{w_A(\vGa)}{|\Aut(\vGa)|},
$$
where
\begin{equation}\label{eqn:disk-zero}
\Phi^1_{-2}(X) + \sum_{a=1}^p \tau_a \Phi^{h_a}_{-1}(X)
= \frac{1}{|G|^2 w_1 w_2 w_3}\Big( \sum_{\gamma\in G^*}\txi^\gamma_{-2}(X)
+ \sum_{a=1}^p \tau_a \prod_{i=1}^3 w_i^{c_i(h_a)}\sum_{\gamma\in G^*}
\chi_\gamma(h_a)\txi^\gamma_{-1}(X)\Big)\Big|_{\sv=1}.
\end{equation}
\item {\em (annulus invariants)}
$$
F_{0,2}^{\cX,(\cL,f)}(\btau;X_1, X_2) = F_{0,2}^{\cX,(\cL,f)}(0;X_1,X_2)+
 \sum_{\vGa\in \bGa_{0,2}(\cX)}\frac{w_A(\vGa)}{|\Aut(\vGa)|},
$$
where $F_{0,2}^{\cX,(\cL,f)}(0;X_1,X_2)$ is determined by Equation \eqref{eqn:annulus-zero}.
\item For $2g-2+n>0$,
$$
F_{g,n}^{\cX,(\cL,f)}(\btau;X_1,\ldots, X_n)
=\sum_{\vGa\in \bGa_{g,n}(\cX)}\frac{w_A(\vGa)}{|\Aut(\vGa)|}.
$$
\end{enumerate}
\end{theorem}

\section{The mirror curve and its compactification} \label{sec:mirror-curve}
The equation of the framed mirror curve is given by
$$
X^r Y^{-s-rf} + Y^m + 1 +\sum_{a=1}^p q_a X^{m_a} Y^{n_a-m_af}=0,
$$
where
$$
m_a = r c_1(h_a),\quad n_a = -sc_1(h_a) + m c_2(h_a).
$$
Let $\triangle$ be the triangle on $\bR^2$ with vertices $(r,-s)$, $(0,m)$, and $(0,0)$.
The compactified mirror curve $\bar{\Si}_q$ is embedded in the toric surface $S_\triangle$
defined by the  polytope $\triangle$. In this section, we assume $q_1,\ldots, q_p\in \bC$ are sufficiently small, so that the compactified mirror curve $\bar{\Si}_q$ is
a smooth compact Riemann surface. Let $\fg$ be the genus of $\bar{\Si}_q$,
and let $\fn$ be the number of points in $\bar{\Si}_q\setminus \Si_q$.

\subsection{The Liouville 1-form}\label{sec:SW}
Let $\tSi_q:=\pi^{-1}(\Si_q) \subset \bC^2$
be the preimage of $\Si_q\subset (\bC^*)^2$ under the (universal)
covering map $\pi: \bC^2 \to (\bC^*)^2$ defined by $(x,y)\mapsto (X=e^{-x},Y=e^{-y})$.
Then $\pi_q:=\pi|_{\tSi_q}: \tSi_q\to \Si_q$ is a regular $\bZ^2$-cover.
The 1-form $\log Y\frac{dX}{X}$ is multi-valued on $(\bC^*)^2$, but $\pi^*(\log Y \frac{dX}{X}) = ydx$
is a well-defined holomorphic 1-form on
$\bC^2$; indeed, it is the Liouville 1-form of the cotangent bundle $T^*\bC=(\bC)^2$:
$$
d(y dx)=dy\wedge dx.
$$
We define
$$
\Phi_q:= \big(\log Y \frac{dX}{X}\big)\Big|_{\Si_q}
$$
which is a multi-valued 1-form on $\Si_q$. Then $\pi_q^*\Phi_q = ydx|_{\tSi_q}$ is a well-defined
holomorphic 1-form on $\tSi_q$.

Give $q\in \bC^p$ such that $\Sigma_q$ is smooth, We choose a base point $p_q\in \Si_q$ and a point
$\tp_q\in \pi_q^{-1}(p_q)$. We have a commutative diagram:
\begin{equation}\label{eqn:fundamental-group}
\begin{CD}
1 @>>> \pi_1(\tSi_q,\tp_q) @>{\pi_{q*}}>> \pi_1(\Si_q,p_q) @>{I_{q*}}>> \pi_1((\bC^*)^2,p_q) @>>> 1\\
& & @V{\talpha}VV @V{\alpha}VV @V{\bar{\alpha}}VV & \\
&  & H_1(\tSi_q;\bZ) @>{\pi_{q*}}>> H_1(\Si_q;\bZ) @>{I_{q*}}>> H_1((\bC^*)^2;\bZ) @>>> 0.
\end{CD}
\end{equation}
In the above diagram:
\begin{itemize}
\item $I_{q*}$ is induced by the inclusion map $I_q:\Si_q\hookrightarrow (\bC^*)^2$.
\item The first row is a short exact sequence of multiplicative groups.
\item The second row is exact at $H_1(\Si_q;\bZ)=\bZ^{2\fg+\fn-1}$ and $H_1((\bC^*)^2;\bZ)=\bZ^2$.
\item The maps $\talpha$ and $\alpha$ are surjective group homomorphisms given by abelianization,
and $\bar{\alpha}$ is a group isomorphism.
\item $\pi_1(\Si_q,p_q)$ is a free group generated by $2\fg+\fn-1$ elements.
\end{itemize}
We define
$$
K_1(\Si_q;\bZ):= \mathrm{Ker}\left(I_{q*}: H_1(\Si_q;\bZ)\to H_1((\bC^*)^2;\bZ)\right)
=\mathrm{Im}\left(\pi_{q*}: H_1(\tSi_q;\bZ)\to H_1(\Si_q;\bZ)\right).
$$
Then $K_1(\Si_q;\bZ)\cong\bZ^{2\fg+\fn-3}$. There is a well-defined map
$$
\pi_1(\tSi_q,\tp_q)\longrightarrow \bC,\quad \tA \mapsto \int_{\tA} \pi_q^*\Phi_q.
$$
Given any $A \in K_1(\Si_q;\bZ)$, there exists $\tA\in \pi_1(\tSi_q,\tp_q)$ such that
$$
A = \pi_{q*} \circ \talpha (\tA)= \alpha \circ \pi_{q*} (\tA).
$$
If $\tA_1,\tA_2 \in \pi_1(\tSi_q,\tp_q)$ and $\alpha\circ \pi_{q*}(\tA_1)=\alpha\circ \pi_{q*}(\tA_2)= A$ then
$\pi_{q*}(\tA_2) = \pi_{q*}(\tA_1)B_1 B_2 B_1^{-1}B_2^{-1}$ for some $B_1,B_2\in \pi_1(\Si_q,p_q)$.
If $I_*(B_1)=(m_1,n_1)$ and $I_*(B_2)=(m_2,n_2)$, where $m_1,n_1,m_2,n_2\in \bZ$, then
$$
\int_{\tA_2} \pi_q^* \Phi = \int_{\tA_1}\pi_q^*\Phi + (2\pi\sqrt{-1})^2(m_1n_2-m_2n_1).
$$
So we have a map
\begin{equation}\label{eqn:integral-Z}
K_1(\Si_q;\bZ)\longrightarrow \bC/(2\pi\sqrt{-1})^2\bZ,\quad  A\mapsto \int_\tA \pi_q^*\Phi_q + (2\pi\sqrt{-1})^2\bZ
\end{equation}
where $\tA$ is any element in $(\pi_{q*}\circ\talpha)^{-1}(A) \subset \pi_1(\tSi_q,\tp_q)$.
Let $\int_A\Phi$ denote the image of $A\in K_1(\Si_q;\bZ)$ under the above map. Then
$\int_A\Phi_q \in \bC/(2\pi\sqrt{-1})^2\bZ$ is independent of choice of $p_q$ and $\tp_q$.

Let $B_\ep$ denotes the open ball in $\bC^p$ with radius $\ep>0$ and center $0$.
There exists $\ep>0$ such that $\Si_q$ is smooth for all $q\in B_\ep$.
For $q\in B_\ep$ there are canonical isomorphisms
$$
H_1(\Si_q;\bZ)\cong H_1(\Si_0;\bZ),\quad K_1(\Si_q;\bZ)\cong K_1(\Si_0;\bZ).
$$
Give $A\in K_1(\Si_0;\bZ)\cong K_1(\Si_q;\bZ)$, $\int_A \Phi_q$ can be viewed as
an element in $\Gamma(B_\ep,\bC)/(2\pi\sqrt{-1})^2\bZ$, where $\Gamma(B_\ep,\bC)$
is the space of holomorphic functions on $B_\ep$
and $(2\pi\sqrt{-1})^2\bZ$ is identified with the set of constant functions from $B_\ep$ to $(2\pi\sqrt{-1})^2\bZ$.

The inclusion $J_q: \Si_q\hookrightarrow \bar{\Si}_q$ induces a surjective map
$$
J_{q*}: H_1(\Si_q; \bC)\cong \bC^{2\fg+\fn-1} \to H_1(\bar{\Si}_q;\bC) \cong \bC^{2\fg},
$$
which restricts to a surjective $\bC$-linear map
\begin{equation}
K_1(\Si_q;\bC):=K_1(\Si_q;\bZ)\otimes_{\bZ}\bC\cong \bC^{2\fg+\fn-3}\to
H_1(\bar{\Si}_q;\bC)\cong\bC^{2\fg}.
\end{equation}
The complex vector spaces $K_1(\Si_q;\bC)$, $H_1(\Si_q;\bC)$,
and $H_1(\bar{\Si}_q; \bC)$ form flat complex vector bundles
$\cK$, $\cH$, $\bar{\cH}$ over $B_\ep$ of rank $2\fg+\fn-3$,
$2\fg+\fn-1$, and $2\fg$, respectively.
Given $A\in K_1(\Si_q;\bC)$ which is a flat section of $\cK$ with respect to the Gauss-Manin connection,
$\int_A\Phi_q$ can be viewed as an element in $\Gamma(B_\ep,\bC)/\bC$, where $\bC$ is identified
with the space of constant functions from $B_\ep$ to $\bC$.

\subsection{Holomorphic 1-forms on $\Si_q$}
For any integers $m,n$,
$$
\varpi_{m,n}:= \Res_{H_f=0}\frac{X^m Y^{n-fm}}{H_f}\cdot \frac{dX}{X}\wedge\frac{dY}{Y}=
\frac{-X^m Y^{n-fm}}{Y\frac{\partial}{\partial Y} H_f(X,Y,q)}\frac{dX}{X}
$$
is a holomorphic 1-form on $\Si_q$.
We view $X$ as a flat coordinate independent of $q$, and use implicit differentiation to obtain
$$
\frac{\partial Y}{\partial q_a}= \frac{-X^{m_a} Y^{n_a-fm_a}}{\frac{\partial H}{\partial Y}}.
$$

We define
$$
\nabla_{\frac{\partial}{\partial q_a}}\Phi :=
\left.\frac{\partial \Phi}{\partial q_a}\right|_{X=\textup{constant}}
=\frac{\partial Y}{\partial q_a}\frac{dX}{Y X}.
$$
Then
$$
\nabla_{\frac{\partial}{\partial q_a}}\Phi=\varpi_{m_a,n_a},\quad
a=1,\ldots,p.
$$

\subsection{Differentials of the first kind on $\bar{\Si}_q$}\label{sec:first}

Recall that a differential of the first kind is a holomorphic 1-form.
By results in \cite{BC94}, the holomorphic 1-form $\varpi_{m,n}$ on $\Si_q$
extends to a holomorphic 1-form on $\bar{\Si}_q$ iff $(m,n)$ is in
the interior of the triangle $\triangle$ with vertices $(r,-s)$, $(0,m)$, $(0,0)$.
Without loss of generality, we may assume that
$(m_a,n_a)$ is in the interior of $\triangle$ for $1\leq a\leq \fg$ and
is on the boundary of $\triangle$ for $\fg+1\leq a\leq p$; note that $\fg$ can be zero.
Let $\fn= p-\fg+3$ be the number of lattice points on the boundary of $\triangle$.
Then $\Si_q$ is a Riemann surface of genus $\fg$ with $\fn$ punctures, and
$\bar{\Si}_q$ is a compact Riemann surface of genus $\fg$.
We have
$$
p=\dim_{\bC}H^2_{\CR}(\cX;\bC),\quad \fg =\dim_{\bC}H^4_{\CR}(\cX;\bC),
$$
and
$$
-\chi(\Si_q) = 2\fg-2 + \fn = 1+p +\fg = |G| = \dim_{\bC}H^*_{\CR}(\cX;\bC).\\
$$
A basis of $H^0(\bar{\Si}_q,\omega_{\bar{\Si}_q})=H^{1,0}(\bar{\Si}_q,\bC)$  is given by
$$
\{ \nabla_{\frac{\partial}{\partial q_a}}\Phi=\varpi_{m_a,n_a}: a=1,\ldots,\fg\}.
$$

Let $B_\ep=\{ (q_1,\ldots, q_p)\in \bC^p: |q|< \ep\}$
be an open ball in $\bC^p$ with radius $\ep>0$ centered at the origin. Let
$\boSi_\ep =\{(X,Y,q)\in \bC^*\times \bC^* \times \bC^p: H_f(X,Y,q)=0\}$
be the family of mirror curves over $B_\ep$, and let
$\pi_\ep: \boSi_\ep\to B_\ep$ be given by $(X,Y,q)\mapsto q$,
so that $\pi_\ep^{-1}(q)=\Si_q$. Then $G^*$ acts on
the total space $\boSi_\ep$ by
\begin{equation}
\alpha\cdot (X,Y,q_a) = (\chi_\alpha(\eta_1) X, \chi_\alpha(\eta_2) Y,  \chi_\alpha(h_a^{-1})q_a),
\end{equation}
where $\alpha\in G^*$, $\chi_\alpha:G\to \bC^*$ is the associated character, $\eta_1,\eta_2\in G$
are defined as in Equation \eqref{eqn:eta}, and $h_a\in G$ corresponds to
$(m_a,n_a,1)\in \Bsi$ for $a=1\ldots, p$.
The $G^*$-action on $\boSi_\ep$ extends to the family of compactified mirror curve
$\bar{\pi}_\ep: \overline{\boSi}_\ep\to B_\ep$, where
$\bar{\pi}_\ep^{-1}(q)=\bar{\Si}_q$. In particular, $G^*$ acts on $\bar{\Si}_0$,
and the induced $G^*$-action on $H_1(\bar{\Si}_0;\bC)$ preserves the
intersection pairing $\cdot$ on $H_1(\bar{\Si}_{0};\bC)$. We choose a symplectic basis
$\bar{A}_1,\bar{B}_1,\ldots, \bar{A}_\fg, \bar{B}_\fg$ of $H_1(\bar{\Si}_0;\bC)$ such that
$$
\bar{A}_a\cdot \bar{A}_b=\bar{B}_a\cdot \bar{B}_b=0,\quad \bar{A}_a\cdot \bar{B}_b =\delta_{ab},
\quad a, b=1,\ldots, \fg,
$$
and
$$
\frac{1}{2\pi\sqrt{-1}}\int_{\bar{A}_a}(\nabla_{\frac{\partial}{\partial q_b}}\Phi)\big|_{q=0}
=\delta_{ab}, \quad a,b=1,\ldots,\fg.
$$
The $G^*$-action on $\bar{A}_a$ is given by
$\alpha\cdot \bar{A}_a = \chi_\alpha(h_a) \bar{A}_a$, where $\alpha\in G^*$
and $1\leq a\leq \fg$. We may choose $\bar{B}_a$ such that
$\alpha\cdot \bar{B}_a =\chi_\alpha(h_a^{-1})\bar{B}_a$ for $\alpha\in G^*$ and $1\leq a\leq\fg$.

The inclusion $I_0: \Si_0\to (\bC^*)^2$ is $G^*$-equivariant, where $\alpha\in G^*$ acts on
$(\bC^*)^2$ by $\alpha\cdot(X,Y)= (\chi_\alpha(\eta_1)X,\chi_\alpha(\eta_2)Y)$. So the $G^*$-actions
on $H_1((\bC^*)^2;\bZ)$ and $H_1((\bC^*)^2;\bC)$ are trivial, and
the $G^*$-action on $H_1(\Si_0;\bC)$ preserves $K_1(\Si_0;\bC)=\mathrm{Ker}(I_*:H_1(\Si_0;\bC)\to H_1((\bC^*)^2;\bZ)$.

We extend $\bar{A}_1,\bar{B}_1,\ldots, \bar{A}_{\fg}, \bar{B}_{\fg}$ to a symplectic
basis of $H_1(\bar{\Si}_q;\bC)$ such that $\bar{A}_i$, $\bar{B}_i$ are flat sections
of $\bar{\cH}$ with respect to the Gauss-Manin connection.
Note that the composition
$$
K_1(\Si_q;\bC) \to  H_1(\Si_q;\bC) \stackrel{J_*}{\longrightarrow} H_1(\bar{\Si}_q;\bC)
$$
is surjective. For $a=1,\ldots, \fg$, we lift $\bar{A}_a, \bar{B}_a\in H_1(\bar{\Si}_q;\bC)$ to
$A_a, B_a \in K_1(\Si_q;\bC)$.
We choose $A_a, B_a\in K_1(\Si_q;\bC)$ such that they are flat sections of $\cK$
with respect to the Gauss-Manin connection, and are eigenvectors
of the $G^*$-action on $K_1(\Si_0;\bC)$ at $q=0$.

\subsection{Differentials of the third kind on $\bar{\Si}_q$} \label{sec:third}

Recall that a differential of the third kind is a meromorphic 1-form with  only simple poles.

Let $E_1, E_2, E_3$ be the edges of the triangle $\triangle$ opposite to the vertices $(r,-s)$,  $(0,m)$, $(0,0)$, respectively.
Let $\fn_i+1$ be the number of lattice points on $E_i$ (including
the end points).
Then $G_i:=\{ h\in G: \chi_i(h)=1\}\cong \bmu_{\fn_i}$; in particular, $\fn_1=m$.
For $i=1,2,3$, we have short exact sequences of abelian groups
$$
1\to G_i =\bmu_{\fn_i}\to G\stackrel{\chi_i}{\to} \bmu_{r_i}\to 1,\quad
1\to \bmu_{r_i}^* \to G^* \to G_i^*\to 1.
$$
The $G^*$-action on $\bar{\Si}_0$ preserves the finite set $\bar{\Si}_0\cap D_i$,
where $D_i\subset S_\triangle$ is the torus invariant divisor
associated to the edge $E_i$. The $G^*$-action on $\bar{\Si}_0\cap D_i$
induces a free and transitive $G_i^*$-action on $\bar{\Si}_0\cap D_i$.
So we may label the point in $\bar{\Si}_0\cap D_i$ (noncanonically) by elements in
$G_i^* =\bmu_{\fn_i}^*$: $\bar{\Si}_0\cap D_i =\{ \bar{p}^i_\ell: \ell\in \bmu_{\fn_i}^*\}$.

If $\fg+1\leq a\leq p$ then $h_a\in G_i\setminus \{1\}\cong \bmu_{\fn_1}\setminus\{1\}$  for some
$i\in \{1,2,3\}$. The meromorphic 1-form
$\nabla_{\frac{\partial}{\partial q_a}}\Phi$ has a simple pole at each point in
$\{ \bar{p}^i_\ell: \ell\in G_i^*\}$, and is holomorphic elsewhere on
the compactified mirror curve $\bar{\Si}_q$, so it is a differential of the third kind on
$\bar{\Si}_q$. We have
$$
\Res_{p\to \bar{p}^i_\ell}\bigl( \nabla_{\frac{\partial}{\partial q_a}}\Phi\bigr)\big|_{q=0}
=\frac{e^{\pi \sqrt{-1} k/\fn_i}}{\fn_i}\chi_\ell(k),
$$
if $h_a=e^{2\pi \sqrt{-1}k/\fn_i} \in \bmu_{\fn_i}$, $1\leq k\leq \fn_i-1$,
and $\chi_\ell:G_i\to \bC^*$ is the character associated to $\ell\in G_i^*$.
Let $C^i_\ell \in H_1(\Si_0;\bZ)$ be the class of a small circle around
$\bar{p}^i_\ell$.  Then
there exists $A_a\in K_1(\Si_0;\bC)$ ($a=\fg+1,\cdots,p$) such that
\begin{itemize}
\item $A_a$ is an eigenvector of the $G^*$-action on $K_1(\Si_0;\bC)$;
\item $A_a$  is a $\bC$-linear combination of $C^i_\ell$, $\ell\in G_i^*$;
\item for $a,b=1,\ldots,p$,
$$
\frac{1}{2\pi\sqrt{-1}}\int_{A_a} (\nabla_{\frac{\partial}{\partial q_b}}\Phi )\big|_{q=0} = \delta_{ab}.
$$
\end{itemize}
We extend $A_{\fg+1},\ldots, A_p$ to flat sections of $K_1(\Si_q;\bC)$.
Then for $a,b\in \{1,\ldots,p\}$,
$$
\frac{1}{2\pi\sqrt{-1}}\int_{A_a} \nabla_{\frac{\partial}{\partial q_b}}\Phi = \delta_{ab} + O(|q|).
$$
\begin{definition}[B-model closed string flat coordinates]\label{flat-coordinates}
For $a=1,\ldots,p$, let $\tau_a(q)$ be the unique function in $\Gamma(B_\ep,\bC)$ such that $\tau_a(0)=0$ and
$$
\frac{1}{2\pi\sqrt{-1}}\int_{A_a}\Phi  = \tau_a(q) + \bC\in \Gamma(B_\ep,\bC)/\bC.
$$
\end{definition}
Then
$$
\tau_a(q)= q_a + O(|q|^2)
$$
and
$$
\frac{1}{2\pi\sqrt{-1}}\int_{A_a}\nabla_{\frac{\partial}{\partial \tau_b}}\Phi=\delta_{ab},\quad
a,b=1,\ldots, p.
$$

The B-model closed string flat coordinates are related to $q_a,\ a=1,\dots,p$ via the following
hypergeometric formulae
\begin{equation}\label{eqn:mirror-map}
\tau_a(q)=q_a\Big(\sum_{\substack{d_b\in\bZ_{\geq 0}\\ \{\sum_{b=1}^p d_b c_i(h_b)\}=0,\ i=1,2,3}}\prod_{i=1}^3\frac{\Gamma(-\{c_i(h_a)\}+1)}{\Gamma(-c_i(h_a)-\sum_{b=1}^pd_b c_i(h_b)+1)}\cdot\frac{1}{\prod_{b=1}^p d_b!} \prod_{b=1}^p q_b^{d_b}\Big).
\end{equation}
Here $\{ x\}$ denotes the fractional part of $x$.
This mirror map is obtained by solving the GKZ-type Picard-Fuchs equations. Iritani explains these GKZ-operators and explicitly writes down the mirror map for general toric orbifolds \cite{Ir09}. The integrals over A-cycles on mirror curves as flat coordinates and their hypergeometric expressions for toric Calabi-Yau 3-folds are discussed in \cite{CKYZ} together with genus $0$ closed Gromov-Witten mirror symmetry.

Let  $D_q^\infty = \bar{\Si}_q\setminus \Si_q = \bigcup_{i=1}^3 \{\bar{p}^i_1,\ldots, \bar{p}^i_{\fn_i}\}$.
By Lefschetz duality, there is a perfect pairing
\begin{equation}\label{eqn:pairing}
H_1(\Si_q;\bC)\times H_1(\bar{\Si}_q,D_q^\infty;\bC)\to \bC.
\end{equation}
The inclusion $\Si_q\subset \bar{\Si}_q$ induces a surjective $\bC$-linear map
$$
H_1(\Si_q;\bC)\cong \bC^{2\fg+\fn-1}\to H_1(\bar{\Si}_q;\bZ)\cong \bC^{2g}.
$$
Our choice of $A_1,B_1,\ldots, A_{\fg}, B_{\fg}$ gives a splitting
\begin{equation}\label{eqn:j-one}
j_1:H_1(\bar{\Si}_q;\bC)\to H_1(\Si_q;\bC).
\end{equation}
The long exact sequence of relative homology groups of the pair $(\bar{\Si}_q, D^\infty_q)$
gives an injective $\bC$-linear map
\begin{equation}\label{eqn:j-two}
j_2: H_1(\bar{\Si}_q;\bC)\to H_1(\bar{\Si}_q,D^\infty_q;\bC).
\end{equation}
Under the inclusion maps $j_1$ and $j_2$,
the perfect pairing \eqref{eqn:pairing} restricts to the intersection
pairing
$$
H_1(\bar{\Si}_q;\bC)\times H_1(\bar{\Si}_q;\bC)\to \bC,
$$
which is perfect by Poincar\'{e} duality. For $\fg+1\leq a\leq p$, choose
$B_a \in  H_1(\bar{\Si}_q,D_q^\infty;\bC)$
such that
\begin{itemize}
\item $(A_1,\ldots, A_p, B_1,\ldots, B_\fg)$
and $(B_1,\ldots,B_p, -A_1,\ldots, -A_{\fg})$
are dual under the pairing \eqref{eqn:pairing};
\item $\alpha\cdot B_a =\chi_\alpha(h_a^{-1}) B_a$ for
$\alpha\in G^*$, $\fg+1\leq a\leq p$.
\end{itemize}
Then for $a=1,\ldots,p$,
$$
(\nabla_{\frac{\partial}{\partial \tau_a}}\Phi)(z)
= \int_{z'\in B_a} B(z',z),
$$
where $B(z',z)$ is the fundamental normalized differential of the second kind on
$\bar{\Si}_q$  (see Section \ref{sec:second}).

\subsection{Critical points and Lefschetz thimbles}\label{sec:thimbles}
We choose framing $f$ such that $X:\Si_0\to \bC^*$ has only simple branch points.
Then $X:\Si_q\to \bC^*$ has only simple branch points for $q$ sufficiently small.

The critical points of
$X:\Si_0 =\{ (X,Y)\in (\bC^*)^2: X^r Y^{-s-rf} + Y^m + 1=0\}\to \bC^*$
are:
$$
\{ (X_{j\ell}, Y_{j\ell}): j\in \{0,1,\ldots, r-1\}, \ell\in \{0,1,\ldots, m-1\} \},
$$
where
\begin{eqnarray*}
X_{j\ell} &=& \exp\Big(2\pi\sqrt{-1}(\frac{j}{r} + \ell\frac{s+rf}{rm})\Big) m^{1/r} (s+rf)^{\frac{s+rf}{rm}}(-m-s-rf)^{\frac{-m-s-rf}{rm} } = (\chi_1^j \chi_2^\ell)(\eta_1)\prod_{i=1}^3(|G|w_i)^{w_i},\\
Y_{j\ell} &=& \exp\Big(2\pi\sqrt{-1}\frac{\ell}{m}\Big) (\frac{s+rf}{-m-s-rf})^{1/m}
= (\chi_1^j\chi_2^\ell)(\eta_2) (\frac{w_2}{w_3})^{1/m}.
\end{eqnarray*}
If $\alpha =\chi_1^j\chi_2^\ell\in G^*$ then we define
$(X_\alpha,Y_\alpha)=(X_{j\ell}, Y_{j\ell})$. The
critical points of $X:\Si_0\to \bC^*$ are
$\{(X_\alpha,Y_\alpha): \alpha\in G^*\}$. We define
\begin{equation}\label{eqn:a-alpha}
a_\alpha := - \log(X_\alpha)= -\sum_{i=1}^3 w_i \log w_i - \log(\chi_\alpha(\eta_1))
\end{equation}
\begin{equation}\label{eqn:b-alpha}
b_\alpha := -\log(Y_\alpha) =  \frac{1}{m}\log(w_3/w_2) - \log(\chi_\alpha(\eta_2)).
\end{equation}
In Equation \eqref{eqn:a-alpha}, we use the identity
$\sum_{i=1}^3 w_i \log(|G|w_i) = \sum_{i=1}^3 w_i \log w_i$ since
$\sum_{i=1}^3 w_i =0$. Note that
$$
\log(\chi_\alpha(\eta_1))=\sqrt{-1}\vartheta_\alpha,\quad
\log(\chi_\alpha(\eta_2))=\sqrt{-1}\varphi_\alpha
$$
for some $\vartheta_\alpha,\varphi_\alpha\in \bR$.
We may assume that $\vartheta_\alpha,\varphi_\alpha\in [0,2\pi)$.

Let $X= e^{-x}$, $Y=e^{-y}$. Around each critical point $p_\alpha$, we set up the following local coordinates
$$
x = a_\alpha(q) +\zeta_\alpha(q)^2, \quad
y = b_\alpha(q) + \sum_{d=1}^\infty h_d^\alpha(q) \zeta_\alpha(q)^d,
$$
where
$$
h_1^\alpha(q) = \sqrt{\frac{2}{\frac{d^2x}{dy^2}(b_\alpha(q))} }.
$$

Let $\gamma_\alpha$ be the Lefschetz thimble of the superpotential
$x=-\log X: \Sigma_q\to \bC$, such that $x(\gamma_\alpha)=a_\alpha+\bR^+$.
Then $\{\gamma_\alpha:\alpha\in G^*\}$ is a basis of the relative homology group
$$
H_1(\Si_q, \{(X,Y)\in \Si_q: \log X \ll 0  \};\bZ) \cong \bZ^{|G|}.
$$

\subsection{Differentials of the second kind on $\bar{\Si}_q$}\label{sec:second}

Recall that a differential of the second kind is a meromorphic 1-form with
no residues.

We choose a symplectic basis $\{A_1, B_1,\ldots, A_{\fg}, B_{\fg}\}$ of
$H_1(\bar{\Si}_q,\bC)$ as in Section \ref{sec:first}.
Let $B(p_1,p_2)$ be the fundamental normalized differential of the second kind on
$\bar{\Si}_q$ (see e.g. \cite{Fay}), which is characterized by:
\begin{enumerate}
\item $B(p_1,p_2)$ is a bilinear symmetric meromorphic differential on $\bar{\Si}_q\times \bar{\Si}_q$.
\item $B(p_1,p_2)$ is holomorphic everywhere except for a double pole along the diagonal $p_1=p_2$.
If $z_1$, $z_2$ are local coordinates on $\bar{\Si}_q\times \bar{\Si}_q$ then
$$
B(z_1,z_2)= (\frac{1}{(z_1-z_2)^2} + f(z_1,z_2) )dz_1 dz_2.
$$
where $f(z_1, z_2)$ is holomorphic.

\item $\displaystyle{\int_{p_1\in A_i} B(p_1,p_2)=0 }$, $i=1\ldots, \fg$.
\end{enumerate}

Let $p_\alpha(q)= (X_\alpha(q),Y_\alpha(q))$ be the branch point
of $X:\bar{\Si}_q\to \bC^*$ such that
$$
\lim_{q\to 0}(X_\alpha(q),Y_\alpha(q))=(X_\alpha, Y_\alpha).
$$
Following \cite{E14, EO15}, given $\alpha\in G^*$ and $d\in \bZ_{\geq 0}$, define
\begin{equation}\label{eqn:theta}
\theta^\alpha_d(p):= -(2d-1)!! 2^{-d}\Res_{p'\to p_\alpha}
B(p,p')\zeta_\alpha^{-2d-1}.
\end{equation}
(In this paper, we use the symbol $\theta^\alpha_d$ instead of the symbol $d\xi_{\alpha,d}$ in \cite{E11, EO15} because the 1-form defined by the right hand side of \eqref{eqn:theta}
is not necessarily exact.)
Then $\theta^\alpha_d$ satisfies the following properties.
\begin{enumerate}
\item $\theta^\alpha_d$ is a meromorphic 1-form on $\bar{\Si}_q$ with
a single pole of order $2d+2$ at $p_\alpha$.
\item In local coordinate $\zeta_\alpha$ near $p_\alpha$,
$$
\theta^\alpha_d = \Big( -\frac{(2d+1)!!}{2^d \zeta_\alpha^{2d+2}}
+ f(\zeta_\alpha)\Big) d\zeta_\alpha,
$$
where $f(\zeta_\alpha)$ is analytic around $p_\alpha$.
The residue of $\theta^\alpha_d$ at $p_\alpha$ is zero,
so $\theta^\alpha_d$ is a differential of the second kind.

\item $\int_{A_i}\theta^\alpha_d=0$ for $i=1,\ldots, \fg$.
\end{enumerate}
The meromorphic 1-form $\theta^\alpha_d$ is characterized by the above
properties; $\theta^\alpha_d$ can be viewed as a section in
$H^0(\bar{\Si}_q, \omega_{\bar{\Si}_q}((2d+2) p_\alpha) )$.

\begin{lemma}\label{mero}
Suppose that $f$ is a meromorphic function on $\bar{\Si}_q$ with simple poles
at the ramification points $\{ p_\beta:\beta \in G^*\}$, and is
holomorphic on $\bar{\Si}_q\setminus\{ p_\beta:\beta\in G^*\}$.
Then
$$
df=\sum_{\beta\in G^*} c_\beta \theta^\beta_0,
$$
where $c_\beta = \displaystyle{ \lim_{p\to p_\beta}f\zeta_\beta }$.
\end{lemma}
\begin{proof}
Let $\displaystyle{\Delta \omega = df-\sum_{\beta\in G^*} c_\beta \theta^\beta_0 }$.
Then $\Delta\omega$ is a holomorphic 1-form on $\bar{\Si}_\alpha$, and
$$
\int_{A_i}\Delta\omega =0, \quad i=1,\ldots, \fg.
$$
So $\Delta\omega =0$.
\end{proof}

\begin{remark}
In Lemma \ref{mero}, different choice of $A$-cycles $A_1,\ldots A_{\fg}$
will give different meromorphic 1-forms $\theta^\beta_0$, but the equality
$$
df =\sum_{\beta\in G^*} c_\beta \theta^\beta_0
$$
still holds, where the coefficients $c_\beta =
\displaystyle{ \lim_{p\to p_\beta}f\zeta_\beta}$ do not
depend on the choice of $A$-cycles.
\end{remark}

We make the following observations:
\begin{enumerate}
\item $dx=-\frac{dX}{X}$ is a meromorphic 1-form on $\bar{\Si}_q$
which is holomorphic on $\Si_q$. It
has a simple zero at each of the $|G|$ ramification points $\{ p_\alpha: \alpha\in G^*\}$
and a simple pole at each of the $\fn$ punctures $\bar{p}_1,\ldots, \bar{p}_{\fn}$.

\item $dy=-\frac{dY}{Y}$ is a meromorphic 1-form on $\bar{\Si}_q$
and a holomorphic 1-form on $\Si_q$. It is nonzero at each of the $|G|$ ramification points, and has at most simple pole at each of the $\fn$ punctures.

\item For $a=1,\ldots, \fg$,
$\displaystyle{ \nabla_{\frac{\partial}{\partial q_a}}\Phi }$
is a  holomorphic 1-forms on $\bar{\Si}_q$ which is nonzero at each of the $|G|$ ramification point.

\item For $a=\fg+1,\ldots, p$, $\displaystyle{\nabla_{\frac{\partial}{\partial q_a}}\Phi}$
is a meromorphic 1-forms on $\bar{\Si}_q$ which is holomorphic on $\Si_q$.
It is nonzero at each of the $|G|$ ramification points, and
has at most simple pole at each of the $\fn$ punctures.
\end{enumerate}

Based on the above observations, the following are
meromorphic function on $\bar{\Si}_q$ satisfying the assumption of Lemma \ref{mero}:
$$
\frac{dy}{dx}, \quad \frac{ \nabla_{\frac{\partial}{\partial q_a}}\Phi }{dx} =\frac{\partial y}{\partial q_a},
$$
where $a=1,\ldots, p$ and $\alpha\in G^*$.
\begin{proposition}\label{Phi-zero-one}
\begin{equation}\label{eqn:Phi-zero}
d(\frac{dy}{dx}) = \frac{1}{2} \sum_{\beta\in G^*} h^\beta_1(q) \theta^\beta_0.
\end{equation}
For $a=1,\ldots, p$,
\begin{equation}\label{eqn:Phi-one}
d\Bigl(\frac{ \nabla_{\frac{\partial}{\partial q_a}}\Phi }{dx}  \Bigr)
=\frac{1}{2} \sum_{\beta \in G^*} \frac{X_\beta(q)^{m_a} Y_\beta(q)^{n_a-fm_a} }
{\frac{\partial H}{\partial x}(X_\beta(q), Y_\beta(q),q)}
h^\beta_1(q) \theta^\beta_0.
\end{equation}
\end{proposition}
Equation \eqref{eqn:Phi-zero} was proved in \cite[Appendix D]{EO15}; we include it
for completeness.
\begin{proof}[Proof of Proposition \ref{Phi-zero-one}]  Near $p_\beta$,
$$
\frac{dy}{dx}\zeta_\beta  = \frac{\sum_{k=1}^\infty h^\beta_k(q) k(\zeta_\beta)^k d\zeta_\beta}{2\zeta_\beta d\zeta_\beta}=
\frac{1}{2}\sum_{k=1}^\infty h^\beta_k(q)  k(\zeta_\beta)^{k-1}.
$$
$$
\frac{\nabla_{\frac{\partial}{\partial q_a}}\Phi}{dx} \zeta_\beta
= \frac{-X^{m_a} Y^{n_a-fm_a}}{\frac{\partial H}{\partial y}(X,Y,q)}\zeta_\beta
=\frac{X^{m_a} Y^{n_a-fm_a}}{\frac{\partial H}{\partial x}(X,Y,q)} \cdot \frac{dy}{dx}\zeta_\beta
$$
So
\begin{eqnarray*}
\lim_{p\to p_\beta} \frac{dy}{dx}\zeta_\beta &=& \frac{h_1^\beta(q)}{2}\\
\lim_{p\to p_\beta}\frac{\nabla_{\frac{\partial}{\partial q_a}}\Phi}{dx}\zeta_\beta &= &
\frac{X_\beta(q)^{m_a} Y_\beta(q)^{n_a-fm_a}}{\frac{\partial H}{\partial x}(X_\beta(q), Y_\beta(q),q)}\frac{h_1^\beta(q)}{2}.
\end{eqnarray*}
The proposition follows from Lemma \ref{mero}.
\end{proof}

\begin{proposition}\label{Phi-One-Two}
\begin{equation}\label{eqn:Phi-One}
d\Bigl(\frac{ \nabla_{\frac{\partial}{\partial \tau_a}}\Phi }{dx}  \Bigr)
=\frac{1}{2} \sum_{\beta \in G^*} \left.\frac{\frac{\partial H}{\partial \tau_a} }
{\frac{\partial H}{\partial x}}\right|_{X=X_\beta(q), Y=Y_\beta(q)}\cdot
h^\beta_1(q) \theta^\beta_0.
\end{equation}
\begin{equation}\label{eqn:Phi-Two}
\nabla_{\frac{\partial}{\partial \tau_a}}\nabla_{\frac{\partial}{\partial \tau_b}}\Phi
=\frac{1}{2} \sum_{\beta \in G^*}
\left.\frac{\frac{\partial H}{\partial \tau_a} } {\frac{\partial H}{\partial x}}
\frac{\frac{\partial H}{\partial \tau_b} } {\frac{\partial H}{\partial x}}
\right|_{X=X_\beta(q), Y=Y_\beta(q)}\cdot
h^\beta_1(q) \theta^\beta_0.
\end{equation}
\end{proposition}
\begin{proof} Equation \eqref{eqn:Phi-One} follows from  Equation \eqref{eqn:Phi-one}. It remains
to prove \eqref{eqn:Phi-Two}.

By special geometry property of the topological recursion (\cite[Theorem 4.4]{EO15},  proved in \cite{EO07}),
$$
(\nabla_{\frac{\partial}{\partial \tau_a}} \nabla_{\frac{\partial}{\partial \tau_b}}\Phi)(p)=\int_{p_1\in B_a}\int_{p_2\in B_b} \omega_{0,3}(p_1,p_2,p),
$$
where (see Example \ref{pants})
$$
\omega_{0,3}(p_1,p_2,p)= \sum_{\beta\in G^*}\frac{-1}{2h^\beta_1} \theta^\beta_0(p_1)\theta^\beta_0(p_2)\theta^\beta_0(p).
$$
So $\displaystyle{ \nabla_{\frac{\partial}{\partial \tau_a}} \nabla_{\frac{\partial}{\partial \tau_b}}\Phi} $ is a linear combination of
$\theta^\beta_0$:
$$
\nabla_{\frac{\partial}{\partial \tau_a}} \nabla_{\frac{\partial}{\partial \tau_b}}\Phi =\sum_{\beta\in G^*} c_\beta \theta^\beta_0,
$$
where the coefficient $c_\beta$ is given by
$c_\beta = \displaystyle{ \Res_{\zeta_\beta\to 0} (\zeta_\beta \cdot \nabla_{\frac{\partial}{\partial \tau_a}} \nabla_{\frac{\partial}{\partial \tau_b}}\Phi) }$.
We have
$$
\nabla_{\frac{\partial}{\partial \tau_a}} \nabla_{\frac{\partial}{\partial \tau_b}} \Phi
=\sum_{c,d} \frac{\partial q_c}{\partial \tau_a}\frac{\partial q_d}{\partial \tau_b}
\nabla_{\frac{\partial}{\partial q_c}} \nabla_{\frac{\partial}{\partial q_d}} \Phi
+\sum_c \frac{\partial^2 q_c}{\partial \tau_a\partial \tau_b}\nabla_{\frac{\partial}{\partial q_c}}\Phi,
$$
where $\nabla_{\frac{\partial}{\partial q_c}}\Phi$ is holomorphic on $\Sigma_q$, and
$$
\nabla_{\frac{\partial}{\partial q_c}} \nabla_{\frac{\partial}{\partial q_d}} \Phi
= \Bigl(-\frac{ \frac{\partial H}{\partial q_a}\frac{\partial H}{\partial q_b}\frac{\partial^2 H}{\partial y^2} }{(\frac{\partial H}{\partial y})^3}
+ \frac{  \frac{\partial^2 H}{\partial q_b\partial y}\frac{\partial H}{\partial q_a}
         +\frac{\partial^2 H}{\partial q_a\partial y}\frac{\partial H}{\partial q_b} }{(\frac{\partial H}{\partial y})^2}\Bigr) dx
$$
So
$$
c_\beta
= \left.\frac{\partial H}{\partial \tau_a}\frac{\partial H}{\partial \tau_b}\right|_{X=X_\beta(q),Y=Y_\beta(q)}
\cdot \Res_{\zeta_\beta\to 0} \Big(\zeta_\beta\frac{-\frac{\partial^2 H}{\partial y^2} }{(\frac{\partial H}{\partial y})^3}dx\Big)
=\frac{1}{2} \left. \frac{ \frac{\partial H}{\partial \tau_a} }{ \frac{\partial H}{\partial x} }
 \frac{ \frac{\partial H}{\partial \tau_b} }{ \frac{\partial H}{\partial x} }\right|_{X=X_\beta(q),Y=Y_\beta(q)}
\cdot h^1_\beta(q).
$$
\end{proof}

Given $\alpha,\beta\in G^*$ and $k,l \in \bZ_{\geq 0}$, define
$B^{\alpha,\beta}_{k,l}$ to be the coefficients of
the expansion of $B(p_1,p_2)$ near $(p_\alpha,p_\beta)\in \bar{\Si}_q\times \bar{\Si}_q$ in
coordinates $\zeta_1= \zeta_\alpha(p_1)$ and $\zeta_2= \zeta_\beta(p_2)$. Then
$$
B(p_1,p_2) =\Big( \frac{\delta_{\alpha,\beta}}{ (\zeta_1-\zeta_2)^2}
+ \sum_{k,l\in \bZ_{\geq 0}} B^{\alpha,\beta}_{k,l} \zeta_1^k \zeta_2^l \Big) d\zeta_1d\zeta_2.
$$
We define
\begin{equation}\label{eqn:BcheckB}
\check{B}^{\alpha,\beta}_{k,l} = \frac{(2k-1)!! (2l-1)!!}{2^{k+l+1}} B^{\alpha,\beta}_{2k,2l}.
\end{equation}

The following lemma is the differential of \cite[Equation (D.4)]{E11}
and holds globally on the compactified mirror curve.
\begin{lemma}\label{xi-recursion}
$$
\theta^\alpha_{k+1} = -d(\frac{\theta^\alpha_k}{dx})- \sum_{\beta\in G^*} \check{B}^{\alpha,\beta}_{k,0}\theta^\beta_0.
$$
\end{lemma}
\begin{proof} We have the following Laurent series expansion of
the meromorphic function $\displaystyle{ \frac{\theta^\alpha_k}{dx}}$ near  $p_\beta$
in the local coordinate $\zeta_\beta$:
$$
\frac{\theta^\alpha_k}{dx} =\delta_{\alpha,\beta} \frac{-(2d+1)!!}{2^{d+1} \zeta_\beta^{2d+3}}
-\frac{\check{B}^{\alpha,\beta}_{k,0} }{\zeta_\beta} + h(\zeta_\beta),
$$
where $h(\zeta_\beta)$ is a power series in  $\zeta_\beta$. Set
$$
\Delta\omega =
\theta^\alpha_{k+1} + d(\frac{\theta^\alpha_k}{dx}) + \sum_{\beta\in G^*} \check{B}^{\alpha,\beta}_{k,0} \theta^\beta_0.
$$
Then $\Delta\omega$ is a holomorphic 1-form on $\bar{\Si}_q$, and
$$
\int_{A_i}\Delta\omega=0,\quad i=1\ldots,\fg.
$$
So $\Delta \omega =0$.
\end{proof}

Given $\alpha\in G^*$ and $k\in\bZ_{>0}$, we define
\begin{equation}\label{eqn:hxi}
\hxi_{\alpha,k}: = (-1)^k (\frac{d}{dx})^{k-1} (\frac{ \theta^\alpha_0}{dx})
= (X\frac{d}{dX})^{k-1}(X\frac{ \theta^\alpha_0 }{dX}),
\end{equation}
which is a meromorphic function on $\bar{\Si}_q$. We define $d\hxi_{\alpha,0}:= \theta^\alpha_0$.
\begin{lemma}\label{xihxi}
$$
\theta^\alpha_k = d\hxi_{\alpha,k}-
\sum_{i=0}^{k-1} \sum_{\beta\in G^*}\check{B}^{\alpha,\beta}_{k-1-i,0}d\hxi_{\beta,i}.
$$
\end{lemma}
\begin{proof} We prove by induction on $k$. When $k=0$, $\theta^\alpha_0=d\hxi_{\alpha,0}$
by definition. Suppose that the lemma holds for $k=d$. By Lemma \ref{xi-recursion},
\begin{eqnarray*}
\theta^\alpha_{d+1} &=&
-d(\frac{ \theta^\alpha_d}{dx})- \sum_{\beta\in G^*} \check{B}^{\alpha,\beta}_{d,0} \theta^\beta_0\\
&=& -d( \frac{d\hxi_{\alpha,d}}{dx} -
\sum_{i=0}^{d-1} \sum_{\beta\in G^*}\check{B}^{\alpha,\beta}_{d-1-i,0}\frac{d\hxi_{\beta,i}}{dx})
- \sum_{\beta\in G^*} \check{B}^{\alpha,\beta}_{d,0}d\hxi_{\beta,0} \\
&=& d\hxi_{\alpha,d+1} - \sum_{i=0}^{d-1} \sum_{\beta\in G^*}\check{B}^{\alpha,\beta}_{d-1-i,0}
d\hxi_{\beta,i+1} - \sum_{\beta\in G^*} \check{B}^{\alpha,\beta}_{d,0}d\hxi_{\beta,0} \\
&=& d\hxi_{\alpha,d+1}-
\sum_{i=0}^d \sum_{\beta\in G^*}\check{B}^{\alpha,\beta}_{d-i,0}d\hxi_{\beta,i}.
\end{eqnarray*}
The second equality follows from  the induction hypothesis. So the lemma holds for $k=d+1$.
\end{proof}

\section{B-model topological string} \label{sec:Bmodel}

\subsection{Laplace transform} \label{sec:laplace}
The Laplace transform of a meromorphic 1-form $\lambda$ along a Lefschetz thimble $\gamma_\alpha$ is given by
$$
\int_{z\in \gamma_\alpha} e^{-ux}\lambda.
$$
\begin{definition}
$$
f^\alpha_\beta (u,q):= \frac{e^{ua_\alpha}}{2\sqrt{\pi u}}\int_{z\in \gamma_\alpha} e^{-ux} \theta^\beta_0.
$$
\end{definition}
It is straightforward to check that
$$
f^\alpha_\beta(u,q)= \delta_{\alpha \beta} +O(u^{-1}).
$$

The Eynard-Orantin invariants $\{ \omega_{g,n}\}$ of the mirror curve
can be expressed as a graph sum involving  $f^\alpha_\beta(u,q)$, $\theta^\beta_d$, and
descendant integrals over moduli spaces of stable curves (\cite{KO}, \cite{E11}, \cite{E14} \cite{DOSS}).
The goal of this subsection and the next two subsections is to relate $f^\alpha_\beta(u,q)$
and $\theta^\beta_d$ to terms in the A-model graph sum. The following is our strategy:
\begin{enumerate}
\item[1.] In this section (Section \ref{sec:laplace}), we relate $f^\alpha_\beta(u,q)$ to the following Laplace transforms:
\begin{equation}\label{eqn:laplace-Phi}
\int_{\gamma_\alpha} e^{-ux} \Phi, \quad
\int_{\gamma_\alpha}e^{-ux} \nabla_{\frac{\partial}{\partial \tau_a}}\Phi,\quad
\int_{\gamma_\alpha}e^{-ux} \nabla_{\frac{\partial}{\partial \tau_a}} \nabla_{\frac{\partial}{\partial \tau_b}}\Phi.
\end{equation}
\item[2.] In Section \ref{sec:oscillatory}, we evaluate the oscillatory integrals
in the Landau-Ginzburg mirror of $\cX=[\bC^3/G]$ for any small $q$.
These oscillatory integrals can be identified with
the Laplace transforms in \eqref{eqn:laplace-Phi} by dimensional reduction.

\item[3.] In Section \ref{sec:inverse}, we expand an antiderivative of $\theta^\beta_0$ near $X=0$ and relate
it to $\txi^\beta_0$ in Section \ref{sec:open-closed}.
\end{enumerate}

By Equation \eqref{eqn:Phi-zero},
$$
\int_{z\in \gamma_\alpha} e^{-ux} \Phi
=  u^{-2}\int_{z\in \gamma_\alpha} e^{-ux} d(\frac{dy}{dx})
= \frac{1}{2} u^{-2}\sum_{\beta\in G^*} h^\beta_1 \int_{z\in \gamma_\alpha} e^{-ux} d\xi_{\beta,0}
= \sqrt{\pi} u^{-3/2} e^{-u a_\alpha} \sum_{\beta\in G^*} h_1^\beta f^\alpha_\beta(u,q).
$$
By Equation \eqref{eqn:Phi-One},
\begin{eqnarray*}
\int_{z\in \gamma_\alpha}e^{-ux} \nabla_{ \frac{\partial}{\partial \tau_a} }\Phi
&=& -u^{-1} \int_{z\in \gamma_\alpha} e^{-ux}
d\bigg(\frac{ \nabla_{ \frac{\partial}{\partial \tau_a} }\Phi }{dx}\bigg)\\
&=& -\frac{1}{2} u^{-1}\sum_{\beta\in G^*}
\left. \frac{\frac{\partial H}{\partial \tau_a}}{ \frac{\partial H}{\partial x}}\right|_{X=X_\beta,Y=Y_\beta} \cdot
h_1^\beta \int_{z\in \gamma_\alpha} e^{-ux} \theta^\beta_0\\
&=& -\sqrt{\pi} u^{-1/2}e^{-u a_\alpha}\sum_{\beta\in G^*}
\left. \frac{ \frac{\partial H}{\partial \tau_a}} { \frac{\partial H}{\partial x} }\right|_{X=X_\beta,Y=Y_\beta} \cdot
h_1^\beta(q) f_\beta^\alpha(u,q).
\end{eqnarray*}
By Equation \eqref{eqn:Phi-Two},
\begin{eqnarray*}
\int_{z\in \gamma_\alpha}e^{-ux} \nabla_{ \frac{\partial}{\partial \tau_b} }
\nabla_{ \frac{\partial}{\partial \tau_a} }\Phi
&=& \frac{1}{2} \sum_{\beta\in G^*}
\left. \frac{\frac{\partial H}{\partial \tau_a}}{ \frac{\partial H}{\partial x}}
\cdot \frac{ \frac{\partial H}{\partial \tau_b} }{\frac{\partial H}{\partial x}} \right|_{X=X_\beta,Y=Y_\beta} \cdot
h_1^\beta \int_{z\in \gamma_\alpha} e^{-ux} \theta^\beta_0\\
&=& \sqrt{\pi} u^{1/2}e^{-u a_\alpha}\sum_{\beta\in G^*}
\left. \frac{\frac{\partial H}{\partial \tau_a}}{ \frac{\partial H}{\partial x}}
\cdot \frac{ \frac{\partial H}{\partial \tau_b} }{\frac{\partial H}{\partial x}} \right|_{X=X_\beta,Y=Y_\beta} \cdot
h_1^\beta(q) f_\beta^\alpha(u,q).
\end{eqnarray*}

In the remainder of this subsection, we consider the limit $q\to 0$. We have
\begin{equation}\label{eqn:h-zero}
\lim_{q\to 0} h_1^\beta(q) = \frac{1}{|G|}\sqrt{\frac{-2}{w_1 w_2 w_3}}
\end{equation}
\begin{equation}\label{eqn:XY-zero}
\lim_{q\to 0} \left.\frac{\frac{\partial H}{\partial \tau_a}}
{ \frac{\partial H}{\partial x}} \right|_{X=X_\beta, Y=Y_\beta}
=\chi_\beta(h_a)\prod_{i=1}^3 w_i^{c_i(h_a)},
\end{equation}
where $h_a\in G$ corresponds to $(m_a,n_a,1)\in \Bsi$,
so that $\age(h_a)=1$.

We introduce some notation. For $h\in G$, let
$$
\nabla_h \Phi = \begin{cases}
\Phi,& h=1\\
\nabla_{ \frac{\partial}{\partial \tau_a}  }\Phi, & h=h_a,\\
\nabla_{ \frac{\partial}{\partial \tau_b} }\nabla_{ \frac{\partial}{\partial \tau_a  }}\Phi, & h=h_a h_b
\textup{ and }\age(h)=2.
\end{cases}
$$
Then
\begin{equation}\label{eqn:omega-f}
\begin{aligned}
\lim_{q\to 0} \int_{z\in \gamma_\alpha} e^{-ux} \nabla_h \Phi
=& \frac{\sqrt{-2\pi} (-1)^{\age(h)} }{|G|} u^{\age(h)-\frac{3}{2}}
\Big(\prod_{i=1}^3 w_i^{c_i(h)-\frac{1}{2}}\Big) \\
&\cdot e^{u(\sqrt{-1}\vartheta_\alpha+\sum_{i=1}^3 w_i \log w_i)}
\sum_{\beta\in G^*} f^\alpha_\beta(u,0)\chi_\beta(h).
\end{aligned}
\end{equation}
We conclude:
\begin{proposition} \label{pro:f-omega}
\begin{eqnarray*}
f^\alpha_\beta(u,0) &=& \frac{1}{\sqrt{-2\pi}}
\exp\Big(-(\sqrt{-1}\vartheta_\alpha +\sum_{i=1}^3 w_i \log w_i)u \Big) \\
&& \cdot \sum_{h\in G}\chi_\beta(h^{-1})
\cdot (-1)^{\age(h)}u^{\frac{3}{2}-\age(h)} \prod_{i=1}^3 w_i^{-c_i(h)+\frac{1}{2}}\cdot\lim_{q\to 0} \int_{\gamma_\alpha} e^{-ux} \nabla_h\Phi.
\end{eqnarray*}
\end{proposition}
We will evaluate
$\displaystyle{\lim_{q\to 0} \int_{\gamma_\alpha} e^{-ux} \nabla_h \Phi }$
in the next subsection.

\subsection{Oscillatory integrals} \label{sec:oscillatory}
The equivariant Landau-Ginzburg mirror of a general toric orbifolds
has been studied by Iritani \cite{Ir09}.

The mirror B-model to $\cX=[\bC^3/G]$ is a Landau-Ginzburg model $W_q:(\bC^*)^3\to \bC$, where
\begin{align*}
&W_q=X_1^rX_2^{-s-rf}X_3+X_2^m X_3 +X_3+\sum_{a=1}^p q_a X_1^{m_a} X_2^{n_a-m_af} X_3.
\end{align*}
Define $H=\frac{W}{X_3}$. Following Iritani \cite{Ir09}, the equivariantly perturbed B-model superpotential $\tW_q$ is
$$
\tW_q=W_q- u \log X_1.
$$
We assume $w_1,w_2,u >0$ and $w_3<0$ as usual. Define
$$
t_1 =X_1^rX_2^{-s-rf}X_3, \quad  t_2 = X_2^m X_3, \quad\hat t_1 =X_1^rX_2^{-s-rf}, \quad  \hat t_2= X_2^m.
$$

We use the notation in  Section \ref{sec:thimbles}.
For each critical point $p_\alpha(q)=(a_\alpha(q),b_\alpha(q))$ on the mirror curve, we have
$$
\Im (a_\alpha(0))= -\pi w_3 -\vartheta_\alpha,\quad
\Im (b_\alpha(0))= \frac{\pi}{m} -\varphi_\alpha,\quad
e^{\sqrt{-1} \vartheta_\alpha}=\chi_\alpha(\eta_1),
\quad e^{\sqrt{-1}\varphi_\alpha}=\chi_\alpha(\eta_2).
$$
where $\Im(z)$ is the imaginary part of $z$. Notice that we have a preferred choice of the labeling of branch points by the elements in $G^*$ with
$\vartheta_1=\varphi_1=0$. 
Let $C= -1+\sqrt{-1}\bR\subset \bC^*$, and define (Lagrangian) cycles
$\Gamma^{\mathrm{red}}_{\alpha,q}\subset (\bC^*)^2$ and
$\Gamma_{\alpha,q}\subset (\bC^*)^3$ by
$$
\Gamma^{\mathrm{red}}_{\alpha,q} :=
(\bR^+ e^{-a_\alpha(q)}) \times (\bR^+ e^{-b_\alpha(q)}),\quad
\Gamma_{\alpha,q}:=\Gamma^{\mathrm{red}}_{\alpha,q}\times C \subset (\bC^*)^3.
$$
In the perturbed superpotential $\tW_q$, the
logarithm is defined in the following way: when $q=0$, $X_3<0$, $\Im(\log X_1)= \pi w_3+
\vartheta_\alpha,\ \Im(\log X_2)= -\frac{\pi}{m}+ \varphi_\alpha,\ \Im
(\log X_3)=\pi.$
Since the cycle $\Gamma_{\alpha,q}$ is contractible and deforms
continuously  with respect to $q$, the choice is fixed
by these conditions on $\Gamma_{\alpha,q}$.

Define the oscillatory integral of $\tW_q$ to be
$$
\tI^\cX_\alpha (u)=\int_{\Gamma_{\alpha,q}} e^{-\tW_q} \frac{dX_1}{X_1} \frac{dX_2}{X_2} \frac{dX_3}{X_3}.
$$
Define  $L_\alpha := \{ (q_1,\ldots,q_p)\in \bC^p\mid q_a X_\alpha^{m_a} Y_\alpha^{n_a-m_a f}\in \bR\}$, which is a totally
real linear subspace of $\bC^p$.
We will view $\tI^\cX_\alpha(u)$ as a function of $q\in L_\alpha\cong \bR^p$ and
consider the power series expansion at $q=0$.
\begin{lemma}\label{Lalpha}
If $q \in L_\alpha$ is sufficiently small then
$\Gamma^{\mathrm{red}}_{\alpha,q}=\Gamma^{\mathrm{red}}_{\alpha,0}$.
\end{lemma}
\begin{proof}
Consider the change of variables  $X=X_\alpha A$, $Y=Y_\alpha B$,
and $q_a= X_\alpha^{-m_a} Y_\alpha^{-n_a+m_a f}\epsilon_a$ for
$a=1,\ldots,p$. Then $q=(q_1,\ldots,q_p)\in L_\alpha$ iff $\epsilon =(\epsilon_1,\ldots,\epsilon_p)\in \bR^p$.
We have
$$
\hat{t}_1= \frac{w_1}{w_3} A^r B^{-s-rf},\quad
\hat{t}_2 = \frac{w_2}{w_3} B^m,\quad
H(X,Y,q)= F(A,B,\epsilon),\quad
Y\frac{\partial H}{\partial Y} =G(A,B,\epsilon),
$$
where
\begin{align*}
F(A,B,\epsilon)& = \frac{w_1}{w_3} A^r B^{-s-rf} + \frac{w_2}{w_3} B^m +1 +\sum_{a=1}^p \epsilon_a A^{m_a} B^{n_a-fm_a},\\
G(A,B,\epsilon)&= \frac{|G|w_1w_2}{w_3}(- A^r B^{-s-rf} + B^m) +\sum_{a=1}^p (n_a-fm_a)\epsilon_a A^{m_a} B^{n_a-fm_a}
\end{align*}
are $\bR$-valued real analytic functions on $\bR^+\times \bR^p$. We have
$F(1,1,0)=G(1,1,0)=0$, and the Jacobian
$|\frac{\partial(F,G)}{\partial(A,B)}(1,1,0)| =-|G|^2 w_1w_2 \neq 0$.
By the implicit function theorem, there exist real analytic functions
$c, d:U\to \bR^+$, where $U$ is an open neighborhood of $0$ in $\bR^p$, such that $c(0)=d(0)=1$ and
$F(c(\epsilon),d(\epsilon),\epsilon)=G(c(\epsilon), d(\epsilon),\epsilon)=0$.
If $\epsilon\in U$  then $e^{-a_\alpha(q)}=e^{-a_\alpha(0)}c(\epsilon)$ and $e^{-b_\alpha(q)} = e^{-b_\alpha(0)}d(\epsilon)$,
so $\Gamma^{\mathrm{red}}_{\alpha,q} =\Gamma^{\mathrm{red}}_{\alpha,0}$.
\end{proof}

By Lemma \ref{Lalpha} and its proof, if $q\in L_\alpha$ and
$(X_1,X_2)\in \Gamma^{\mathrm{red}}_{\alpha,q}$ then
$\hat t_1, \hat t_2\in \bR^-$ and $H\in \bR$. (Recall that $w_1,w_2>0$ and $w_3=-w_1-w_2<0$.)
We first integrate out $\hat t_1,\hat t_2$, and then integrate $X_3\in C$. Recall that $h_1,\dots, h_p$ are
age $1$ elements in $G$. Given $\vh= (r_1,\ldots, r_p)\in \bZ^{\oplus p}$,
where $r_a \geq 0$, define
$c_i(\vh)=\sum_{a=1}^p r_a c_i(h_a)\in \bQ$, and define
$\chi_\alpha(\vh)=\chi_\alpha(\prod_{a=1}^p h_a^{r_a})\in U(1)$.

\begin{align*}
\tI_\alpha^\cX(u)=&\int_{\Gamma_{\alpha,q}} \Big(e^{-\sum_{a=1}^p q_a X_1^{m_a}
  X_2^{n_a-m_af} X_3-t_1-t_2-X_3} (-\hat t_1)^{uw_1} (-\hat t_2)^{uw_2}\cdot \\
&\quad e^{u(w_1+w_2+w_3)(\log X_3-\sqrt{-1}\pi)}
  e^{\sqrt{-1}(\vartheta_\alpha+\pi w_3) u}\frac{d(-\hat t_1) d(-\hat t_2)
  dX_3}{|G|(-\hat t_1) (-\hat t_2) X_3}\\
= &\frac{-1}{|G|}e^{\sqrt{-1}(\vartheta_\alpha + \pi w_3) u}
\sum_{\substack{\vh=(r_1,\ldots,r_p) \\r_a\in \bZ_{\geq 0} }}
e^{\sqrt{-1}\pi c_3(\vh)}\chi_\alpha(\vh) \prod_{a=1}^p\frac{(-q_a)^{r_a}}{r_a!}\\
&\cdot \int_{\Gamma_{\alpha,q}} \Big ( (-\hat t_1)^{ c_1(\vh)+uw_1-1}
  (-\hat t_2)^{c_2(\vh)+uw_2-1}
e^{(\log
  X_3-\sqrt{-1}\pi)(c_1(\vh)+c_2(\vh)+c_3(\vh)+u(w_1+w_2+w_3)-1)}\\
& \cdot e^{-\hat
  t_1X_3-\hat t_2X_3-X_3}\Big)  d(-\hat t_1) d(-\hat t_2) dX_3\\
=&\frac{-1}{|G|}e^{\sqrt{-1}(\vartheta_\alpha + \pi w_3) u}
\sum_{\substack{\vh=(r_1,\ldots,r_p)\\ r_a\in \bZ_{\geq 0} }}
e^{\sqrt{-1}\pi c_3(\vh)}\chi_\alpha(\vh)
  \prod_{a=1}^p\frac{(-q_a)^{r_a}}{r_a!}\\
&\cdot \Big(\int_{X_3\in C }
 \big ( \int_{-\hat t_1>0}  e^{-X_3 \hat t_1} (-\hat t_1)^{c_1(\vh)
  +uw_1-1} d(-\hat t_1) \big) \big( \int_{-\hat t_2>0} e^{-\hat t_2}
  (-\hat t_2)^{c_2(\vh)+uw_2-1}  d(-\hat t_2) \big )\\
& \cdot e^{-X_3} e^{(\log X_3-\sqrt{-1}\pi)(c_1(\vh)+c_2(\vh)+
  c_3(\vh)+u(w_1+w_2+w_3)-1)} dX_3 \Big )\\
=&\frac{-1}{|G|}e^{\sqrt{-1}(\vartheta_\alpha + \pi w_3) u}
\sum_{\substack{\vh=(r_1,\ldots,r_p)\\ r_a\in \bZ_{\geq 0} }}
e^{\sqrt{-1}\pi c_3(\vh)}\chi_\alpha(\vh)
  \prod_{a=1}^p\frac{(-q_a)^{r_a}}{r_a!}\Gamma(uw_1+c_1(\vh)) \Gamma(uw_2+c_2(\vh))\\
&\cdot \Big( \int_{X_3\in C}
  e^{-X_3} e^{(\log X_3-\sqrt{-1}\pi)(c_3(\vh)+uw_3-1)} dX_3 \Big )\\
=&\frac{2\pi\sqrt{-1}}{|G|}e^{\sqrt{-1}(\vartheta_\alpha +\pi w_3) u}
\sum_{\substack{ \vh=(r_1,\ldots,r_p)\\ r_a\in\bZ_{\geq 0} } }
e^{\pi\sqrt{-1}c_3(\vh)}
\chi_\alpha(\vh) \prod_{a=1}^p\frac{(-q_a)^{r_a}}{r_a!} \\
 &\cdot \frac{\Gamma(uw_1+ c_1(\vh)) \Gamma(uw_2+ c_2(\vh))}{\Gamma(-uw_3-c_3(\vh)+1)}.
\end{align*}
Here we use the following identity:
\begin{lemma} \label{hankel}  If $\Re z>0$ then
$\displaystyle{ \frac{\sqrt{-1}}{2\pi} (\int_C e^{-z(\log(s)-\sqrt{-1}\pi)}e^{-s} ds)=\frac{1}{\Gamma(z)} }$.
\end{lemma}
\begin{proof} Hankel's representation of the reciprocal Gamma function says
$$
\frac{\sqrt{-1}}{2\pi}(\int_{C_\delta} e^{-z(\log(t)-\sqrt{-1}\pi)}e^{-t} dt)=\frac{1}{\Gamma(z)}
$$
where $C_\delta$ is a Hankel contour (see Figure 2). The integrand is holomorphic on $\bC-[0,\infty)$, and
\begin{equation}\label{eqn:integrand}
|e^{-z(\log(t)-\sqrt{-1}\pi)}e^{-t}| = e^{-\Re(z)\log|s|+\Im(z)(\arg(t)-\pi)-\Re t} \leq e^{\pi \Im z} |t|^{-\Re z} e^{-(\Re z)t},
\end{equation}
so we may deform $C_\delta$ to the contour $C_{a,b}$ (see Figure 2) without changing
the value of the contour integral. Therefore,
\begin{equation}\label{eqn:Cab}
\frac{\sqrt{-1}}{2\pi}(\int_{C_{a,b}} e^{-z(\log(t)-\sqrt{-1}\pi)}e^{-t} dt)=\frac{1}{\Gamma(z)}
\end{equation}
for any $a,b\in (0,\infty)$. The estimate \eqref{eqn:integrand} implies that contributions to the above contour integral from the two horizontal
rays in $C_{a,b}$ tend to zero as $a,b\to +\infty$. The lemma follows from taking the limit of
\eqref{eqn:Cab} as $a,b\to +\infty$.
\end{proof}
\begin{figure}[h]
\begin{center}
\psfrag{C}{$C_\delta$}
\psfrag{Cab}{$C_{ab}$}
\psfrag{-1+ai}{\footnotesize $-1+a\sqrt{-1}$}
\psfrag{-1-bi}{\footnotesize $-1-b\sqrt{-1}$}
\psfrag{0}{\footnotesize $0$}
\includegraphics[scale=0.4]{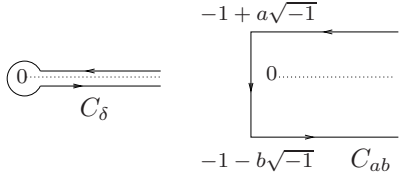}
\end{center}
\caption{A  Hankel contour $C_\delta$ and the contour $C_{a,b}$ ($a,b>0$)}
\end{figure}
\begin{remark}
Let $f(t)$ be the inverse Laplace transform of $F(s)=\displaystyle{ \frac{\Gamma(z)}{s^z} }$, where
$z\in \bC$ is a constant with $\Re z>0$.  Lemma \ref{hankel} implies that $f(t)=t^{z-1}$ for $t>0$.
\end{remark}

By Hori-Iqbal-Vafa \cite{HIV00}, this oscillatory integral could be
reduced to a Laplace transform on the mirror curve. The
Landau-Ginzburg model on $(\bC^*)^3$ is equivalent to a
$5$-dimensional Landau-Ginzburg model, and the $5$d model is again
reduced to a Calabi-Yau threefold without potential. Further dimensional reduction reduces it to the mirror curve.

Introduce two variables $v^+,v^-\in \bC$, and the cycles
$$
\tGamma_{\alpha,q}=\Gamma_{\alpha,q}  \times \{v^+=\overline{v^-}\},\quad
\tGammar=\Gamma^{\mathrm{red}}_{\alpha,q}\times \{v^+=\overline{v^-}\}.
$$
The equation $H(X_1,X_2,q_1,\dots,q_p)=0$ prescribes the mirror curve. Define the holomorphic volume form  $$\Omega=\frac{dX_1}{ X_1}\frac{d X_2}{X_2}\frac{dv^-}{v^-}=dx dy \frac{dv^-}{v^-}.$$

We reduce the oscillatory integral to the mirror curve as
follows.
\begin{align*}
\tI^\cX_\alpha(u)
=&\frac{-1}{2\sqrt{-1}\pi} \int_{\tGamma_{\alpha,q}} e^{-X_3 (H-v^+ v^-)}e^{-ux} \frac{d  X_1}{ X_1}\frac{d  X_2}{X_2} dX_3 dv^+ dv^-\\
=& \int_{\tGammar} \delta(H-v^+ v^-) e^{-ux} \frac{d  X_1}{ X_1}\frac{d  X_2}{X_2} dv^+ dv^-\\
=&-  \int_{\tGammar \cap \{H-v^+ v^-=0\}} e^{-ux} \frac{d  X_1}{ X_1}\frac{d  X_2}{X_2}\frac{dv^-}{v^-}.
\end{align*}
The first identity comes from integrating out $v^+$ and $v^-$ (notice
that on $C$, $\mathrm {Re}(X_3)<0$), using
$$
\int_{v^+=\overline{v^-}}e^{X_3 v^+ v^-}dv^+ dv^-=-\frac{2\pi\sqrt{-1}}{X_3}.
$$
The second identity is the Fourier transform
$$
\int_{C} e^{-X_3(H-v^+v^-)} dX_3=\int_{-\infty}^{+\infty}
e^{(1+s\sqrt{-1} )(H-v^+v^-)}(-\sqrt{-1})ds=-2\pi\sqrt{-1} \delta(H-v^+v^-).
$$
For the third identity, fixing $X_1,X_2$ such that $H\geq 0$, and letting $v^\pm=re^{\pm\sqrt{-1}\theta}$,
$$
\int_{v^+=\overline{v^-}} \delta(H-v^+v^-) dv^+ dv^- = 2\pi \sqrt{-1} \int_{0}^\infty  \delta(H-r^2) 2r dr
=2\pi\sqrt{-1}=-\int_{|v^-|=H} \frac{dv^-}{v^-}.
$$

This integration is further reduced to the mirror curve $H(e^{-x},e^{-y})=0$ as follows.
$$
\tI_{\alpha}^\cX(u)=- \int_{\tGammar\cap \{H-v^+ v^-=0\}} e^{-ux}dx dy  \frac{dv^-}{v^-}\\
= 2\sqrt{-1}\pi\int_{\gamma_\alpha=\tGammar \cap {\{H=v^-=0\}} }  e^{-ux} y dx.
$$
Notice that we use the fact $d(e^{-ux} ydx\frac{dv^-}{v^-})=-e^{-ux} \Omega$ near $\tGammar\cap \{H-v^+ v^-=0\}$. Therefore, we obtain the following formula.
\begin{theorem}
$$
\int_{\gamma_\alpha} e^{-ux}\Phi
= \frac{1}{|G|}e^{\sqrt{-1}(\vartheta_\alpha +\pi w_3) u}
\sum_{\substack{ \vh=(r_1,\ldots,r_p)\\ r_a\in\bZ_{\geq 0} } }
e^{\pi\sqrt{-1}c_3(\vh)}
\chi_\alpha(\vh) \prod_{a=1}^p\frac{(-q_a)^{r_a}}{r_a!}
\cdot \frac{\Gamma(uw_1+ c_1(\vh)) \Gamma(uw_2+ c_2(\vh))}{\Gamma(-uw_3-c_3(\vh)+1)}.
$$
\end{theorem}
\begin{corollary} \label{Phi-q-zero}
$$
\lim_{q\to 0} \int_{\gamma_\alpha} e^{-ux} \nabla_h\Phi
= \frac{e^{\sqrt{-1}(\vartheta_\alpha+\pi w_3) u}}{|G|}
(-1)^{\age(h)}  e^{\sqrt{-1}\pi c_3(h)}\chi_\alpha(h)
\frac{\Gamma(uw_1+c_1(h)) \Gamma(uw_2+c_2(h))}{\Gamma(-uw_3-c_3(h)+1)}.
$$
\end{corollary}

\begin{theorem}\label{thm:f-zero}
$$
f^\alpha_\beta(u,0)=\sum_{h\in G}\frac{\chi_\alpha(h)\chi_\beta(h^{-1}) }{|G|}
\exp\Bigl( \sum_{m\geq 1}\frac{(-1)^{m+1}}{m(m+1)}\sum_{i=1}^3 B_{m+1}(c_i(h))(w_i u)^{-m} \Bigr).
$$
\end{theorem}
\begin{proof} By Proposition \ref{pro:f-omega} and Corollary \ref{Phi-q-zero},
$$
f^\alpha_\beta(u,0) = \sum_{h\in G}\frac{\chi_\alpha(h)\chi_\beta(h^{-1}) }{|G|}
 \cdot C_h\cdot \frac{\Gamma(uw_1+c_1(h)) \Gamma(uw_2+c_2(h))}{\Gamma(-uw_3-c_3(h)+1)},
$$
where
$$
C_h=\frac{1}{\sqrt{-2\pi}} \exp\big(-u \sum_{i=1}^3 w_i \log w_i\big)
\cdot u^{\frac{3}{2}-\age(h)} \prod_{i=1}^3 w_i^{-c_i(h)+\frac{1}{2}}
e^{\sqrt{-1}\pi (w_3 u+c_3(h))}.
$$

By Stirling formula \cite{KP11},
\begin{eqnarray*}
&& \log \frac{\Gamma(w_1 u + c_1(h))\Gamma(w_2 u + c_2(h))}{\Gamma(-w_3 u+1-c_3(h))} \\
&=& (w_1 u +c_1(h)-\frac{1}{2})\log(w_1 u) + (w_2u + c_2(h) -\frac{1}{2})\log(w_2 u) \\
&& -(-w_3 u - c_3(h) +\frac{1}{2})\log(-w_3 u)  + \frac{1}{2}\log(2\pi)   \\
&& + \sum_{m\geq 1}\frac{(-1)^{m+1}}{m(m+1)}\sum_{i=1}^3 B_{m+1}(c_i(h))(w_i u)^{-m} \\
&=& \log\sqrt{-2\pi} -\sqrt{-1}\pi(w_3 u+c_3(h)) + (\age(h)-\frac{3}{2})\log u
+ \sum_{i=1}^3 (c_i(h)-\frac{1}{2})\log w_i \\
&& +  u \sum_{i=1}^3 w_i \log w_i
 +   \sum_{m\geq 1}\frac{(-1)^{m+1}}{m(m+1)}\sum_{i=1}^3 B_{m+1}(c_i(h)) (w_i u)^{-m}
\end{eqnarray*}
\begin{eqnarray*}
&& \frac{\Gamma(w_1 u + c_1(h))\Gamma(w_2 u + c_2(h))}{\Gamma(-w_3 u+1-c_3(h))} \\
&=&\sqrt{-2\pi} e^{-\sqrt{-1}\pi(w_3 u +c_3(h))}u^{\age(h)-\frac{3}{2}} \prod_{i=1}^3 w_i^{c_i(h)-\frac{1}{2}} \exp(u\sum_{i=1}^3 w_i\log w_i) \\
&& \cdot \exp\Bigl(\sum_{m\geq 1}\frac{(-1)^{m+1}}{m(m+1)}\sum_{i=1}^3 B_{m+1}(c_i(h)) (w_i u)^{-m} \Bigr)\\
&=& C_h^{-1}\cdot \exp\Bigl(\sum_{m\geq 1}\frac{(-1)^{m+1}}{m(m+1)}\sum_{i=1}^3 B_{m+1}(c_i(h)) (w_i u)^{-m} \Bigr).
\end{eqnarray*}
Therefore,
$$
f^\alpha_\beta(u,0)=\sum_{h\in G}\frac{\chi_\alpha(h)\chi_\beta(h^{-1}) }{|G|}
\exp\Bigl( \sum_{m\geq 1}\frac{(-1)^{m+1}}{m(m+1)}\sum_{i=1}^3 B_{m+1}(c_i(h))(w_i u)^{-m} \Bigr).
$$
\end{proof}

Define
\begin{equation} \label{eqn:h-check}
\check{h}^\alpha(u,q) : =\frac{u^{3/2}}{\sqrt{\pi}} e^{u a_\alpha}\int_{z\in \gamma_\alpha} e^{-ux}\Phi
= \sum_{\beta\in G^*}f^\alpha_\beta(u,q)h^\beta_1.
\end{equation}

\begin{corollary}\label{dilaton-zero}
$$
\check{h}^\alpha(u,0) = \frac{1}{|G|}\sqrt{\frac{-2}{w_1 w_2 w_3}} \exp\Big(\sum_{m\geq 1}\frac{(-1)^{m+1}}{m(m+1)} \sum_{i=1}^3 B_{m+1}(w_i u)^{-m}\Big).
$$
\end{corollary}
\begin{proof} Recall that
$$
\lim_{q\to 0} h_1^\beta= \frac{1}{|G|}\sqrt{ \frac{-2}{w_1 w_2 w_3} },
$$
for any $\beta\in G^*$, so
$$
\check{h}^\alpha(u,0) = \frac{1}{|G|} \sqrt{\frac{-2}{w_1 w_2 w_3}}\sum_{\beta\in G^*}f^\alpha_\beta(u,0).
$$
By Theorem \ref{thm:f-zero},
$$
\sum_{\beta\in G^*}f^\alpha_\beta(u,0)
= \sum_{h\in G}\frac{\chi_\alpha(h)}{|G|} \Big(\sum_{\beta\in G^*} \chi_\beta(h^{-1})\Big)
\exp\Bigl( \sum_{m\geq 1}\frac{(-1)^{m+1}}{m(m+1)}\sum_{i=1}^3 B_{m+1}(c_i(h))(w_i u)^{-m} \Bigr),
$$
where
$$
\sum_{\beta\in G^*}\chi_\beta(h^{-1}) = \sum_{\beta\in G^*} \chi_\beta(h^{-1})\chi_\beta(1) = |G|\delta_{h^{-1},1}.
$$
So
\begin{eqnarray*}
\check{h}^\alpha(u,0)
&=&\frac{1}{|G|}\sqrt{ \frac{-2}{w_1 w_2 w_3} } \chi_\alpha(1) \exp\Bigl( \sum_{m\geq 1}\frac{(-1)^{m+1}}{m(m+1)}\sum_{i=1}^3 B_{m+1}(c_i(1))(w_i u)^{-m} \Bigr) \\
&=& \frac{1}{|G|}\sqrt{ \frac{-2}{w_1 w_2 w_3} }\exp\Bigl( \sum_{m\geq 1}\frac{(-1)^{m+1}}{m(m+1)}\sum_{i=1}^3 B_{m+1}(w_i u)^{-m} \Bigr).
\end{eqnarray*}
\end{proof}

\subsection{Inverse Laplace transform and expansion at $X=0$} \label{sec:inverse}

Recall that on the Lefschetz thimble $\gamma_\alpha$, the local coordinate is $\zeta_\alpha(q)$, and $x=a_\alpha(q)+\zeta_\alpha(q)^2$. We choose $\zeta_\alpha(q)$ such that when $\zeta_\alpha(0) \to -\infty$, the corresponding point on the mirror curve $\Si_0$ approaches a point with $X=0$ and $Y^m=-1$ (this is consistent with our choice $h^\alpha_1(0)>0$). Furthermore, we require that for any element $\alpha'\in G^*$, its action on the compactified mirror curve $\overline \Si_0$ moves the end point $\zeta_\alpha=-\infty $ on
$\gamma_\alpha$ to the end point $\zeta_{\alpha'\alpha}=-\infty$
on $\gamma_{\alpha'\alpha}$. (Here $\alpha'\alpha \in G^*$
is the product of $\alpha, \alpha'\in G^*$: $\chi_{\alpha'\alpha}(h)=
\chi_{\alpha'}(h)\chi_\alpha(h)$ for all $h\in G$.)

We label the $m$ points with $X=0$ and $Y^m=-1$ on $\overline \Si_0$ by elements
in $\bmu_m^*\cong \bZ_m$ as follows: $\bar{p}_\ell: = (0, e^{\pi\sqrt{-1}\frac{2\ell+1}{m}})$, $\ell\in \bZ_m$.
With our convention, if $\chi_\alpha =\chi_1^j\chi_2^\ell$ then the end point $\zeta_\alpha=-\infty$
is $\bar{p}_{\ell-1}$.

Every 1-form on $\gamma_\alpha$ is exact, so there exists a function $\xi^\alpha_{\beta,0}$
on $\gamma_\alpha$ such that $d\xi^\alpha_{\beta,0}=\theta^\beta_0|_{\gamma_\alpha}$.
The Laplace transform of $\theta^\beta_0$ along $\gamma_\alpha$ is the following
\begin{eqnarray*}
f^\alpha_\beta(u)
&=&  \frac{e^{u a_\alpha}}{2\sqrt{\pi u}} \int_{z \in \gamma_\alpha} e^{-ux(z)}\theta^\beta_0
=\frac{\sqrt{u}}{2\sqrt{\pi}}\int_{z \in \gamma_\alpha} e^{-u(x(z)-a_\alpha)}
\xi^\alpha_{\beta,0} d(x(z)-a_\alpha)\\
&=&  \frac{\sqrt{u}}{2\sqrt{\pi}}\int_{x - a_\alpha\in \bR^+}
e^{-u(x-a_\alpha)}(\xi_{\beta,0}^{\alpha+}(x)-\xi_{\beta,0}^{\alpha-}(x)) d(x-a_\alpha), \\
\end{eqnarray*}
where $\xi_{\beta,0}^{\alpha\pm}(x)=\xi^\alpha_{\beta,0}(x, y(\pm\sqrt{x-a_\alpha}))
\in L^2(0,\infty)$.
From the proof of Theorem \ref{thm:f-zero},
\begin{align*}
f^\alpha_\beta(u,0)=&\frac{e^{-u\sum_{i=1}^3 w_i \log w_i}}{\sqrt{-2\pi}}
\sum_{h\in G}\Big(e^{\sqrt{-1}\pi(w_3 u+ c_3(h))} \cdot \frac{\chi_\alpha(h)\chi_\beta(h^{-1})}{|G|} \\
& \cdot u^{\frac{3}{2}-\age(h)}\cdot \prod_{i=1}^3 w_i^{\frac{1}{2}-c_i(h)}  \frac{\Gamma(w_1 u+ c_1(h))\Gamma(w_2u+c_2(h))}{\Gamma(-w_3u+1-c_3(h))}\Big).
\end{align*}
So the ``classical Laplace transform" is
\begin{align*}
& \int_{x-a_\alpha\in \bR^+}e^{-u(x-a_\alpha)}(\xi^\alpha_{\beta,0}\vert_{q=0}) d(x-a_\alpha)\\
= & -\sqrt{-2}  e^{-u\sum_{i=1}^3 w_i \log w_i} \cdot \sum_{h\in G}
\Big(e^{\sqrt{-1}\pi(w_3 u+c_3(h))}\frac{\chi_\alpha(h)\chi_\beta(h^{-1})}{|G|} \\
& \cdot u^{1-\age(h)}\prod_{i=1}^3 w_i^{\frac{1}{2}-c_i(h)}
\cdot \frac{\Gamma(w_1 u+ c_1(h))\Gamma(w_2u+c_2(h))}{\Gamma(-w_3u+1-c_3(h))}\Big).
\end{align*}

By the inverse Laplace transform formula
\begin{align*}
-\lim_{q\to 0} \xi^{\alpha-}_{\beta,0}=&-\sqrt{-2}
\sum_{\substack{(d_0,h)\in \bZ\times G\\ d_0/r - c_1(h)\in \bZ_{\geq 0}} } \Res_{u=-d_0}(\Gamma(w_1 u+ c_1(h))) \Big( e^{u(x-a_\alpha-\sum_{i=1}^3 w_i \log w_i)} e^{\sqrt{-1}\pi(w_3 u+c_3(h))} \\
&\cdot \frac{\chi_\alpha(h)\chi_\beta(h^{-1})}{|G|}u^{1-\age(h)}\prod_{i=1}^3 w_i^{\frac{1}{2}-c_i(h)} \frac{\Gamma(w_2u+c_2(h))}{\Gamma(-w_3u+1-c_3(h))}\Big)\Big|_{u=-d_0}; \\
\lim_{q\to 0}\xi^{\alpha+}_{\beta,0}=&-\sqrt{-2}
\sum_{\substack{(d_0,h)\in \bZ\times G\\ w_2 d_0 - c_2(h)\in \bZ_{\geq 0}} } \Res_{u=-d_0}(\Gamma(w_2u+c_2(h)))\Big ( e^{u(x-a_\alpha-\sum_{i=1}^3 w_i \log w_i)}  e^{\sqrt{-1}\pi(w_3 u+c_3(h))} \\
&\cdot \frac{\chi_\alpha(h)\chi_\beta(h^{-1})}{|G|}u^{1-\age(h)}\prod_{i=1}^3 w_i^{\frac{1}{2}-c_i(h)} \frac{\Gamma(w_1 u+ c_1(h))}{\Gamma(-w_3u+1-c_3(h))}\Big)\Big|_{u=-d_0}.
\end{align*}
\begin{remark}
We obtain $\xi_{\beta,0}^{\alpha+}$ and $\xi_{\beta,0}^{\alpha-}$ by taking residues around the poles of $\Gamma(w_1 u+c_1(h))$ and $\Gamma(w_2u+c_2(h))$. When both $w_1u+c_1(h)$ and $w_2 u+ c_2(h)$ are non-positive integers, the sum of the respective terms in $\xi^{\alpha+}_{\beta,0}-\xi^{\alpha-}_{\beta,0}$ is the residue.
\end{remark}
We are interested in the expansion of $\xi^\alpha_{\beta,0}$ at $\zeta=-\infty$, i.e. the expansion of $\xi^{\alpha-}_{\beta,0}$ at $X=0$.
\begin{align*}
\lim_{q\to 0}\xi^{\alpha-}_{\beta,0}=& \sqrt{-2}\sum_{\substack{ (d_0,k)\in \bZ_{\geq 0}\times \bZ_m \\ h=h(d_0,k)} }
e^{-d_0 x}\chi_\alpha(\eta_1^{-d_0})e^{-\sqrt{-1}\pi(d_0w_3 - c_3(h))} \\
& \cdot \frac{\chi_\alpha(h)\chi_\beta(h^{-1})}{|G|}(-d_0)^{1-\age(h)}\prod_{i=1}^3 w_i^{\frac{1}{2}-c_i(h)}
\frac{(-1)^{\lfloor \frac{d_0}{r}\rfloor}\Gamma(-d_0 w_2+c_2(h))}{\Gamma(d_0 w_1 -c_1(h)+1)\Gamma(d_0 w_3+1-c_3(h))}\Big)\\
=&\sqrt{-2} \sum_{\substack{ (d_0,k)\in \bZ_{\geq 0}\times \bZ_m \\ h=h(d_0,k)} }
X^{d_0}e^{-\sqrt{-1}\pi(d_0w_3 - c_3(h))}
\frac{\chi_\beta(h^{-1})}{m}(\frac{1}{d_0})^{\age(h)-1}\\
& \cdot \prod_{i=1}^3 w_i^{\frac{1}{2}-c_i(h)}
\frac{\Gamma(d_0(w_1+w_2)+c_3(h))}{\Gamma(d_0 w_1 - c_1(h)+1)\Gamma(d_0 w_2-c_2(h)+1)}\Big)
\chi_\alpha(\eta_2^{-k}).
\end{align*}
This is the expansion of $\lim_{q\to 0}\xi^\alpha_{\beta,0}$ at $X=0$,
$Y=e^{-\sqrt{-1}\pi/m}\chi_\alpha(\eta_2)$.
If $\alpha=\chi_1^j\chi_2^\ell$ then
$$
\chi_\alpha(\eta_2^{-k}) = e^{-2\sqrt{-1}\pi k\ell/m}
$$
depends only on $\ell\in \{0, 1,\ldots,  m-1\}$. We denote
$\xi^{\ell-}_{\beta,0}$ to be the expansion of $\xi^\alpha_{\beta,0}$ at
$\bar{p}_\ell=(0,e^{\pi\sqrt{-1}\frac{2\ell+1}{m}})$ for $\alpha=\chi_1^j\chi_2^{\ell+1}$. Define
$$
\psi_\ell:= \frac{1}{m} \sum_{k=0}^{m-1}\omega_m^{-k\ell} \one'_{\frac{k}{m}}, \quad \ell=0,1, \ldots, m-1,
$$
where $\omega_m= e^{2\pi\sqrt{-1}/m}$. Then $\{\psi_0,\ldots, \psi_{m-1}\}$ is a canonical basis of $H^*_\CR(\cB\bmu_m;\bC)$.
Define
\begin{equation}\label{eqn:xi-beta}
\xi^\beta_0:=\sum_{\ell=0}^{m-1} \xi_{\beta,0}^{\ell-} \psi_\ell,
\end{equation}
which takes values in $H^*_\CR(\cB\bmu_m;\bC)$. Compare with the definition of $\txi^\beta_0$
in Section \ref{sec:open-closed}, we obtain the following identity.
\begin{proposition}\label{xi-txi}
$$
\lim_{q\to 0} \xi^\beta_0=-\frac{1}{|G|}\sqrt{\frac{-2}{w_1w_2w_3}} \txi^\beta_0\big|_{\sv=1}.
$$
\end{proposition}

\subsection{Eynard-Orantin topological recursion}\label{sec:TRR}
Let $\omega_{g,n}$ be defined recursively by the Eynard-Orantin topological recursion \cite{EO07}:
$$
\omega_{0,1}=0,\quad  \omega_{0,2}=B(z_1,z_2).
$$
When $2g-2+n>0$,
\begin{eqnarray*}
\omega_{g,n}(p_1,\ldots, p_n) &=& \sum_{\alpha\in G^*}\Res_{p\to p_\alpha}
\frac{\int_{\xi = p}^{\bar{p}} B(p_n,\xi)}{2(\Phi(p)-\Phi(\bar{p}))}
\Big( \omega_{g-1,n+1}(p,\bar{p},p_1,\ldots, p_{n-1}) \\
&&\quad\quad  + \sum_{g_1+g_2=g}
\sum_{ \substack{ I\cup J=\{1,..., n-1\} \\ I\cap J =\emptyset } } \omega_{g_1,|I|+1} (p,p_I)\omega_{g_2,|J|+1}(\bar{p},p_J),
\end{eqnarray*}
where $p,\bar{p}\to p_\alpha$, $X(p)=X(\bar{p})$ and $p\neq \bar{p}$.

The B-model invariants $\omega_{g,n}$ can be expressed as a sum over labeled graphs involving intersection numbers on  moduli spaces of stable curves \cite{KO, E11, E14, DOSS}. We will use the formula stated in \cite[Theorem 3.7]{DOSS},  which is
equivalent to the formula in \cite[Theorem 5.1]{E11}.
To state this graph sum, we introduce some definitions.
Following Eynard \cite{E11}, define Laplace transform of the Bergman kernel $B(z_1,z_2)$
\begin{equation}
\check{B}^{\alpha,\beta}(u,v,q) :=\frac{uv}{u+v} \delta_{\alpha,\beta}
+ \frac{\sqrt{uv}}{2\pi} e^{ua_\alpha + va_\beta}\int_{z_1\in \gamma_\alpha}\int_{z_2\in \gamma_\beta}
B(z_1,z_2) e^{-ux(z_1) -v x(z_2)},
\end{equation}
where $\alpha,\beta\in G^*$. By \cite[Equation (B.9)]{E11},
\begin{equation}\label{eqn:B-ff}
\check{B}^{\alpha,\beta}(u,v,q) = \frac{uv}{u+v}(\delta_{\alpha,\beta}
-\sum_{\gamma\in G^*} f^\alpha_\gamma(u,q)f^\beta_\gamma(v,q)).
\end{equation}
Let $\check{B}^{\alpha,\beta}_{k,l}(q)$ be defined by \eqref{eqn:BcheckB}. Then
$$
\check{B}^{\alpha,\beta}(u,v,q) = \sum_{k,l} \check{B}^{\alpha,\beta}_{k,l}(q) u^{-k} v^{-l}.
$$
Let $\check{h}^\alpha_k(q)$ be defined by
$$
\check{h}^\alpha(u,q) = \sum_k \check{h}^\alpha_k(q) u^{-k},
$$
where $\check{h}^\alpha$ is defined by \eqref{eqn:h-check}.

Given a labeled graph $\vGa=(\Gamma,g,\alpha,k)\in \bGa_{g,0,n}(\BG)$
with $L^O(\Ga)=\{l_1,\ldots,l_n\}$, we define its weight to be
\begin{eqnarray*}
w(\vGa)&=& (-1)^{g(\vGa)-1}\prod_{v\in V(\Gamma)} \Big(\frac{h^\alpha_1}{\sqrt{-2}}\Big)^{2-2g-\val(v)} \langle \prod_{h\in H(v)} \tau_{k(h)}\rangle_{g(v)}
\prod_{e\in E(\Gamma)} \check{B}^{\alpha(v_1(e)),\alpha(v_2(e))}_{k(e),l(e)}(e) \\
&& \cdot \prod_{j=1}^n \frac{1}{\sqrt{-2}}  \theta^{\alpha(l_j)}_{k(l_j)}(z_j)
\prod_{l\in \cL^1(\Gamma)}(- \frac{1}{\sqrt{-2}} )\check{h}^{\alpha(l)}_{k(l)}(l).
\end{eqnarray*}
In our notation \cite[Theorem 3.7]{DOSS} is equivalent to:
\begin{theorem} \label{thm:DOSS}
For $2g-2+n>0$,
$$
\omega_{g,n} = \sum_{\Gamma \in \bGa_{g,0,n}(\BG)}\frac{w(\vGa)}{|\Aut(\vGa)|}.
$$
\end{theorem}

\begin{example}[pair of pants]\label{pants}
$$
\omega_{0,3}(p_1,p_2,p_3) =  \sum_{\alpha\in G^*} \frac{1}{2h^\alpha_1} \theta^\alpha_0(p_1)\theta^\alpha_0(p_2)\theta^\alpha_0(p_3).
$$
\end{example}

We now consider the unstable case $(g,n)=(0,2)$.
Recall that $dx =-\frac{dX}{X}$ is a meromorphic 1-form on $\bar{\Si}_q$, and
$\frac{d}{dx}=-X\frac{d}{dX}$ is a meromorphic vector field on $\bar{\Si}_q$.
Define
\begin{equation}\label{eqn:C}
C(z_1,z_2):=
(-\frac{\partial}{\partial x}(z_1)-\frac{\partial}{\partial x}(z_2))\Big(\frac{\omega_{0,2}}{dx_1 dx_2}\Big)(z_1,z_2)
d(x(z_1))(dx(z_2)).
\end{equation}
Then $C(z_1,z_2)$ is meromorphic on $(\bar{\Si}_q)^2$ and holomorphic
on $(\bar{\Si}_q \setminus \{p_\alpha:\alpha\in G^*\})^2$. It only has
double poles on $\{p_\alpha:\alpha\in G^*\}\times \bar \Si_q$ and
$\bar \Si_q \times \{p_\beta:\beta\in G^*\}$.
\begin{lemma}\label{C}
$$
C(z_1,z_2)= \frac{1}{2} \sum_{\gamma\in G^*} \theta^\gamma_0(z_1) \theta^\gamma_0(z_2).
$$
\end{lemma}
\begin{proof}
First we show that the Laplace transforms of both sides are
asymptotically equal. For any $\alpha,\beta\in G^*$,
\begin{eqnarray*}
&& \int_{z_1\in \gamma_\alpha}\int_{z_2\in \gamma_\beta}
e^{-u(x(z_1)-a_\alpha)-v(x(z_2)-a_\beta)} C(z_1, z_2)\\
&=& (-u-v)\int_{z_1\in \gamma_\alpha}\int_{z_2\in \gamma_\beta}
e^{-u(x(z_1)-a_\alpha)-v(x(z_2)-a_\beta)}\omega_{0,2}\\
&\sim& 2\pi \sqrt{uv} \sum_{\gamma\in G^*}f^\alpha_\gamma(u,q)f^\beta_\gamma(v,q)\\
&\sim &\frac{1}{2}\int_{z_1\in \gamma_\alpha}\int_{z_2\in \gamma_\beta}
e^{-u(x(z_1)-a_\alpha)-v(x(z_2)-a_\beta)} \theta^\gamma_0(z_1) \theta^\gamma_0(z_2).
\end{eqnarray*}
Define their difference
$$
\omega=C(z_1,z_2)-\frac{1}{2} \sum_{\gamma\in G^*} \theta^\gamma_0(z_1) \theta^\gamma_0(z_2),
$$
then its Laplace transform at $\gamma_\alpha\times \gamma_\beta$ for
any $\alpha,\beta\in G^*$ vanishes. For $i=1,\dots,\fg$,
$$\int_{p_2\in A_i} \omega_{0,2}(p_1,p_2)=0,
\ \int_{A_i}
\theta^\alpha_0=0.$$
Then
$$
\int_{p_2\in A_i}\omega(p_1,p_2) =0,
$$ and the
following residue $1$-form has
$$
\int_{p_2\in A_i} \Res_{p_1\to p_{\alpha}} \zeta_{1,\alpha}(p_1) \omega(p_1,p_2)=0,
$$
for all $i=1,\dots, \fg$. Notice that the $1$-form $$\Res_{p_1\to
  p_\alpha}
\zeta_{1,\alpha}(p_1)\omega(p_1,p_2)=[u^0] \frac{-1}{2\sqrt{\pi u}}\int_{p_1\in
  \gamma_\alpha} e^{-u(x(p_1)-a_\alpha)}
  \omega(p_1,p_2) $$
is a well-defined holomorphic form on $\bar \Si_q$. It has no poles, otherwise
any possible double pole at $p_\beta$ implies non-zero Laplace
transform of $\omega$ at $\gamma_\alpha \times \gamma_{\beta}$. It
follows from the vanishing A-cycles integrals that $$\Res_{p_1\to
  p_\alpha} \zeta_{1,\alpha}(p_1) \omega(p_1,p_2)=0$$ for any
$\alpha\in G^*$, and then $\omega$ does
not have any poles since it is symmetric and could only have double
poles at the ramification points. Therefore by the vanishing A-periods of $\omega$
we know $\omega=0$.
\end{proof}

\subsection{B-model potentials}
Choose $\delta>0$, $\epsilon>0$ sufficiently small, such that
for $|q|<\epsilon$, the meromorphic function
$X: \bar{\Si}_q \to \bC\cup \{\infty\}$
restricts to an isomorphism
$$
X_q^\ell: D_q^\ell\to D_\delta=\{ X\in \bC: |X|<\delta\},
$$
where $D_q^\ell$ is an open neighborhood of $\bar{p}_\ell\in X^{-1}(0)$, $\ell=0,\ldots,m-1$.
Define
$$
\rho_q^{\ell_1,\ldots,\ell_n}:= (X_q^{\ell_1})^{-1}\times \cdots \times (X_q^{\ell_n})^{-1}:
(D_\delta)^n\to D_q^{\ell_1}\times \cdots \times D_q^{\ell_n} \subset (\bar{\Si}_q)^n.
$$

\begin{enumerate}
\item (disk invariants)
At $q=0$, $Y(\bar{p}_\ell)^m=-1$ for $\ell=0,\ldots,m-1$. When
$\epsilon$ and $\delta$ are sufficiently small,
$Y(\rho_q^\ell(X))\in \bC\setminus[0,\infty)$. Choose a branch of logarithm
$\log:\bC\setminus [0,\infty)\to (0,2\pi)$, and define
$$
y_q^\ell (X) = -\log Y (\rho_q^\ell(X)).
$$
The function $y_q^\ell(X)$ depends on the choice of logarithm, but
$y_q^\ell(X)-y_q^\ell(0)$ does not. $dx = -dX/X$ is a meromorphic
1-form on $\bC$ with a simple pole at $X=0$, and
$$
(y_q^\ell(X)-y_q^\ell(0))dx
$$
is a holomorphic 1-form on $D_\delta$.
Define the {\em B-model disk potential} by
$$
\check{F}_{0,1}(q;X):= \sum_{\ell\in \bZ_m}\int_0^X (y_q^\ell(X')-y_q^\ell(0))(-\frac{dX'}{X'}) \cdot \psi_\ell,
$$
which takes values in $H^*(\cB\bmu_m;\bC)$.

\item (annulus invariants)
$$
(\rho_q^{\ell_1,\ell_2})^*\omega_{0,2}-\frac{dX_1dX_2}{(X_1-X_2)^2}
$$
is holomorphic on $D_\delta\times D_\delta$.  Define the {\em B-model annulus potential} by
$$
\check{F}_{0,2}(q;X_1,X_2):=
\sum_{\ell_1,\ell_2\in \bZ_m}
\int_0^{X_1}\int_0^{X_2} \Big((\rho_q^{\ell_1,\ell_2})^*\omega_{0,2}-\frac{dX_1' dX_2'}{(X_1'-X_2')^2}\Big) \cdot
\psi_{\ell_1}\otimes\psi_{\ell_2},
$$
which takes values in $H^*(\cB\bmu_m;\bC)^{\otimes 2}$.
\item For $2g-2+n>0$, $(\rho_{\btau}^{\ell_1,\ldots,\ell_n})^*\omega_{g,n}$ is
holomorphic on $(D_\delta)^n$. Define
$$
\check{F}_{g,n}(q;X_1,\ldots,X_n):=
\sum_{\ell_1,\ldots,\ell_n\in \bZ_m}
\int_0^{X_1}\cdots\int_0^{X_n} (\rho_q^{\ell_1,\ldots,\ell_n})^*\omega_{g,n} \cdot
\psi_{\ell_1}\otimes\cdots \otimes\psi_{\ell_n},
$$
which takes values in $H^*(\cB\bmu_m;\bC)^{\otimes n}$.
\end{enumerate}
For $g\in\bZ_{\geq 0}$ and $n\in \bZ_{>0}$,
$\check{F}_{g,n}(q;X_1,\ldots,X_n)$ is holomorphic on $B_\epsilon \times (D_\delta)^n$ when
$\epsilon, \delta>0$ are sufficiently small. By construction,
the power series expansion of  $\check{F}_{g,n}(q;X_1,\ldots,X_n)$ only involves
positive powers of $X_i$.

For $\beta\in G^*$, let $\xi_{\beta,0}^{\ell-}$ be defined as in Section \ref{sec:inverse}. Then
$$
\xi_{\beta,0}^{\ell-}(X)=\int_0^X(\rho_q^\ell)^*\theta^\beta_0.
$$

\subsection{Special geometry and B-model graph sum}\label{sec:Bgraph}

By the special geometry property of the topological
recursion (\cite[Theorem 4.4]{EO15},  proved in \cite{EO07}):
$$
\nabla_{\frac{\partial}{\partial \tau_a}} \omega_{g,n}(z_1,\cdots,z_n) =
\int_{z_{n+1}\in B_a}\omega_{g,n+1}(z_1,\cdots,z_{n+1}),
$$
where $B_a$ is defined as in Section \ref{sec:third}. At $q=0$,
\begin{eqnarray*}
\int_{z_{n+1}\in B_a }\theta^\alpha_k(z_{n+1})
&=& [u^{-k}]\int_{z_{n+1}\in B_a}(-\frac{\sqrt{u} e^{u a_\alpha}}{\sqrt{\pi}}) \int_{z'\in \gamma_\alpha}B(z_{n+1},z') e^{-ux(z')}\\
&=& [u^{-k}](-\frac{\sqrt{u} e^{u a_\alpha}}{\sqrt{\pi}})
\int_{z'\in \gamma_\alpha}(\nabla_{\frac{\partial}{\partial \tau_a}} \Phi)(z') e^{-ux(z')} \\
&=& [u^{-k}]\frac{1}{|G|} \sqrt{\frac{-2}{w_1 w_2 w_3}}\prod_{i=1}^{3}w_i^{c_i(h_a)}\sum_{\beta \in G^*}f^\alpha_\beta(u)\chi_\beta(h_a).\\
\end{eqnarray*}

Define the B-model primary leaf to be
\begin{equation}\label{eqn:primaryB}
(\check{h}^{\btau})^{\alpha}_{k}= [u^{-k}]\Bigl(\frac{1}{|G|}\sqrt{ \frac{-2}{w_1 w_2 w_3} }\sum_{\beta\in G^*}
\sum_{a=1}^p \prod_{i=1}^{3}w_i^{c_i(h_a)}f_\beta^\alpha(u,0)\chi_\beta(h_a)\tau_a\Bigr).
\end{equation}

For any $\alpha\in G^*$, $k\in\bZ_{\geq 0}$, $\ell\in \bZ_m$, $(\rho^\ell_0)^*\theta^\alpha_k$
is a holomorphic 1-form on $D_\delta$ for small enough $\delta>0$. To each open leaf
$l_j$ of a labeled graph $\vGa\in \Gamma_n(\cX)$, we assign
\begin{equation}\label{eqn:openB}
(\cL^{\xi})^\alpha_k(l_j) =
\frac{1}{\sqrt{-2}}\sum_{\ell\in \bZ_m} \Big(\int_0^{X_j} (\rho_0^\ell)^*\theta^\alpha_k \Big)\psi_\ell.
\end{equation}

The Taylor series expansion of $(\rho_{\btau}^{\ell_1,\ldots,\ell_n})^*\omega_{g,n}$ at $\btau=0$ is
$$
(\rho_{\btau}^{\ell_1,\ldots,\ell_n})^*\omega_{g,n}=\sum_{l_1,\cdots,l_p\in \bZ_{\geq 0}}
(\rho_0^{\ell_1,\ldots,\ell_n})^*\Big((\nabla_{\frac{\partial}{\partial \tau_1}})^{l_1}
\cdots (\nabla_{\frac{\partial}{\partial \tau_p}})^{l_p}
\omega_{g,n}\Big|_{\btau=0}\Big)\cdot \frac{\tau_1^{l_1}}{l_1!}\cdots \frac{\tau_p^{l_p}}{l_p!}.
$$

Recall that $\bGa_{g,n}(\cX)$ is the set of stable genus $g$ graphs with $n$ open leaves
and any number of primary leaves. Given  $\vGa=(\Gamma,g, \alpha,k)\in \bGa_{g,n}(\cX)$
with open leaves $l_1,\ldots,l_n$, we define its B-model weight to be
\begin{equation} \label{eqn:wB}
\begin{aligned}
w_B(\vGa) =&(-1)^{g(\vGa)-1}\prod_{v\in V(\Gamma)}
\Big(\frac{h^\alpha_1|_{\btau=0}}{\sqrt{-2}}\Big)^{2-2g-\val(v)} \langle \prod_{h\in H(v)} \tau_{k(h)}\rangle_{g(v)}
 \prod_{e\in E(\Gamma)} \check{B}^{\alpha(v_1(e)),\alpha(v_2(e))}_{k(e),l(e)}(e)\Big|_{\btau=0}\\
& \cdot\prod_{l\in L^o(\Gamma)}\frac{1}{\sqrt{-2}}(\check{h}^{\btau})^{\alpha(l)}_{k(l)}.
\prod_{j=1}^n (\cL^\xi)^\alpha_k(l_j)
\prod_{l\in L^1(\Gamma)}(-\frac{1}{\sqrt{-2}})\check{h}^{\alpha(l)}_{k(l)}(l)\Big|_{\btau=0}.
\end{aligned}
\end{equation}

We have the following B-model graph sum formulae.
\begin{theorem}[B-model graph sum]\label{thm:Bsum}
\begin{enumerate}
\item $($disk invariants$)$
$$
\check{F}_{0,1}(\btau;X)= \check{F}_{0,1}(0;X)
+\sum_{a=1}^p \tau_a \frac{\partial \check{F}_{0,1}}{\partial \tau_a}(0;X)+ \sum_{\vGa \in \bGa_{0,1}(\cX)}\frac{w_B(\vGa)}{|\Aut(\vGa)|}.
$$
\item $($annulus invariants$)$
$$
\check{F}_{0,2}(\btau;X_1,X_2)=  \check{F}_{0,2}(0;X_1,X_2)+
\sum_{\vGa \in \bGa_{0,2}(\cX)}\frac{w_B(\vGa)}{|\Aut(\vGa)|}.
$$
\item For $2g-2+n>0$,
$$
\check{F}_{g,n}(\btau;X_1,\ldots,X_n) = \sum_{\vGa \in \bGa_{g,n}(\cX)}\frac{w_B(\vGa)}{|\Aut(\vGa)|}.
$$
\end{enumerate}
\end{theorem}

\subsection{All genus open-closed mirror symmetry} \label{sec:final}

\begin{proposition}\label{AB-disk}
$$
\check{F}_{0,1}(0;X)+\sum_{a=1}^p \tau_a \frac{\partial \check{F}_{0,1}}{\partial \tau_a}(0;X)\\
=\Phi_{1,-2}(X)+\sum_{a=1}^p \tau_a \Phi_{h_a,-1}(X).
$$
\end{proposition}
\begin{proof} From the definition,
$$
-X\frac{d}{dX}\check{F}_{0,1}(0;X) = \sum_{\ell\in \bZ_m} (y^\ell_0(X)-y^\ell_0(0))\psi_\ell.
$$
By Proposition \ref{Phi-zero-one}, Equation \eqref{eqn:Phi-one}, Equation \eqref{eqn:h-zero},
and Proposition \ref{xi-txi},
\begin{eqnarray*}
(X\frac{d}{dX})^2 \check{F}_{0,1}(0;X)&=&\sum_{\beta\in G^*} \frac{1}{|G|\sqrt{-2w_1w_2w_3}} \lim_{q\to 0} \xi^\beta_0(X)
= \sum_{\beta\in G^*}\frac{1}{|G|^2 w_1w_2w_3} \txi^\beta_0(X)|_{\sv=1}\\
(X\frac{d}{dX}) \frac{\partial \check{F}_{0,1}}{\partial \tau_a}(0;X) &=&
\frac{1}{|G|\sqrt{-2w_1w_2w_3}}\sum_{\beta\in G^*}\chi_\beta(h_a)\prod_{i=1}^3 w_i^{c_i(h_a)}
\lim_{q\to 0} \xi^\beta_0(X)\\
&=& \frac{1}{|G|^2w_1w_2w_3} \sum_{\beta\in G^*}\chi_\beta(h_a)\prod_{i=1}^3 w_i^{c_i(h_a)}\txi^\beta_0(X)|_{\sv=1}.
\end{eqnarray*}
$\check{F}_{0,1}(0;X)$ and $\frac{\partial \check{F}_{0,1}}{\partial \tau_a}(0;X)$ are power series in $X$ with
no constant term, so
\begin{equation}\label{eqn:Bdisk-zero}
\check{F}_{0,1}(0;X)= \frac{1}{|G|^2 w_1w_2w_3} \txi^\beta_{-2}(X)|_{\sv=1}
\end{equation}
\begin{equation}\label{eqn:Bdisk-one}
\frac{\partial \check{F}_{0,1}}{\partial \tau_a}(0;X)= \frac{1}{|G|^2 w_1w_2w_3}
 \sum_{\beta\in G^*}\chi_\beta(h_a)\prod_{i=1}^3 w_i^{c_i(h_a)} \txi^\beta_{-1}(X)|_{\sv=1}.
\end{equation}
The lemma follows form \eqref{eqn:Bdisk-zero}, \eqref{eqn:Bdisk-one},
and \eqref{eqn:disk-zero}.
\end{proof}

\begin{proposition}\label{AB-annulus}
$$
\check{F}_{0,2}(0;X_1,X_2) =- F_{0,2}^{\cX,(\cL,f)}(0;X_1,X_2).
$$
\end{proposition}
\begin{proof} Let $C= C(z_1,z_2)$ be defined by \eqref{eqn:C}. Then
\begin{eqnarray*}
&& (X_1\frac{\partial}{\partial X_1} + X_2\frac{\partial}{\partial X_2})\check{F}_{0,2}(0;X_1,X_2)\\
&=& \sum_{\ell_1,\ell_2\in \bZ_m} \int_0^{X_1}\int_0^{X_2}(\rho_0^{\ell_1,\ell_2})^*C
\psi_{\ell_1}\otimes\psi_{\ell_2}\\
&=&\frac{1}{2} \sum_{\ell_1,\ell_2\in \bZ_m} \sum_{\gamma\in G^*} \int_0^{X_1}(\rho_0^{\ell_1})^* \theta^\gamma_0
\int_0^{X_2}(\rho_0^{\ell_2})^* \theta^\gamma_0 \psi_{\ell_1}\otimes \psi_{\ell_2}\\
&=&\lim_{q\to 0}\frac{1}{2}\sum_{\gamma\in G^*} \xi^\gamma_0(X_1)\xi^\gamma_0(X_2)\\
&=&- \frac{1}{|G|^2 w_1w_2w_3} \Big(\sum_{\gamma\in G^*}\txi^\gamma_0(X_1)\txi^\gamma_0(X_2)\Big)|_{\sv=1} \\
&=& -(X_1\frac{\partial}{\partial X_1} + X_2\frac{\partial}{\partial X_2})F_{0,2}^{\cX,(\cL,f)}(0;X_1,X_2),
\end{eqnarray*}
where the second equality follows from Lemma \ref{C}, the fourth equality follows from
Proposition \ref{xi-txi}, and the last equality follows from \eqref{eqn:annulus-zero}.
Both $W_{0,2}(0;X_1,X_2)$ and $F^{\cX,(\cL,f)}_{0,2}(0;X_1,X_2)$ are $H^*_{\CR}(\cB\bmu_m;\bC)^{\otimes 2}$-valued
power series in $X_1, X_2$ which vanish at $(X_1,X_2)=(0,0)$, so
$$
\check{F}_{0,2}(0;X_1,X_2)=- F_{0,2}^{\cX,(\cL,f)}(0;X_1,X_2).
$$
\end{proof}

\begin{proposition}\label{AB-stable}
For any $\vGa\in \bGa_{g,n}(\cX)$,
$$
w_B(\vGa)=(-1)^{g(\vGa)-1 +n} w_A(\vGa),
$$
where $w_A(\vGa)$ is defined in Section \ref{sec:Agraph}, and
$w_B(\vGa)$ is defined in Section \ref{sec:Bgraph}.
\end{proposition}

\begin{proof} We fix $\vGa\in \bGa_{g,n}(\cX)$.
From \eqref{eqn:Rcan} and Theorem \ref{thm:f-zero}, for any $\alpha,\beta\in G^*$,
\begin{equation}\label{eqn:Rf}
R(-z)^{\beta}_{\alpha}=f^{\alpha}_{\beta}(\frac{1}{z},0).
\end{equation}
Here we have used the fact that $B_m(1-x)=(-1)^mB_m(x)$, $B_m=0$ when $m$ is odd and $c_i(h)+c_i(h^{-1})=1-\delta_{c_i(h),0}$.

\noindent
1. {\it Vertex.} By Equation \eqref{eqn:h-zero},
\begin{equation}\label{eqn:ABvertex}
\frac{\lim_{\btau\to 0}h^\alpha_1}{\sqrt{-2}}=\frac{1}{|G|\sqrt{w_1w_2w_3}}.
\end{equation}

\noindent
2. {\it  Edge.}  By Equation \eqref{eqn:B-ff} (i.e. \cite[Equation (B.9)]{E11}),
\begin{eqnarray*}
\check{B}^{\alpha,\beta}_{k,l}\Big|_{\btau=0}&=&[u^{-k}v^{-l}]\left(\frac{uv}{u+v}(\delta_{\alpha,\beta}
-\sum_{\gamma\in G^*} f^\alpha_\gamma(u,0)f^\beta_\gamma(v,0))\right)\\
&=&[z^{k}w^{l}]\left(\frac{1}{z+w}(\delta_{\alpha,\beta}
-\sum_{\gamma\in G^*} f^\alpha_\gamma(\frac{1}{z},0)f^\beta_\gamma(\frac{1}{w},0))\right).
\end{eqnarray*}
By definition,
$$
\cE^{\alpha,\beta}_{k,l}
=[z^{k}w^{l}]\left(\frac{1}{z+w}(\delta_{\alpha,\beta}
-\sum_{\gamma\in G^*} R(-z)^{\gamma}_{\alpha}R(-w)^{\gamma}_{\beta})\right).
$$
It follows from \eqref{eqn:Rf} that
\begin{equation}\label{eqn:ABedge}
\check{B}^{\alpha,\beta}_{k,l}\Big|_{\btau=0} = \cE^{\alpha,\beta}_{k,l}.
\end{equation}

\noindent
3. {\it Primary leaf.}
By Equation \eqref{eqn:primaryB},
$$
\frac{1}{\sqrt{-2}}(\check{h}^{\btau})^{\alpha}_{k}=[u^{-k}]\Bigl(\frac{1}{|G|\sqrt{w_1w_2w_3}} \sum_{\beta\in G^*}
\sum_{a=1}^p \prod_{i=1}^{3}w_i^{c_i(h_a)}f_\beta^\alpha(u)\chi_\beta(h_a)\tau_a\Bigr).
$$
By Equation \eqref{eqn:primaryA},
$$
(\cL^\tau)^\alpha_k =[z^{k}]\Bigl(\frac{1}{|G|\sqrt{w_1w_2w_3}} \sum_{\beta\in G^*}
\sum_{a=1}^p\prod_{i=1}^{3}w_i^{c_i(h_a)} R(-z)_\alpha^\beta\chi_\beta(h_a)\tau_a\Bigr).
$$
It follows from \eqref{eqn:Rf} that
\begin{equation}\label{eqn:ABprimary}
\frac{1}{\sqrt{-2}}(\check{h}^{\btau})^\alpha_k =(\cL^{\btau})^\alpha_k.
\end{equation}

\bigskip

\noindent
4. {\em Open leaf.}
Given $\beta\in G^*$ and $k\in\bZ_{\geq 0}$, define
\begin{equation}\label{eqn:hxi-X}
\hxi^\beta_k(X) :=\sum_{\ell=0}^{m-1} \int_0^X (\rho_0^\ell)^*(d\hxi_{\beta,k}) \psi_\ell,
\end{equation}
which is an $H^*_{\CR}(\cB\bmu_m;\bC)$-valued holomorphic function on $D_\delta$.  Then
$$
\hxi^\beta_k(X)= (X\frac{d}{dX})^k \hxi^\beta_0(X),
$$
and
$$
\hxi^\beta_0(X) = \lim_{q\to 0} \xi^\beta_0,
$$
where $\xi^\beta_0$ is defined by Equation \eqref{eqn:xi-beta}. By Proposition \ref{xi-txi}
and definitions of $\hxi^\beta_k(X)$, $\txi^\beta_k(X)$,
\begin{equation}\label{eqn:hxi-txi}
\hxi^\beta_k(X)=-\frac{1}{|G|}\sqrt{\frac{-2}{w_1w_2w_3}} \txi^\beta_k(X).
\end{equation}
By \eqref{eqn:openB}, \eqref{eqn:hxi-X}, \eqref{eqn:hxi-txi}, and Lemma \ref{xihxi},
$$
(\cL^\xi)^\alpha_k(l_j)=\frac{-1}{|G| \sqrt{w_1w_2w_3}}\Big(\txi^\alpha_k -
 \sum_{i=0}^{k-1} \sum_{\beta\in G^*}\big(\check{B}^{\alpha,\beta}_{k-1-i,0}\big|_{\btau=0}\big)\txi^\beta_i(X_j)\Big).
$$
By item 2 ({\em Edge}) above, for $k\in \bZ_{\geq 0}$,
$$
\check{B}^{\alpha,\beta}_{k,0}\Big|_{\btau=0}=[z^{k}w^{0}]\left(\frac{1}{z+w}(\delta_{\alpha,\beta}
-\sum_{\gamma\in G^*} R(-z)^{\gamma}_{\alpha}R(-w)^{\gamma}_{\alpha})\right)
=[z^{k+1}](-R(-z)^{\beta}_{\alpha}).
$$
We also have
$$
[z^0](R(-z)^{\beta}_\alpha)=\delta_{\alpha,\beta}.
$$

Therefore,
$$
(\cL^{\xi})^\alpha_k(l_j)=[z^k]\Big(\frac{-1}{|G|\sqrt{w_1w_2w_3}} \sum_{\beta\in G^*}R(-z)^\beta_\alpha
\txi^\beta(z,X_j)\Big).
$$
Comparing the above expression with the definition of the A-model open leaf $(\cL^{\txi})^\alpha_k(l_j)$,
we conclude:
\begin{equation}\label{eqn:ABopen}
(\cL^\xi)^\alpha_k(l_j)=-(\cL^{\txi})^\alpha_k(l_j).
\end{equation}
\noindent
5. {\it Dilaton leaf.}
$$
(-\frac{1}{\sqrt{-2}})\check{h}^{\alpha}_{k}\Big|_{\btau=0}=
[u^{1-k}](-\frac{1}{\sqrt{-2}})\check{h}^{\alpha}(u)\Big|_{\btau=0}
=[u^{1-k}]\frac{-1}{|G|\sqrt{w_1 w_2 w_3}}\sum_{\beta\in G^*}f^\alpha_\beta(u,0),
$$
where the second identity follows from Corollary \ref{dilaton-zero}. By definition,
$$
(\cL^1)^\alpha_k =[z^{k-1}]\frac{-1}{|G|\sqrt{w_1w_2w_3}}\sum_{\beta\in G^*}R^\beta_\alpha(-z).
$$
It follows from \eqref{eqn:Rf} that
\begin{equation}\label{eqn:ABdilaton}
(-\frac{1}{\sqrt{-2}})\check{h}^{\alpha}_{k}\Big|_{\btau=0} =(\cL^1)^\alpha_k.
\end{equation}

\bigskip
By \eqref{eqn:wA}, \eqref{eqn:wB}, \eqref{eqn:ABvertex},
\eqref{eqn:ABedge}, \eqref{eqn:ABprimary}, \eqref{eqn:ABopen}, and \eqref{eqn:ABdilaton},
$$
w_B(\vGa)=(-1)^{g(\vGa)-1 +n}  w_A(\vGa).
$$
\end{proof}

Combining Theorem \ref{thm:Asum} (A-model graph sum), Theorem \ref{thm:Bsum} (B-model graph sum),
Proposition \ref{AB-disk}, Proposition \ref{AB-annulus}, and Proposition \ref{AB-stable}, we obtain:

\begin{theorem}[All genus open-closed mirror symmetry]\label{thm:main}
$$
\check{F}_{g,n}(\btau; X_1,\ldots, X_n) =(-1)^{g-1+n} F_{g,n}^{\cX, (\cL,f)}(\btau;X_1,\ldots, X_n).
$$
\end{theorem}
As a consequence, the $H^*(\cB\bmu_m)^{\otimes n}$-valued formal power series
$F_{g,n}^{\cX,(\cL,f)}(\btau;X_1,\ldots, X_n)$
converges in an open neighborhood of the origin in $\bC^p\times
\bC^n$.

\end{document}